\documentclass[11pt]{amsart}
\usepackage{amssymb,amsthm}
\usepackage{amsmath}
\usepackage{amsfonts}
\usepackage{enumerate}
\usepackage{dsfont}
\usepackage{cite}
\usepackage[colorlinks, citecolor=red]{hyperref}
\usepackage{mathrsfs}
\usepackage{epsfig}
\usepackage{lscape}
\usepackage{subfigure}
\usepackage{booktabs}
\usepackage{epstopdf} 
\usepackage{caption}
\usepackage{algorithm}
\usepackage{algpseudocode}
\usepackage{multirow}
\usepackage{geometry}
\usepackage{paralist}
\usepackage{enumerate}
\usepackage{url}
\usepackage{graphicx,cite}
\usepackage{longtable}
\usepackage{tikz}
\usepackage{adjustbox}
\usepackage{CJKutf8}
\usepackage{hyperref}
\usepackage{url}

\newtheorem{definition}{Definition}[section]

\newtheorem{prop}[definition]{Proposition}
\newtheorem{thm}[definition]{Theorem}
\newtheorem{lemma}[definition]{Lemma}

\newtheorem{remark}[definition]{Remark}

\newtheorem{assumption}{Assumption}[section]

\date{}

\begin{document} 
		\baselineskip 18pt
		\bibliographystyle{plain}

\title
[Randomized CGLS]{Randomized conjugate gradient least squares}
	
		\author{Yun Zeng}
	\address{School of Mathematical Sciences, Beihang University, Beijing, 100191, China. }
	\email{zengyun@buaa.edu.cn}
	
	\author{Jian-Feng Cai}
	\address{Department of Mathematics,	Hong Kong University of Science and Technology, Clear Water Bay, Kowloon, Hong Kong SAR, China. }
	\email{jfcai@ust.hk}	
	
	\author{Deren Han}
	\address{LMIB of the Ministry of Education, School of Mathematical Sciences, Beihang University, Beijing, 100191, China. }
	\email{handr@buaa.edu.cn}

	\author{Jiaxin Xie}
	\address{LMIB of the Ministry of Education, School of Mathematical Sciences, Beihang University, Beijing, 100191, China. }
	\email{xiejx@buaa.edu.cn}
	
		\begin{abstract}
			We develop a novel randomized conjugate gradient least squares (RCGLS) method for solving least-squares problems, in which iterative sketching is employed at each step to reduce the dimension and hence the computational cost. In particular, we propose a new perspective on the classical CGLS method, where the next descent direction is determined via a constraint correction problem associated with the gradient. Based on this insight, we replace the gradient with a randomized coordinate gradient that naturally satisfies the variance reduction property, leading directly to the proposed RCGLS method. We prove that RCGLS converges linearly in expectation, with a better convergence bound compared to  the randomized coordinate descent  method. Furthermore, we investigate an implementation of the method that avoids full-dimensional vector operations, which are the major bottleneck of vanilla RCGLS for sparse matrices and render it impractical. We also show how to apply the RCGLS method to solve the ridge regression problem, yielding a lightweight, parallelizable, and accelerated method for such problems. Numerical experiments are provided to confirm our results.
			\end{abstract}

	\maketitle	
	
	\let\thefootnote\relax
\let\thefootnote\relax		
		
		\section{Introduction}

	Solving the linear least-squares   problem
	\begin{equation} \label{LS_pro}
		\min_{x \in \mathbb{R}^{d}} \  \frac{1}{2} \|Ax - b\|_{2}^{2}, \ \  A \in \mathbb{R}^{n \times d}, \ b \in \mathbb{R}^{n},
	\end{equation}
	is a cornerstone of computational science and engineering \cite{bjorck2024numerical,scott2025sparse}. In large-scale and sparse settings, iterative methods based on the conjugate gradient (CG) method, which are implicitly applied to the normal equations \(A^\top A x = A^\top b\), are often preferred. 
	The CG for least squares (CGLS) method \cite{hestenes1952methods}, also referred to as CGNR \cite[Section 11.3.9]{golub2013matrix}, is derived via a simple algebraic rearrangement of the standard CG method. A key advantage of CGLS is that it can avoid the explicit formation of the normal matrix, making it widely applicable in practice. In particular, CGLS recursively updates the original residual \(r = b - Ax\), rather than the normal-equation residual \(A^\top r\), and minimizes the quantity \(\|r\|_2\) at each iteration. We note that CGLS is mathematically equivalent to the Lanczos bidiagonalization-based LSQR method \cite{paige1982lsqr} when $A$ has full column rank.
	
	The age of big data has sparked growing interest in solving large-scale computational problems. Such problems arise extensively in data-driven applications, including machine learning, data science, and scientific computing. The coefficient matrix $A$ involved in these tasks is often high-dimensional, making it essential to divide the original computational task into smaller, more manageable subproblems. Hence, traditional iterative methods built upon full matrix-vector products are no longer fully applicable and require further improvement.
	Recently, randomized strategies that sample partial information from $A$ to update iterates have become increasingly popular \cite{han2024randomized,zeng2026stochastic,zeng2024adaptive,zeng2025adaptive,Str09,han2026pseudoinverse,loizou2020momentum,Gow15,leventhal2010randomized}. In addition to their theoretical benefits, numerous numerical experiments have demonstrated the effectiveness of randomized methods \cite{Str09,Gow15,han2024randomized,zeng2025adaptive}.
	Given this situation, we intend to introduce randomization techniques into the CGLS method to efficiently solve large-scale least-squares  problems.

		\subsection{Our contributions}
		In this paper, we present a generic RCGLS algorithmic framework for solving the least-squares  problem \eqref{LS_pro}. The main contributions of this work are as follows.
		\begin{enumerate}
			\renewcommand{\labelenumi}{\arabic{enumi}.}
			\item We propose a novel and flexible perspective for deriving the classical CGLS method, in which the subsequent descent direction is determined by solving a constraint correction problem associated with the gradient; see Section \ref{section-per-CGLS}. To the best of our knowledge, this is the first derivation from such a viewpoint. Building on this insight, we employ randomized sketching matrices $S_k\in\mathbb{R}^{d\times q}$ to extract gradient information at each step, where $S_k$ is drawn from a user-defined distribution $\mathcal{D}$, leading directly to the proposed RCGLS method. Notably, the sketched gradient is designed as a generalization of the randomized coordinate gradient, so that it naturally inherits the variance reduction property; see Remarks \ref{remark-VR-RCD2} and \ref{remark-VR}.
			\item We prove that RCGLS converges linearly in expectation, with a convergence factor dependent on both the choice of the distribution $\mathcal{D}$ and the coefficient matrix $A$; see Theorem \ref{RCGLS_rate}. In addition, its convergence upper bound can be tighter than that of the randomized coordinate descent (RCD) method \cite{leventhal2010randomized}. To facilitate efficient implementation, we first establish a representation adopting a rescaled search direction, which recovers the classic CG modification applied to the normal equations when $\mathcal{D}$ is a fixed distribution with $S_k=I$ \cite[Section 9]{hestenes1952methods}; see Remark \ref{modification}. Based on this, we further provide an equivalent formulation inspired by the variable transformation techniques in \cite{lee2013efficient, fercoq2015accelerated, zeng2026stochastic}, which largely avoids full-dimensional vector operations when the coefficient matrix is sparse.
			\item As a practical application, we extend the RCGLS method to solve ridge regression, where an $\ell_2$-regularization term has been incorporated into the standard least-squares problem. By exploiting the block-orthogonal structure of the related augmented linear system, we reformulate ridge regression into two alternative least squares problems of different dimensions. Applying the RCGLS method to these reformulations yields a lightweight, parallelizable, and accelerated solver for such problems.  Numerical experiments are provided to validate the theoretical findings and demonstrate the efficiency of the proposed method.
		\end{enumerate}
		
\subsection{Related Work}

\subsubsection{Iterative methods for least squares}
\label{section-1-2-1}

For solving large  linear systems or least-squares problems, iterative methods like CGLS \cite{hestenes1952methods}, LSQR \cite{paige1982lsqr}, and LSMR \cite{fong2011lsmr} are often preferred over direct methods, as they typically require much less storage. However, these methods rely on full matrix–vector multiplications, which become infeasible when the matrix $A$ is extremely large and cannot be fully stored in memory. To address this limitation, a class of iterative methods that only access partial information of $A$ at each iteration has been developed.  

If the linear system $Ax=b$ is consistent, the Kaczmarz method, also known as the algebraic reconstruction technique (ART) \cite{herman1993algebraic,gordon1970algebraic}, is a classic and efficient row-action iterative solver. The method alternates between selecting a row of $A$ and updating the current solution via projection onto the hyperplane defined by the chosen row.  A significant advance is the randomized Kaczmarz (RK) method proposed by Strohmer and Vershynin \cite{Str09}, which achieves linear convergence in expectation under row-norm-proportional sampling. Subsequently, there is a large amount of work on the development of the Kaczmarz-type methods, including block Kaczmarz methods \cite{Nec19,needell2014paved,xie2025randomized,Gow15}, accelerated RK methods \cite{sun2025connecting,liu2016accelerated,han2026pseudoinverse,loizou2020momentum,zeng2024adaptive,wang2026linear,rieger2023generalized,su2024greedy}, randomized Douglas-Rachford methods \cite{han2024randomized}, etc. 
 
However, for inconsistent systems, Needell \cite{needell2010randomized} proved that RK-type methods only converge within a bounded radius (i.e., the convergence horizon) around the least-squares solution; see also \cite{bai2021greedy,ma2015convergence,zeng2025adaptive} for further discussions. To address this issue, Zouzias and Freris \cite{Zou12} modified the standard RK method and proposed the randomized extended Kaczmarz (REK) method. A large body of work has further advanced REK-type methods, including block and deterministic variants \cite{Du20Ran,popa1998extensions,popa1999characterization,bai2019partially,zeng2025adaptive}, the greedy randomized augmented Kaczmarz  method \cite{bai2021greedy}, and the randomized extended Gauss-Seidel (REGS) method \cite{Du19,ma2015convergence}, among others.
We note that although REK-type methods converge to the unique minimum Euclidean norm least-squares solution $A^\dagger b$, they require accessing both row and column information of $A$ at each iteration. 

Another class of randomized methods for least-squares problems only requires column information of $A$ per iteration, namely the randomized Gauss-Seidel (RGS) method, also referred to as the  RCD method \cite{leventhal2010randomized}. The RCD method converges linearly in expectation to the least-squares solution and, in theory,  it can be employed to solve least-squares problems of any rank, either overdetermined or underdetermined. In this paper, we combine these advantages of the RCD method with our newly derived perspective on the classical CGLS method to develop the novel RCGLS method. Similar to RCD, the proposed RCGLS method converges linearly in expectation while further achieving a tighter convergence upper bound compared with the standard RCD method.

\subsubsection{Acceleration methods without full-dimensional operations}

Although Nesterov momentum \cite{nesterov1983method,nesterov2003introductory} can be theoretically adopted to accelerate the RCD method, Nesterov \cite{nesterov2012efficiency} noted that the resulting accelerated scheme may suffer from high per-iteration complexity. Specifically, the update of auxiliary variables (e.g., $y^k$) requires full-dimensional vector operations, which severely limits the practical efficiency of the acceleration method. This bottleneck also exists for heavy ball momentum (HBM) \cite{polyak1964some} when accelerating randomized iterative methods \cite{loizou2020momentum,zeng2026stochastic}.

For the accelerated RCD method, Lee and Sidford \cite{lee2013efficient} first eliminated full-dimensional operations for unconstrained convex quadratic minimization by carefully modifying the standard Nesterov acceleration scheme. Fercoq and Richt{\'a}rik \cite{fercoq2015accelerated} further extended this framework to the accelerated randomized proximal coordinate gradient method under more general settings, where the objective function satisfies certain general structural conditions. Additional efforts have further exploited the inherent structure of regularized empirical risk minimization problems to entirely avoid full-dimensional vector computations \cite{lin2015accelerated}. Furthermore, the accelerated random sketch descent algorithm developed in \cite{necoara2021randomized} enables efficient acceleration without full-dimensional operations, provided that the sketching matrix $S$ is sparse and the gradient evaluation of the objective function $\nabla f(\alpha v + \beta u)$ can be efficiently computed for all $\alpha, \beta \in \mathbb{R}$ and $v, u \in \mathbb{R}^n$.

For the HBM method, the stochastic momentum technique was first proposed in \cite{loizou2020momentum} to avoid full-dimensional vector operations. 
Recently, \cite{zeng2026stochastic} developed an efficient implementation framework for linearly constrained convex optimization, where full-dimensional vector operations can be  avoided for sparse matrix $A$, provided that the certain gradient terms can be  evaluated. While our work also eliminates these operations when  $S_k^\top A^\top$ is sparse, our derivation is distinct. We first exploit the specific expressions of the RCGLS parameters to establish a representation adopting a rescaled search direction, whereas \cite{zeng2026stochastic} focuses on variable transformations for arbitrary parameters. In particular, when $\mathcal{D}$ is a fixed distribution with $S=I$, this representation recovers a classic CG modification applied to the normal equations, which is one of the modifications identified as being of interest in \cite[Section 9]{hestenes1952methods}. Based on this, we then derive an equivalent formulation for efficient implementation.

\subsubsection{Randomized methods for ridge regression}

Ridge regression, also known as Tikhonov regularization, is an extension of standard least-squares problems, which stabilizes ill-posed and rank-deficient systems by imposing an $\ell_2$-regularization penalty. It has been widely applied in numerous scientific and engineering fields, including finance \cite{ahn2012using}, image processing \cite{xue2009local}, and machine learning \cite{rajan2022efficient}.  Given that ridge regression can be reformulated as a linear system, the randomized iterative methods discussed in Section \ref{section-1-2-1}  have been extended to solve this problem.

In \cite{ivanov2013kaczmarz}, Ivanov and Zhdanov  solved ridge regression by directly applying the classical RK method to the augmented regularized normal equation. Then, Hefny et al. \cite{hefny2017rows} verified that such naive direct implementations suffer from suboptimal convergence performance, and accordingly developed improved RK and RGS variants under the unified RCD framework. More recently, Gazagnadou et al. \cite{gazagnadou2022ridgesketch} extended the sketch-and-project method \cite{Gow15} to the related normal equation, yielding the RidgeSketch method, and further incorporated momentum to enhance performance.  However, those existing randomized ridge regression solvers suffer from limitations. The improved RK and RGS variants lack acceleration, while RidgeSketch involves per-iteration least-squares subproblems, which can not be parallelizable. In this paper, the proposed RidgeRCGLS method can  inherit the lightweight computation, parallelizability, and accelerated convergence of the RCGLS method, offering an efficient and flexible solver for large-scale ridge regression problems.

\subsection{Notations}

For any random variables $\xi$ and $\zeta_0$, we use $\mathbb{E}[\xi]$ and $\mathbb{E}[\xi\mid \zeta=\zeta_0]$ to denote the expectation of $\xi$ and the conditional expectation of
$\xi$ given $\zeta = \zeta_0$.
  For any matrix $A \in \mathbb{R}^{n \times d}$, we use $A_{i, :}$, $A_{:, j}$, $A^\top$, $\|A\|_F$, $\sigma_{\min}(A)$, $\operatorname{Range}(A)$, and $\operatorname{Null}(A)$ to denote the $i$-th row, the $j$-th column, the transpose, the Frobenius norm, the smallest nonzero singular value, the column space, and the null space of $A$, respectively. For a given index set $\mathcal{J}$,
we use $A_{\mathcal{J} ,:}$ and $A_{:,\mathcal{J} }$ to denote the row and column submatrix indexed by $\mathcal{J} $, respectively. The cardinality of the set $\mathcal{J}$ is denoted by $|\mathcal{J}|$.
For any vector $b \in \mathbb{R}^n$, we use $b_i$ and $\|b\|_2$ to denote the $i$-th  entry and the Euclidean norm of $b$, respectively. In addition, for any positive semidefinite matrix $H \in \mathbb{R}^{n \times n}$, we define the $H$-inner product and the induced $H$-norm by
$
\langle x,y \rangle_{H}= \langle x, Hy \rangle$ and $ \|x\|_{H}=\sqrt{\langle x, x \rangle_{H}}
$, respectively. Finally, we adopt the convention that $\frac{0}{0} = 0$ throughout this paper.

\subsection{Organization}

The remainder of the paper is organized as follows. In Section 2, we propose the RCGLS method and show its linear convergence rate. In Section 3, we show that RCGLS can be equivalently reformulated to substantially avoid full-dimensional operations. In Section 4, we extend the RCGLS method to solve ridge regression. In Section 5, we perform some numerical experiments to show the effectiveness of the proposed method. We conclude the paper in Section 6. 
					
\section{Randomized conjugate gradient least squares} \label{Section-2}

In this section, we first present a novel derivation of the classical CGLS method \cite[Section 10]{hestenes1952methods}. Based on this new interpretation, we further develop the RCGLS method. We theoretically prove that the proposed RCGLS method achieves linear convergence in expectation and admits a tighter convergence factor than the classical RCD method.

\subsection{A new derivation of CGLS}
\label{section-per-CGLS}
For convenience, we define $f(x) = \frac{1}{2} \|Ax - b\|_{2}^{2}$. Accordingly, the least-squares problem \eqref{LS_pro} is equivalent to $\min_{x \in \mathbb{R}^{d}} f(x)$.
 At the $k$-th iteration, let $p^k$ be a descent direction. We adopt the standard update rule
$$
x^{k+1} = x^{k} + \mu_{k} p^{k},
$$
where the stepsize $\mu_k$  is determined via exact line search 
\begin{equation}
	\label{exact-line-pa}
	\min_{\mu \in \mathbb{R}} f(x^{k} + \mu p^{k}).
\end{equation}
Using the quadratic structure of $f$, the optimal stepsize admits the closed-form
\begin{equation}
	\label{exact-line-step}
	\mu_k= \frac{\langle r^{k}, Ap^{k} \rangle}{\|Ap^{k}\|_{2}^{2}},
\end{equation}
where $r^k := b - Ax^k$ denotes the residual vector.

Let $x^*$ be a solution to \eqref{LS_pro}, and set $p^k = x^* - x^k$. Substituting this direction into the stepsize formula and using the optimality condition $A^\top b = A^\top A x^*$, we derive
$$
\mu_{k}=\frac{\langle b-Ax^k, A(x^*-x^{k}) \rangle}{\|A(x^*-x^{k})\|_{2}^{2}}=\frac{\langle A^\top(b-Ax^k), x^*-x^{k} \rangle}{\|A(x^*-x^{k})\|_{2}^{2}}
=\frac{\langle A^\top A(x^*-x^k), x^*-x^{k} \rangle}{\|A(x^*-x^{k})\|_{2}^{2}}=1.
$$
Consequently, $x^{k+1}=x^k+\mu_k p^k = x^*$, so the optimal solution is obtained in a single step. We thus call the error vector $x^* - x^k$ an \emph{ideal descent direction}. A practical iteration therefore requires search directions to approximate $e^k$ as closely as possible. 

From the first-order optimality of exact line search \eqref{exact-line-pa}, we have $\langle \nabla f(x^{k+1}), p^k \rangle = 0$, which yields
\begin{equation} \label{orth}
	\langle x^*-x^{k+1} , p^k \rangle_{A^\top A} = \langle A^\top A(x^* - x^{k+1}), p^k \rangle = -\langle \nabla f(x^{k+1}), p^k \rangle = 0.
\end{equation}
This shows that the ideal descent direction $x^*-x^{k+1} $ is $A^\top A$-conjugate to the previous search direction $p^k$.
Hence, we require that the subsequent direction $p^{k+1}$  satisfies $\langle p^{k+1}, p^{k} \rangle_{A^\top A}=0$. Nevertheless,  enforcing only this conjugacy constraint  cannot guarantee a reliable descent direction. To ensure favorable local reduction, one may further require the new direction to align with the negative gradient $-\nabla f(x^{k+1})$. Combining the conjugacy constraint and gradient alignment, we construct $p^{k+1}$ by solving the constraint correction problem
\begin{equation} \label{cons1}
	\min\limits_{p \in \mathbb{R}^{d}} \|p-(-\nabla f(x^{k+1})) \|_{A^{\top}A}^{2}, \quad \text{subject to} \quad \langle p, p^{k} \rangle_{A^{\top}A}=0.
\end{equation}
Note that $\nabla f(x^{k+1})=A^\top(Ax^{k+1}-b)=-A^\top r^{k+1}$. One minimizer of \eqref{cons1} takes the form
\begin{equation} \nonumber
	p^{k+1} = A^\top r^{k+1} + \tau_{k} p^{k}, \quad \tau_{k} = -\frac{\langle AA^\top r^{k+1}, Ap^{k} \rangle}{\|Ap^{k}\|_{2}^{2}}.
\end{equation}
This directly recovers the standard CGLS iteration \cite[Section 10]{hestenes1952methods}
\begin{equation} \label{CGLS_IP}
	\left\{\begin{aligned}
		&\mu_k = \frac{\langle r^k, Ap^k \rangle}{\|Ap^{k}\|_{2}^{2}}, \\
		&x^{k+1} = x^k + \mu_k p^k, \\
		&r^{k+1} = r^k - \mu_k Ap^k, \\
		&\tau_k = -\frac{\langle AA^\top r^{k+1}, Ap^k \rangle}{\|Ap^k\|_{2}^{2}}, \\
		&p^{k+1} = A^\top r^{k+1} + \tau_{k} p^{k}.
	\end{aligned}\right.
\end{equation}
 We note that $\mu_k$ and $\tau_k$ can also be expressed in the following equivalent forms
\(
\mu_{k}=\frac{\|A^\top r^k\|^2_2}{\|Ap^k\|^2_2}\) and \(
\tau_{k}=\frac{\|A^\top r^{k+1}\|^2_2}{\|A^\top r^k\|_{2}^{2}}
\),
as can be seen from \eqref{simp-form} and \eqref{simp-tau}, together with the gradient orthogonality property inherited from \cite[Theorem 5.1]{hestenes1952methods}, i.e., $\langle A^{\top}r^{k+1}, A^{\top}r^{k} \rangle=0$.

\subsection{The RCGLS method}

Building on the new perspective of CGLS presented in the previous subsection, we naturally derive the proposed RCGLS method. The underlying idea is straightforward. We require each new search direction to align with the \emph{randomized coordinate gradient} $S_kS_k^\top\nabla f(x^{k})$, where $S_k\in\mathbb{R}^{d\times q}$ denotes a randomized sketching matrix sampled from a user-specified distribution $\mathcal{D}$. In particular, in contrast to \eqref{cons1}, the direction $p^{k+1}$ is now obtained by
\begin{equation} \label{cons1R}
	\min\limits_{p \in \mathbb{R}^{d}} \|p-(-S_kS_k^\top\nabla f(x^{k+1})) \|_{A^{\top}A}^{2}, \quad \text{subject to} \quad \langle p, p^{k} \rangle_{A^{\top}A}=0.
\end{equation}

Let $p^0 = S_0 S_0^\top A^\top r^0$ and $v^0 = Ap^0$. 
 Following a similar argument as in the previous subsection, we obtain the following iteration scheme of the proposed RCGLS method
\begin{equation} \label{RCGLS_scheme}
	\left\{\begin{aligned}
		&	\mu_k = \frac{\langle r^k, v^k \rangle}{\|v^{k}\|_{2}^{2}}, \\
		&	x^{k+1} = x^k + \mu_k p^k, \\
		&	r^{k+1} = r^k - \mu_k v^k, \\
		&	\tau_k = -\frac{\langle A S_{k+1} S_{k+1}^{\top} A^\top r^{k+1}, v^{k} \rangle}{\|v^{k}\|_{2}^{2}}, \\
		&	p^{k+1} = S_{k+1} S_{k+1}^{\top} A^{\top}r^{k+1} + \tau_{k} p^{k}, \\
		&	v^{k+1} = AS_{k+1} S_{k+1}^{\top} A^{\top}r^{k+1}+ \tau_{k} v^{k}.
	\end{aligned}\right.
\end{equation}
We further show that the stepsize $\mu_{k}$ in \eqref{RCGLS_scheme} can be rewritten as 
\begin{equation}
	\label{simp-form}
	\mu_k = \frac{\|S_k^\top A^\top r^k\|_2^2}{\|v^k\|_2^2}.
\end{equation}
Indeed, for $k=0$, this identity holds directly from the definition of $v^0$. For $k \geq 1$, we reformulate the inner product $\langle r^k, v^k \rangle$ as
$$
\langle r^k, v^k \rangle=\langle r^k, AS_{k} S_{k}^{\top} A^{\top}r^{k} + \tau_{k-1} v^{k-1} \rangle=\|S_{k}^{\top}A^{\top}r^{k}\|_{2}^{2}+\tau_{k-1} \langle A^{\top}r^k, p^{k-1} \rangle=\|S_{k}^{\top}A^{\top}r^{k}\|_{2}^{2},
$$
where the last equality follows from the first-order optimality condition of exact line search \eqref{exact-line-pa}. Specifically, the orthogonality relation $\langle A^\top r^k, p^{k-1} \rangle=-\langle \nabla f(x^{k}), p^{k-1} \rangle=0$ holds at the previous iteration, which eliminates the cross term. Besides, we also have 
\begin{equation}
	\label{simp-tau}
	\begin{aligned}
		 \tau_k&=-\frac{\langle  S_{k+1}^{\top} A^\top r^{k+1}, S_{k+1}^{\top} A^\top v^{k} \rangle}{\|v^{k}\|_{2}^{2}}=\frac{\langle  S_{k+1}^{\top} A^\top r^{k+1}, S_{k+1}^{\top} A^\top (r^{k+1}-r^k) \rangle}{\mu_k\|v^{k}\|_{2}^{2}}\\
		 &=\frac{\langle  S_{k+1}^{\top} A^\top r^{k+1}, S_{k+1}^{\top} A^\top (r^{k+1}-r^k) \rangle}{\langle r^k, v^k \rangle}=\frac{\|S_{k+1}^\top A^\top r^{k+1}\|^2_2-\langle S_{k+1}^\top A^\top r^{k+1},S_{k+1}^\top A^\top r^{k}\rangle}{\|S_{k}^\top A^\top r^{k}\|^2_2},
	\end{aligned}
\end{equation}

Now, we are ready to present the RCGLS method, which is formally described in
Algorithm \ref{RCGLS_Alg}. 

		\begin{algorithm}[htpb]
			\caption{Randomized CGLS (RCGLS) \label{RCGLS_Alg}}
			\begin{algorithmic}
				\Require
				$A \in \mathbb{R}^{n\times d}$, $b \in \mathbb{R}^n$, distribution $\mathcal{D}$, $k=0$, and the initial point $x^{0} \in \mathbb{R}^d$.
				\begin{enumerate}
					\item[1:] Randomly select a sampling matrix $S_{0}$ from $\mathcal{D}$.
					
					\item[2:] Set $r^{0}=b-Ax^{0}$, $p^{0}=S_{0}S_{0}^{\top}A^{\top}r^{0}$, and $v^{0}=AS_{0}S_{0}^{\top}A^{\top}r^{0}$.
					
					\item[3:] Set $\mu_{k}=\frac{\|S_{k}^{\top}A^{\top}r^{k}\|_{2}^{2}}{\|v^{k}\|_{2}^{2}}$.
					
					\item[4:]  Update $x^{k+1}=x^{k}+\mu_{k} p^{k}$ and $r^{k+1}=r^{k}-\mu_{k} v^{k}$.
					
					\item[5:]  Randomly select a sampling matrix $S_{k+1} $ from $\mathcal{D}$.
					
					\item[6:]  Compute
					$$
					\begin{aligned}
						\tau_k&=-\frac{\langle A S_{k+1} S_{k+1}^{\top} A^\top r^{k+1}, v^{k} \rangle}{\|v^{k}\|_{2}^{2}},\\ 
						p^{k+1}&= S_{k+1} S_{k+1}^{\top} A^{\top}r^{k+1}+\tau_{k} p^{k}, \\
						v^{k+1}&= AS_{k+1} S_{k+1}^{\top} A^{\top}r^{k+1}+\tau_{k} v^{k}.
					\end{aligned}
					$$
					\item[7:] If the stopping rule is satisfied, stop and go to output. Otherwise, set $k=k+1$ and return to Step $3$.
				\end{enumerate}
				
				\Ensure
				The approximate solution $x^k$.
			\end{algorithmic}
		\end{algorithm}
		
		\begin{remark} \label{remark-VR-RCD}
		We refer to $S_kS_k^\top\nabla f(x^{k})$ as the randomized coordinate gradient, because if we use only this gradient, the corresponding method
		\begin{equation}
				\label{GRCD}
				x^{k+1}=x^k-\alpha_k S_kS_k^\top\nabla f(x^{k})
			\end{equation}
			is a generalized RCD (GRCD) method \cite{zeng2026stochastic, han2026pseudoinverse}. Indeed, choose $\mathcal{D}$  as follow:
			$S= \frac{e_i}{\|A_{:,i}\|_2}$ with probability $\text{Prob}(i_k =i)=\frac{\|A_{:,i}\|^2_2}{\|A\|^2_F}$. Then the sketched gradient satisfies
			$$
			S_kS_k^\top\nabla f(x^{k})=\frac{\langle A_{:,i_k}, Ax^k-b\rangle }{\|A_{:,i_k}\|^2_2}e_{i_k}=-\frac{\langle A_{:,i_k}, r^k\rangle }{\|A_{:,i_k}\|^2_2}e_{i_k}.
			$$
			With the stepsize $\alpha_k=1 $, the sketched gradient method \eqref{GRCD} reduces to
			\begin{equation}
				\label{rcd-c}
				x^{k+1}= x^k+\frac{\langle A_{:,i_k}, r^k\rangle}{\|A_{:,i_k}\|^2_2}e_{i_k},
			\end{equation}
			which is exactly the RCD method \cite{wright2015coordinate,nesterov2012efficiency,fercoq2015accelerated}, also known as the RGS method \cite{han2026pseudoinverse,leventhal2010randomized}.
		\end{remark}

	\begin{remark}\label{remark-VR-RCD2}
	The reason we adopt the randomized coordinate gradient $S_kS_k^\top\nabla f(x^{k})$ instead of the stochastic gradient  lies in its inherent variance reduction (VR) property \cite{gower2020variance}. Let $g^k$ denote a gradient estimator of $\nabla f(x^k)$. Based on this estimator, the standard approximate gradient update takes the form $x^{k+1} = x^k - \gamma g^k$ with a constant step size $\gamma > 0$. To guarantee convergence with a constant stepsize, the gradient estimator $g^k$ is required to satisfy the vanishing variance condition
	\begin{equation}
		\label{VRP-0501-1}
		\lim_{k \rightarrow \infty} \mathbb{E}\left[\|g^k - \nabla f(x^k)\|^2_2\right] = 0.
	\end{equation}
	Property \eqref{VRP-0501-1} is precisely the VR property \cite{gower2020variance}. For the randomized coordinate gradient estimator, this property holds naturally. Specifically,
	\[
	\lim_{x^k \to x^*} \mathbb{E}\left[\|SS^\top\nabla f(x^k) - \nabla f(x^k)\|^2_2\right] 
	= \lim_{x^k \to x^*} \mathbb{E}\left[\|(SS^\top - I)\nabla f(x^k)\|^2_2\right] = 0,
	\]
	where the last equality follows from the optimality condition $\nabla f(x^*) = 0$.
	
	In fact, another widely used approach is stochastic gradient descent (SGD), which employs the stochastic gradient estimator $A^\top S_k S_k^\top (Ax^k - b)$ instead of the randomized coordinate gradient $S_k S_k^\top \nabla f(x^k) = S_k S_k^\top A^\top (Ax^k - b)$. Nevertheless, the standard SGD estimator generally fails to satisfy the VR property \eqref{VRP-0501-1} unless the interpolation condition holds \cite[Section 4.3]{garrigos2023handbook} or the linear system $Ax = b$ is consistent \cite[Section 5.1]{zeng2025adaptive}. We note that for the consistent case, the relationship between stochastic heavy-ball momentum methods and the stochastic conjugate gradient normal equation error (CGNE) method has been investigated in \cite{zeng2024adaptive}.
	\end{remark}

\subsection{Convergence analysis}
In this subsection, we investigate the expected linear convergence of the proposed RCGLS method. We demonstrate that the RCGLS method achieves an improved convergence factor compared with the vanilla GRCD method \eqref{GRCD}. We first state the necessary assumption on the distribution $\mathcal{D}$. We note that numerous practical sampling strategies satisfy this assumption, including partition sampling, uniform coordinate sampling, and Gaussian sketching. For a comprehensive discussion regarding the construction and selection of valid distributions, we refer readers to \cite[Section 5]{xie2025randomized}.

\begin{assumption} \label{ass}
	Let $\mathcal{D}$ denote the distribution from which the randomized sketching matrices are sampled. We assume that $\mathbb{E}[SS^{\top}]$ is positive definite.
\end{assumption}

For convenience, let us introduce some notations.  We define
\begin{equation}
	\label{def-gamma}
	\gamma_0:=1\ \ \text{and} \ \ \gamma_{k}:=\inf\limits_{S_k \sim \mathcal{D}} \left\{\left(1-\frac{\langle S_{k} S_{k}^{\top} A^{\top}r^{k}, p^{k-1} \rangle_{A^{\top}A}^{2}}{\|S_{k}  S_{k}^{\top} A^{\top}r^{k}\|_{A^{\top}A}^2 \| p^{k-1} \|_{A^{\top}A}^2} \right)^{-1}\right\}, \ \ \text{if} \ \ k\geq1.
\end{equation}
Since $\langle S_{k} S_{k}^{\top} A^{\top}r^{k}, p^{k-1} \rangle_{A^{\top}A}^{2} \leq \|S_{k}  S_{k}^{\top} A^{\top}r^{k}\|_{A^{\top}A}^2 \| p^{k-1} \|_{A^{\top}A}^2$, which implies $\gamma_{k}\geq1$ for all $k\geq 0$.
We further define
\begin{equation}\label{M}
	M:=\mathop{\mathbb{E}} \left[\frac{S S^{\top}}{\|AS\|_2^2}\right],
\end{equation}
here we define $\frac{0}{0}=0$.
The following lemma shows that $M$ is positive definite under Assumption \ref{ass}.
\begin{lemma}[\cite{lorenz2025minimal}, Lemma 2.3]
	Suppose Assumption \ref{ass} holds and $A \neq 0$. Then, the matrix $M$ defined in \eqref{M} is positive definite.
\end{lemma}
 
We have the following convergence result for Algorithm \ref{RCGLS_Alg}.

\begin{thm}
	\label{RCGLS_rate}
	Suppose that $x^{*}$ is a solution to the LS problem \eqref{LS_pro}, and the distribution $\mathcal{D}$ satisfies Assumption \ref{ass}.  Let $\{x^{k}\}_{k \geq 0}$ be the iteration sequence generated by Algorithm \ref{RCGLS_Alg}. Then
	\begin{equation} \label{RCG_rate_inequ}
		\mathbb{E} \left[\| x^{k+1}-x^{*} \|_{A^{\top}A}^2 \mid x^k\right] \leq \left(1 - \gamma_k \sigma^2_{\min} (AM^{\frac{1}{2}}  ) \right) \|x^{k}-x^{*} \|_{A^{\top}A}^2,
	\end{equation}
	where $\gamma_k\geq 1$ and $M$ are defined by \eqref{def-gamma} and \eqref{M}, respectively.
\end{thm}

\begin{proof} By the iteration scheme of $x^{k+1}$, for all $k \geq 1$, we have
 \begin{equation} \label{AHBM-Pri-Pro}
 	\begin{aligned}
 		\|x^{k+1}-x^{*}\|_{A^{\top}A}^{2}
 		&= \|x^k+\mu_k p^{k}-x^{*}\|_{A^{\top}A}^{2} \\
 		&= \|x^{k}-x^{*}\|_{A^{\top}A}^{2}+2 \mu_k \langle x^{k}-x^{*}, p^{k} \rangle_{A^{\top}A}+\mu_k^2 \| p^{k}\|_{A^{\top}A}^{2} \\
 		&= \|x^{k}-x^{*}\|_{A^{\top}A}^{2}+2 \mu_k \langle x^{k}-x^{*}, S_{k} S_{k}^{\top} A^{\top}r^{k}+\tau_{k-1} p^{k-1} \rangle_{A^{\top}A}+\mu_k^2 \| p^{k}\|_{A^{\top}A}^{2} \\
 		&= \|x^{k}-x^{*}\|_{A^{\top}A}^{2}-2 \mu_k \|S_{k}^{\top} A^{\top}r^{k}\|_2^2+\mu_k^2 \| p^{k}\|_{A^{\top}A}^{2} \\
 		&= \|x^{k}-x^{*}\|_{A^{\top}A}^{2}- \mu_k \|S_{k}^{\top} A^{\top}r^{k}\|_2^2,
 	\end{aligned}
 \end{equation}
 where the fourth equality holds due to the first-order optimality condition of exact line search \eqref{exact-line-pa} and  $A^\top Ax^*=A^\top b$, which yields $\langle x^{k}-x^{*}, p^{k-1} \rangle_{A^{\top}A}=0$, and the last equality follows  from \eqref{simp-form} and $v^k=Ap^k$. For $k=0$, one can readily verify that the relation \eqref{AHBM-Pri-Pro} also holds.
 
 We next establish a lower bound for the stepsize $\mu_k$ when $S^\top_k A^\top r^k\neq 0$. For $k=0$, since $p^{0}=S_{0} S_{0}^{\top} A^{\top}r^{0}$, we have
 \[
 \mu_{0}= \frac{\|S_0^\top A^\top r^0\|_2^2}{\|p^{0}\|_{A^{\top}A}^2}=\frac{\|S_0^\top A^\top r^0\|_2^2}{\|S_{0} S_{0}^{\top} A^{\top}r^{0}\|_{A^{\top}A}^2} \geq \frac{1}{\|AS_{0}\|_2^2}=\frac{\gamma_{0}}{\|AS_{0}\|_2^2}.
 \]
 For $k \geq 1$, we compute
 \begin{align*}
 	\|p^{k}\|_{A^{\top}A}^2
 	&= \|S_{k} S_{k}^{\top} A^{\top}r^{k}+\tau_{k-1} p^{k-1}\|_{A^{\top}A}^2 \\
 	&= \|S_{k} S_{k}^{\top} A^{\top}r^{k}\|_{A^{\top}A}^2+2 \tau_{k-1} \langle S_{k} S_{k}^{\top} A^{\top}r^{k}, p^{k-1} \rangle_{A^{\top}A}+\tau_{k-1}^2 \| p^{k-1} \|_{A^{\top}A}^2 \\
 	&= \|S_{k} S_{k}^{\top} A^{\top}r^{k}\|_{A^{\top}A}^2-\frac{\langle S_{k} S_{k}^{\top} A^{\top}r^{k}, p^{k-1} \rangle_{A^{\top}A}^{2}}{\| p^{k-1} \|_{A^{\top}A}^2} \\
 	&= \|S_{k} S_{k}^{\top} A^{\top}r^{k}\|_{A^{\top}A}^2 \left(1-\frac{\langle S_{k} S_{k}^{\top} A^{\top}r^{k}, p^{k-1} \rangle_{A^{\top}A}^{2}}{\|S_{k} S_{k}^{\top} A^{\top}r^{k}\|_{A^{\top}A}^2 \| p^{k-1} \|_{A^{\top}A}^2} \right) \\
 	&\leq \gamma^{-1}_{k} \|S_{k} S_{k}^{\top} A^{\top}r^{k}\|_{A^{\top}A}^2,
 \end{align*}
 where the last inequality follows from the definition of $\gamma_k$ in \eqref{def-gamma}. Substituting this result yields
 \begin{equation}  \label{proof-mu-0501}
 	\mu_k= \frac{\|S_k^\top A^\top r^k\|_2^2}{\|p^{k}\|_{A^{\top}A}^2}
 	\geq\frac{\|S_k^\top A^\top r^k\|_2^2}{\gamma^{-1}_{k} \|S_{k} S_k^\top A^\top r^k\|_{A^{\top}A}^2} \geq \frac{\gamma_{k}}{\|AS_{k}\|_2^2}.
 \end{equation}
 Thus, the lower bound \eqref{proof-mu-0501} holds for all $k\geq 0$ if $S^\top_k A^\top r^k\neq 0$.
 
	Substituting   \eqref{proof-mu-0501} into \eqref{AHBM-Pri-Pro} for the case where $S^\top_k A^\top r^k\neq 0$, we have
	$$
	\|x^{k+1}-x^{*}\|_{A^{\top}A}^{2} \leq \|x^{k}-x^{*}\|_{A^{\top}A}^{2}- \gamma_k \frac{\|S_{k}^{\top} A^{\top}r^{k}\|_2^2}{ \|AS_{k}\|_2^2}=\|x^{k}-x^{*}\|_{A^{\top}A}^{2}- \gamma_k \frac{\|S_{k}^{\top} A^{\top}A(x^{k}-x^{*})\|_2^2}{ \|AS_{k}\|_2^2}.
	$$
	Moreover, when $S_{k}^{\top}A^{\top}r^{k} = 0$, we have $x^{k+1}=x^{k}$, which implies that the above inequality also holds. Therefore, 
	\[
	\begin{aligned}
		\mathbb{E}_{k}\left[\|x^{k+1}-x^{*}\|_{A^{\top}A}^{2} \ \bigm|   \ x^k\right ]
		\leq& \mathbb{E} \left[\|x^{k}-x^{*}\|_{A^{\top}A}^{2} \ \bigm|   \ x^k\right ]- \gamma_k \mathbb{E} \left[ \frac{\|S_{k}^{\top}A^{\top}A(x^{k}-x^{*})\|_2^2}{ \|AS_{k}\|_2^2} \ \bigm|   \ x^k \right] \\ 
		=&\|x^{k}-x^{*} \|_{A^{\top}A}^2-\gamma_k \|M^{\frac{1}{2}}A^{\top}A(x^{k}-x^{*})\|_2^2 \\
		=&\|x^{k}-x^{*} \|_{A^{\top}A}^2-\gamma_k \|M^{\frac{1}{2}}(A^{\top}A)^{\frac{1}{2}} (A^{\top}A)^{\frac{1}{2}}(x^{k}-x^{*})\|_2^2
		\\
		\leq& \left(1-\gamma_k \sigma^2_{\min}(A M^\frac{1}{2} )\right)\|x^{k}-x^{*} \|_{A^{\top}A}^2,\\
	\end{aligned}
	\]
	where  the last inequality follows from $M$ is positive definite, $(A^{\top}A)^{\frac{1}{2}}(x^{k}-x^{*}) \in \text{Range}((A^{\top}A)^{\frac{1}{2}})$ and $\sigma^2_{\min}((A^{\top}A)^{\frac{1}{2}}M^{\frac{1}{2}})=\sigma^2_{\min}(A M^\frac{1}{2} )$.
	This completes the proof of this theorem.
\end{proof}


\begin{remark}\label{remark-VR}
	We note that \eqref{RCG_rate_inequ} in Theorem \ref{RCGLS_rate} indicates that the RCGLS method exhibits the variance reduction property. In fact, supposing $\mathbb{E}[x]$ is bounded for all $x\in\mathbb{R}^d$, we have 
	\[
	\mathbb{E}\left[\left\|Ax-Ax^* \right\|^2_2 \right]= \left\| \mathbb{E}\left[ Ax-Ax^* \right]\right\|^2_{2}+\mathbb{E}\left[ \left\|Ax-\mathbb{E}[Ax]\right\|^2_{2}\right].
	\]
	This relation implies that the convergence of $\mathbb{E}\left[\left\|Ax-Ax^* \right\|^2_{2}\right]$ guarantees the convergence of $\mathbb{E}\left[ \left\|Ax-\mathbb{E}[Ax]\right\|^2_{2}\right]$, i.e., the reduction of variance, and is consistent with the discussion in Remark \ref{remark-VR-RCD2}.
\end{remark}

\begin{remark} \label{rem:comparison_and_acceleration}
We compare the convergence upper bound derived in Theorem \ref{RCGLS_rate} with that of the GRCD method \eqref{GRCD}. With the stepsize determined by the exact line search rule \eqref{exact-line-step}, the GRCD method yields the following iteration scheme
\[
x^{k+1} = x^k + \frac{\|S_{k}^{\top}A^{\top}r^k\|_2^2}{\|A S_k S_{k}^{\top}A^{\top} r^k \|_{2}^2} S_k S_{k}^{\top}A^{\top} r^k.
\]
By using a similar analysis as in the proof of Theorem \ref{RCGLS_rate} with $p^k=S_k S_{k}^{\top}A^{\top} r^k$, we can establish the following convergence result for the GRCD method
\begin{equation} \label{GRCD_rate}
	\mathbb{E}_k \left[\|x^{k+1}-x^*\|_{A^{\top}A}^2\right] \leq \left( 1 - \sigma^2_{\min}\left(A M^\frac{1}{2}\right) \right) \|x^k-x^*\|_{A^{\top}A}^2.
\end{equation}
In particular, choose the distribution $\mathcal{D}$  as $S=\frac{e_{i}}{\|A_{:,i}\|_{2}} $  with  probability $\frac{\|A_{:,i}\|_{2}^{2}}{\|A\|_{F}^{2}}$. In this case, we have
\(
1 - \sigma^2_{\min}\left(A M^\frac{1}{2}\right) = 1-\frac{\sigma_{\min}^2(A)}{\|A\|^2_F}
\),
which exactly recovers the convergence upper bound of the classical RCD method \eqref{rcd-c} \cite[Theorem 3.1]{leventhal2010randomized}. By comparing \eqref{RCG_rate_inequ} with \eqref{GRCD_rate}, it is clear that the proposed RCGLS method achieves a convergence upper bound that is at least  that of the GRCD method.
\end{remark}

The following remark illustrates that the parameter $\gamma_{k}$ in \eqref{RCG_rate_inequ} can be strictly larger than one in certain cases.

\begin{remark} \label{rem:>1}
	Consider the deterministic distribution $\mathcal{D}$  with $S=I$, for which Algorithm \ref{RCGLS_Alg} reduces to the classical CGLS method. For any iteration $k$ satisfying $A^{\top}r^{k} \neq 0$, the standard properties of conjugate gradient methods (see \cite[Theorem 6.1]{hestenes1952methods}) ensure that $x^{k} \neq x^{k-1}$, which implies $\mu_{k-1} \neq 0$. Combining this fact with the update rule $x^{k} =x^{k-1} + \mu_{k-1} p^{k-1}$, we derive the following identity
    \[
		\begin{aligned}
			\langle A^{\top}r^{k}, p^{k-1} \rangle_{A^{\top}A}
			=&\frac{1}{\mu_{k-1}} \langle A^{\top}r^{k}, x^{k}-x^{k-1} \rangle_{A^{\top}A} \\
            =& \frac{1}{\mu_{k-1}} \langle A^{\top}r^{k}, A^{\top}(Ax^{*}-Ax^{k-1})-A^{\top}A(x^{*}-x^{k}) \rangle \\
            =& \frac{1}{\mu_{k-1}} \left(\langle A^{\top}r^{k}, A^{\top}r^{k-1} \rangle-\|A^{\top}r^{k}\|_{2}^{2}\right) \\
			=&-\frac{1}{\mu_{k-1}}\|A^{\top}r^{k}\|_{2}^{2},
		\end{aligned}
	\]    
	where the last equality follows from the gradient orthogonality property inherited from \cite[Theorem 5.1] {hestenes1952methods}. As a result, we have $\gamma_{k}  >1$ whenever $A^{\top}r^{k} \neq 0$. Furthermore, this conclusion can be extended to the deterministic setting, where  $\mathcal{D}$ is a fixed distribution and  $S$  is a given matrix with 
	$SS^{\top}$ positive definite.
\end{remark}

		\section{Efficient implementation for sparse data}

The RCGLS method outlined in Algorithm \ref{RCGLS_Alg} requires full-dimensional vector operations at each iteration. In particular, the updates of $\mu_k$, $x^{k+1}$, and $r^{k+1}$ rely on $p^{k}$ and $v^{k}$, which can be computationally expensive since both $p^{k}$ and $v^{k}$ are generally dense vectors. These full-dimensional vector updates cost $\mathcal{O}(n+d)$ operations per iteration, making the overall complexity of RCGLS comparable to, or even higher than, that of the GRCD method \eqref{GRCD}. 
Indeed,  the GRCD method \eqref{GRCD}, which can be equivalently reformulated as
\[
\left\{\begin{aligned}
	x^{k+1} &= x^k + \alpha_k S_{k}S_{k}^{\top}A^{\top}r^{k}, \\
	r^{k+1} &= r^k - \alpha_k AS_{k}S_{k}^{\top}A^{\top}r^{k},
\end{aligned}\right.
\]
can potentially circumvent such computational costs when $A$ is sparse, as the update terms $S_{k}S_{k}^{\top}A^{\top}r^{k}$ and $AS_{k}S_{k}^{\top}A^{\top}r^{k}$ may remain sparse under this setting.

By first establishing a representation adopting a rescaled search direction, we develop an equivalent formulation of Algorithm \ref{RCGLS_Alg} inspired by variable transformation techniques \cite{lee2013efficient, fercoq2015accelerated, zeng2026stochastic}, which is presented as Algorithm \ref{RCGLS_EI_Q}.

		\begin{algorithm}[htpb]
			\caption{An efficient implementation of RCGLS \label{RCGLS_EI_Q}}
			\begin{algorithmic}
				\Require
				$A \in \mathbb{R}^{n\times d}$, $b \in \mathbb{R}^n$, distribution $\mathcal{D}$, $k=0$, and the initial point $x^{0} \in \mathbb{R}^d$. Set $\delta_{-1}^{*}=0$ and $(h^{0}, h_{*}^{0})=(x^{0}, Ax^{0})$.
				\begin{enumerate}
					\item[1:] Randomly select a sampling matrix $S_{0}$ from $\mathcal{D}$.
					
					\item[2:] Set $d_{1}^{0}=S_{0}^{\top}A^{\top}(b-Ax^{0})$, $q^{0}=S_{0}d_{1}^{0}/\|d_{1}^{0}\|_{2}^{2}$, $q_{*}^{0}=AS_{0}d_{1}^{0}/\|d_{1}^{0}\|_{2}^{2}$, and $l_{0}=\|q_{*}^{0}\|_{2}^{2}$. If $d_{1}^{0} \neq 0$, set $\theta_{-1}=1/\|d_{1}^{0}\|_{2}^{2}$; otherwise, set $\theta_{-1}=1$.
					
					\item[3:] Set $\eta_k = \frac{\theta_{k-1} \|d_{1}^{k}\|_{2}^{2}}{l_{k}}$ and $\delta_{k} = \delta_{k-1}^{*} + \eta_k$.
					
					\item[4:] Randomly select a sampling matrix $S_{k+1}$ from $\mathcal{D}$.
					
					\item[5:] Compute $d_{1}^{k+1}=S_{k+1}^{\top}A^{\top}b-S_{k+1}^{\top}A^{\top}h_{*}^{k}-\delta_{k}S_{k+1}^{\top}A^{\top}q_{*}^{k}$, $d^{k+1}=S_{k+1}d_{1}^{k+1}$, and $d_{2}^{k+1}=AS_{k+1}d_{1}^{k+1}$.
					
					\item[6:] If $\langle d_{2}^{k+1},q_{*}^{k} \rangle \neq 0$

					
					\quad Set $\delta_{k}^{*}=\delta_{k}$ and $\theta_{k}=-\frac{l_{k}}{\langle d_{2}^{k+1}, q_{*}^{k} \rangle}$.
					 
					 \quad Update 
					\[
					\begin{aligned}
						(h^{k+1}, q^{k+1})&=(h^{k}-\delta_{k} \theta_{k} d^{k+1}, q^{k}+\theta_{k} d^{k+1}),\\
						(h_{*}^{k+1}, q_{*}^{k+1})&=(h_{*}^{k}-\delta_{k} \theta_{k} d_{2}^{k+1}, q_{*}^{k}+\theta_{k} d_{2}^{k+1}),\\
						l_{k+1}&=\theta_{k}^{2} \|d_{2}^{k+1} \|_{2}^{2}-l_{k}.
					\end{aligned}
					\]
					%
					%
					
					Otherwise,
					
					\quad Set $\delta_{k}^{*}=0$ and $\theta_{k}=1$.
					
					\quad Update 
					\[
					\begin{aligned}
						(h^{k+1}, q^{k+1},h_{*}^{k+1}, q_{*}^{k+1})&=(h^{k}+\delta_{k} q^{k}, d^{k+1}, h_{*}^{k}+\delta_{k} q_{*}^{k}, d_{2}^{k+1}),\\
						l_{k+1}&=\|d_{2}^{k+1}\|_{2}^{2}.
					\end{aligned}
					\]
					%
					
					\item[7:] If the stopping rule is satisfied, stop and go to output. Otherwise, set $k=k+1$ and go to Step $3$.
				\end{enumerate}
				
				\Ensure
				The approximate solution $h^{k}+\delta_{k}q^{k}$.
			\end{algorithmic}
		\end{algorithm}

Since the equivalence between Algorithms \ref{RCGLS_Alg} and \ref{RCGLS_EI_Q} is not immediately obvious, we formally state it as the following result. The detailed proof is provided in the Subsection \ref{secA2}. 

\begin{prop} \label{Theo}
	Suppose that Algorithms \ref{RCGLS_Alg} and \ref{RCGLS_EI_Q} share the same sampling matrices $\{S_k\}_{k\geq0}$ and initial point $x^{0}$. Then, for any $k \geq 0$,
	\begin{equation} \nonumber
		x^{k+1} = h^{k} + \delta_{k}q^{k}.
	\end{equation}
	That is, Algorithms  \ref{RCGLS_Alg} and \ref{RCGLS_EI_Q} are equivalent.
\end{prop}


According to Proposition \ref{Theo}, explicit computation of the full iterate \(x^{k+1}\) in Algorithm \ref{RCGLS_EI_Q} is generally unnecessary. Such computation is only required when
\[
\langle S_{k+1}S_{k+1}^{\top}A^{\top}r^{k+1}, p^{k} \rangle_{A^{\top}A}=\frac{1}{\theta_{k-1}} \langle d^{k+1}, q^{k} \rangle_{A^{\top}A}=\frac{1}{\theta_{k-1}} \langle Ad^{k+1}, Aq^{k} \rangle=\frac{1}{\theta_{k-1}} \langle d_{2}^{k+1}, q_{*}^{k} \rangle=0,
\]
where the first equality follows from  the relation $p^{k}=\frac{q^{k}}{\theta_{k-1}}$ in \eqref{pq}.
This exactly corresponds to the case $\tau_{k}=0$ in the original RCGLS method. Instead of directly updating $x^{k+1}$, we introduce two auxiliary vectors $h^{k}, q^{k}$ and a scalar parameter $\delta_{k}$, such that the iterate can be implicitly represented via the decomposition $x^{k+1} = h^{k} + \delta_{k} q^{k}$.
Notably, the recursive updates of the auxiliary variables $h^{k}$, $q^{k}$, $h_{*}^{k}$, and $q_{*}^{k}$ only modify the entries associated with the nonzero patterns of $d^{k+1}=S_{k+1}d_{1}^{k+1}$ and $d_{2}^{k+1}=AS_{k+1}d_{1}^{k+1}$, where
\[
d_{1}^{k+1}=S_{k+1}^{\top}A^{\top}b-S_{k+1}^{\top}A^{\top}h_{*}^{k}-\delta_{k}S_{k+1}^{\top}A^{\top}q_{*}^{k}
\]
depends only on the sketched components of $A$. Meanwhile, the update of the scalar $\delta_{k}$ and other internal parameters in Algorithm \ref{RCGLS_EI_Q} merely require inner product and norm evaluations involving $d_{1}^{k+1}$ and $d_{2}^{k+1}$. As a result, the improved algorithm avoids expensive full-dimensional operations when the matrix $S_k^\top A^\top $ is sparse, yielding significant computational savings at each iteration.

	\subsection{Proof of the Proposition \ref{Theo}} \label{secA2}

To prove Proposition \ref{Theo}, we first introduce two key lemmas. For $k \geq 0$, consider the following iteration scheme
\begin{equation}\label{iteration-I} 
	\begin{cases}
		\eta_k = \frac{\theta_{k-1} \|S_{k}^{\top}A^{\top}r^{k}\|_{2}^{2}}{\|q_{*}^{k}\|_{2}^{2}}, \\
		x^{k+1} = x^k + \eta_k q^k, \\
		r^{k+1} = r^k - \eta_k q_{*}^k, \\ 
		\text{If} \; \langle AS_{k+1}S_{k+1}^{\top}A^{\top}r^{k+1}, q_{*}^{k} \rangle \neq 0, \; \text{then} \;
		\theta_{k}=-\frac{\|q_{*}^{k}\|_{2}^{2}}{\langle AS_{k+1}S_{k+1}^{\top}A^{\top}r^{k+1}, q_{*}^{k} \rangle}, \\
		\qquad \qquad \qquad \qquad \qquad \qquad \; \;
		q^{k+1}=q^{k}+\theta_{k} S_{k+1}S_{k+1}^{\top}A^{\top} r^{k+1}, \\
		\qquad \qquad \qquad \qquad \qquad \qquad \; \;
		q_{*}^{k+1}=
		q_{*}^{k}+\theta_{k} AS_{k+1}S_{k+1}^{\top}A^{\top} r^{k+1}.\\
		\text{If} \; \langle AS_{k+1}S_{k+1}^{\top}A^{\top}r^{k+1}, q_{*}^{k} \rangle = 0,  \; \text{then} \\
		\theta_{k}=1, \; q^{k+1}=S_{k+1}S_{k+1}^{\top}A^{\top} r^{k+1}, \text{and} \; q_{*}^{k+1}=AS_{k+1}S_{k+1}^{\top}A^{\top} r^{k+1}.
	\end{cases}
\end{equation}
The initial conditions are $x^{0} \in \mathbb{R}^{d}$ and $r^{0}=b-Ax^0$. We define $q^{0}=S_0 S_0^\top A^\top r^{0}/\|S_0^\top A^\top r^{0}\|_{2}^{2}$, $q_{*}^{0}=Aq^{0}$, and $\theta_{-1}=1/\|S_{0}^{\top}A^{\top}r^{0}\|_{2}^{2}$ if $S_{0}^{\top}A^{\top}r^{0} \neq 0$, otherwise $\theta_{-1}=1$.

The following lemma establishes the equivalence between Algorithm \ref{RCGLS_Alg} and the iteration scheme \eqref{iteration-I}.

\begin{lemma} \label{Lem_Eq1}
	Suppose that Algorithm \ref{RCGLS_Alg} and the iteration scheme \eqref{iteration-I} share the same initial point $x^0$ and sketching matrices $\{S_k\}_{k \geq 0}$. Then the sequence $\{x^k\}_{k \geq 0}$ generated by Algorithm \ref{RCGLS_Alg} is identical to the one generated by \eqref{iteration-I}. 
\end{lemma}

\begin{proof}
	Let the sequences $\{\mu_{k}, \tau_{k}, x^{k}, r^{k}, p^{k}, v^{k}\}_{k \geq 0}$ be generated by Algorithm \ref{RCGLS_Alg}. We define the scaling parameters $\{\widetilde{\theta}_{k}\}_{k \geq -1}$ by setting $\widetilde{\theta}_{-1}:=\frac{1}{\|S_{0}^{\top}A^{\top}r^{0}\|_{2}^{2}}$ if $S_{0}^{\top}A^{\top}r^{0} \neq 0$, and $\widetilde{\theta}_{-1}:=1$ otherwise; for $k \geq 0$,
	\begin{equation} \nonumber
		\widetilde{\theta}_{k}:=
		\left\{\begin{array}{ll}
			\frac{\widetilde{\theta}_{k-1}}{\tau_{k}} \quad \text{if} \; \langle AS_{k+1}S_{k+1}^{\top}A^{\top}r^{k+1}, v^{k} \rangle \neq 0;
			\\[1.7mm]
			1 \quad\quad\; \text{otherwise}.
		\end{array}
		\right.
	\end{equation}
	We show by induction that $\widetilde{\theta}_{k} \neq 0$ for all $k \geq 0$. For the base case $k=0$, if $\langle AS_{1}S_{1}^{\top}A^{\top}r^{1}, v^{0} \rangle \neq 0$, the definition of $\tau_{0}$ yields $\tau_{0}=-\frac{\langle AS_{1}S_{1}^{\top}A^{\top}r^{1}, v^{0} \rangle}{\|v^{0}\|_{2}^{2}} \neq 0$, which implies $\widetilde{\theta}_{0}=\frac{\widetilde{\theta}_{-1}}{\tau_{0}} \neq 0$ since $\widetilde{\theta}_{-1} \neq 0$. Otherwise, $\widetilde{\theta}_{0}=1 \neq 0$. Thus, $\widetilde{\theta}_{0} \neq 0$ holds. Assume $\widetilde{\theta}_{k} \neq 0$ for some $k \geq 0$. We consider the update for $\widetilde{\theta}_{k+1}$: if $\langle AS_{k+2}S_{k+2}^{\top}A^{\top}r^{k+2}, v^{k+1} \rangle \neq 0$, we have $\tau_{k+1}=-\frac{\langle AS_{k+2}S_{k+2}^{\top}A^{\top}r^{k+2}, v^{k+1} \rangle}{\|v^{k+1}\|_{2}^{2}} \neq 0$, yielding $\widetilde{\theta}_{k+1}=\frac{\widetilde{\theta}_{k}}{\tau_{k+1}} \neq 0$ by the inductive hypothesis; otherwise, $\widetilde{\theta}_{k+1}=1 \neq 0$. Hence, $\widetilde{\theta}_{k} \neq 0$ holds for all $k \geq 0$.
	
	Using the parameters $\{\widetilde{\theta}_{k}\}_{k \geq -1}$, we define $\widetilde{q}^{k} := \widetilde{\theta}_{k-1} p^{k}$ and $\widetilde{q}_{*}^{k} := A\widetilde{q}^{k}$. Since $\widetilde{\theta}_{k-1} \neq 0$ for all $k \geq 0$, the variables $p^{k}$ and $v^{k}$ can be expressed as $p^{k}=\frac{\widetilde{q}^{k}}{\widetilde{\theta}_{k-1}}$ and $v^{k}=Ap^{k}=\frac{A\widetilde{q}^{k}}{\widetilde{\theta}_{k-1}}=\frac{\widetilde{q}_{*}^{k}}{\widetilde{\theta}_{k-1}}$, respectively. By further defining the scaled stepsize $\widetilde{\eta}_k := \frac{\mu_k}{\widetilde{\theta}_{k-1}}$, the update for $x^{k+1}$ can be reformulated as
	$$
	x^{k+1}=x^k + \mu_k p^k=x^k + (\widetilde{\theta}_{k-1} \widetilde{\eta}_k) \frac{\widetilde{q}^{k}}{\widetilde{\theta}_{k-1}}= x^k + \widetilde{\eta}_k \widetilde{q}^k,
	$$
	which leads to $r^{k+1}=b-Ax^{k+1}=r^k - \widetilde{\eta}_k \widetilde{q}_{*}^k$. Moreover, Substituting $v^{k}=\frac{\widetilde{q}_{*}^{k}}{\widetilde{\theta}_{k-1}}$ into the definition of $\mu_{k}$ in Algorithm \ref{RCGLS_Alg} yields
	$$
	\widetilde{\eta}_k=\frac{1}{\widetilde{\theta}_{k-1}} \frac{\|S_{k}^{\top}A^{\top}r^{k}\|_{2}^{2}}{\|\widetilde{q}_{*}^{k}/\widetilde{\theta}_{k-1} \|_{2}^{2}}=\frac{\widetilde{\theta}_{k-1} \|S_{k}^{\top}A^{\top}r^{k}\|_{2}^{2}}{\|\widetilde{q}_{*}^{k}\|_{2}^{2}}.
	$$
	
	Next, we derive the explicit expression for $\widetilde{\theta}_k$ and the updates for $\widetilde{q}^{k+1}$ and $\widetilde{q}_{*}^{k+1}$ by considering the following two cases:
	
	\textbf{Case 1.} If $\langle AS_{k+1}S_{k+1}^{\top}A^{\top}r^{k+1}, \widetilde{q}_{*}^{k} \rangle \neq 0$, we have
	$$
	\tau_{k} = -\frac{\langle AS_{k+1}S_{k+1}^{\top}A^{\top}r^{k+1}, v^{k} \rangle}{\|v^{k}\|_{2}^{2}} = -\widetilde{\theta}_{k-1} \frac{\langle AS_{k+1}S_{k+1}^{\top}A^{\top}r^{k+1}, \widetilde{q}_{*}^{k} \rangle}{\|\widetilde{q}_{*}^{k}\|_{2}^{2}}.
	$$
	Thus, according to the relationship $\widetilde{\theta}_{k}=\frac{\widetilde{\theta}_{k-1}}{\tau_{k}}$, we can get
	$$
	\widetilde{\theta}_{k}=- \frac{\|\widetilde{q}_{*}^{k}\|_{2}^{2}}{ \langle AS_{k+1}S_{k+1}^{\top}A^{\top}r^{k+1}, \widetilde{q}_{*}^{k} \rangle}.
	$$
	In addition, the definition $\widetilde{q}^{k+1} = \widetilde{\theta}_k p^{k+1}$ gives
	$$
	\widetilde{q}^{k+1} = \widetilde{\theta}_{k} (S_{k+1}S_{k+1}^{\top}A^{\top}r^{k+1} + \tau_{k} p^{k}) = \widetilde{\theta}_{k} S_{k+1}S_{k+1}^{\top}A^{\top}r^{k+1} + \widetilde{\theta}_{k-1} p^{k}.
	$$
	Since $\widetilde{\theta}_{k-1} p^{k} = \widetilde{q}^{k}$, this simplifies to $\widetilde{q}^{k+1} = \widetilde{q}^{k} + \widetilde{\theta}_{k} S_{k+1}S_{k+1}^{\top}A^{\top}r^{k+1}$. Multiplying by $A$ on both sides yields $\widetilde{q}_{*}^{k+1} = \widetilde{q}_{*}^{k} + \widetilde{\theta}_{k} AS_{k+1}S_{k+1}^{\top}A^{\top} r^{k+1}$.
	
	\textbf{Case 2.} If $\langle AS_{k+1}S_{k+1}^{\top}A^{\top}r^{k+1}, \widetilde{q}_{*}^{k} \rangle = 0$, we have $\widetilde{\theta}_k = 1$ and $\tau_k = 0$. Thus, the update reduces to $\widetilde{q}^{k+1} = 1 \cdot p^{k+1} = S_{k+1}S_{k+1}^{\top}A^{\top}r^{k+1}$, which leads to $\widetilde{q}_{*}^{k+1} = AS_{k+1}S_{k+1}^{\top}A^{\top} r^{k+1}$.
	
	Since the constructed sequences $\{\widetilde{\theta}_k\}_{k \geq -1}$, $\{\widetilde{\eta}_k\}_{k \geq 0}$, $\{\widetilde{q}^k\}_{k \geq 0}$, and $\{\widetilde{q}_{*}^k\}_{k \geq 0}$ satisfy the same recursive relations and branching conditions as those defined in \eqref{iteration-I} , it follows that they coincide with the sequences $\{\theta_k\}_{k \geq -1}$, $\{\eta_k\}_{k \geq 0}$, $\{q^k\}_{k \geq 0}$, and $\{q_{*}^k\}_{k \geq 0}$, respectively.
	This completes the proof of the lemma.
\end{proof}
Based on the proof above, we know that for any $k \geq 0$,
\begin{equation} \label{pq}
	q^{k}=\theta_{k-1}p^{k},
	\end{equation}
where $p^{k}$, $q^{k}$, and $\theta_{k-1}$ are defined in Algorithm \ref{RCGLS_Alg} and \eqref{iteration-I}, respectively. This relation indicates that \eqref{iteration-I} adopts a rescaled search direction. In particular, when $\mathcal{D}$ is a fixed distribution with $S=I$, the following remark shows that \eqref{iteration-I} recovers the classic CG modification applied to the normal equations, which is one of the modifications identified as being of interest in \cite[Section 9]{hestenes1952methods}.

\begin{remark} \label{modification}
	When  $\mathcal{D}$ is a fixed distribution with $S=I$, we assume $A^{\top}r^{k} \neq 0$ for all $k \leq K$. According to \cite[Theorem 6.1]{hestenes1952methods}, it follows that $x^{k+1} \neq x^{k}$ for all $k < K$, which ensures $\eta_{k} \neq 0$. Therefore, for any $k < K$, utilizing the gradient orthogonality $\langle A^{\top}r^{k+1}, A^{\top}r^{k} \rangle=0$, we have
	$$
	\langle AA^{\top}r^{k+1}, q_{*}^{k} \rangle = \frac{1}{\eta_{k}} \langle AA^{\top}r^{k+1}, r^{k}-r^{k+1} \rangle=- \frac{\|A^{\top}r^{k+1}\|_{2}^{2}}{\eta_{k}} \neq 0. 
	$$
	Substituting this into the definition of $\theta_{k}$ yields the recurrence
	$$
	\theta_{k}=-\frac{\|q_{*}^{k}\|_{2}^{2}}{\langle AA^{\top}r^{k+1}, q_{*}^{k} \rangle}=\frac{\eta_{k} \|q_{*}^{k}\|_{2}^{2}}{\| A^{\top}r^{k+1}\|_{2}^{2}}=\frac{\theta_{k-1} \|A^{\top}r^{k}\|_{2}^{2}}{\|A^{\top}r^{k+1}\|_{2}^{2}},
	$$
	where the last equality follows from the definition of $\eta_{k}$. Given the initial condition $\theta_{-1}=\frac{1}{\|A^{\top}r^{0}\|_{2}^{2}}$, it follows by induction that $\theta_{k}=\frac{1}{\|A^{\top}r^{k+1}\|_{2}^{2}}$, which further implies $\eta_{k}=\frac{1}{\|q_{*}^{k}\|_{2}^{2}}$. Consequently, with $q_{*}^{k}=Aq^{k}$, \eqref{iteration-I} reduces to the following form: starting from $r^{0}=b-Ax^{0}$ and $q^{0}=A^\top r^{0}/\|A^\top r^{0}\|_{2}^{2}$, the update for $k \geq 0$ is given by
	\begin{equation} \nonumber
		\begin{cases}
			x^{k+1} = x^k + \frac{q^k}{\|Aq^{k}\|_{2}^{2}}, \\
			r^{k+1} = r^k - \frac{Aq^{k}}{\|Aq^{k}\|_{2}^{2}}, \\
			q^{k+1}=q^{k}+ \frac{A^{\top} r^{k+1}}{\|A^{\top}r^{k+1}\|_{2}^{2}}, \\
		\end{cases}
	\end{equation}
	which recovers the classic CG modification applied to the normal equations \cite[Section 9]{hestenes1952methods}.
	\end{remark}

For $k \geq 0$, consider the following iteration scheme
\begin{equation}\label{iteration-II} 
	\begin{cases}
		\eta_k = \frac{\theta_{k-1} \|d_{1}^{k}\|_{2}^{2}}{\|q_{*}^{k}\|_{2}^{2}}, \\
		\delta_{k} = \delta_{k-1}^{*} + \eta_k, \\
		d_{1}^{k+1}=S_{k+1}^{\top}A^{\top}b-S_{k+1}^{\top}A^{\top}h_{*}^{k}-\delta_{k}S_{k+1}^{\top}A^{\top}q_{*}^{k}, d^{k+1}=S_{k+1}d_{1}^{k+1}, d_{2}^{k+1}=AS_{k+1}d_{1}^{k+1}, \\ 
		\text{If} \; \langle d_{2}^{k+1}, q_{*}^{k} \rangle \neq 0, \; \text{then} \;
		\delta_{k}^{*}=\delta_{k}, \theta_{k}=-\frac{\|q_{*}^{k}\|_{2}^{2}}{\langle d_{2}^{k+1}, q_{*}^{k} \rangle}, \text{and}\\
		(h^{k+1}, q^{k+1},h_{*}^{k+1}, q_{*}^{k+1})=(h^{k}-\delta_{k} \theta_{k} d^{k+1}, q^{k}+\theta_{k} d^{k+1}, h_{*}^{k}-\delta_{k} \theta_{k} d_{2}^{k+1}, q_{*}^{k}+\theta_{k} d_{2}^{k+1}). \\
		\text{If} \; \langle d_{2}^{k+1}, q_{*}^{k} \rangle = 0,  \; \text{then} \; \delta_{k}^{*}=0, \theta_{k}=1, \text{and}\\
		(h^{k+1}, q^{k+1},h_{*}^{k+1}, q_{*}^{k+1})=(h^{k}+\delta_{k} q^{k}, d^{k+1}, h_{*}^{k}+\delta_{k} q_{*}^{k}, d_{2}^{k+1}).
	\end{cases}
\end{equation}
The initial conditions are $x^{0} \in \mathbb{R}^{d}$ and $\delta_{-1}^{*}=0$. We define $d_{1}^{0}=S_{0}^{\top}A^{\top}(b-Ax^{0})$, $(h^{0}, q^{0},h_{*}^{0}, q_{*}^{0})=(x^{0}, S_{0}d_{1}^{0}/\|d_{1}^{0}\|_{2}^{2}, Ax^{0}, AS_{0}d_{1}^{0}/\|d_{1}^{0}\|_{2}^{2})$, and $\theta_{-1}=1/\|d_{1}^{0}\|_{2}^{2}$ if $d_{1}^{0} \neq 0$, otherwise $\theta_{-1}=1$.

Based on Lemma \ref{Lem_Eq1}, we establish the following result, which shows that Algorithm \ref{RCGLS_Alg} and the iteration scheme \eqref{iteration-II} are equivalent.

\begin{lemma} \label{Lem_Eq2}
	Let the sequences $\{x^k\}_{k \geq 0}$ and $\{h^k, q^k, \delta_k\}_{k \geq 0}$ be generated by Algorithm \ref{RCGLS_Alg} and the iteration scheme \eqref{iteration-II}, respectively. Suppose that both sequences share the same initial point $x^0$ and sketching matrices $\{S_k\}_{k \geq 0}$. Then, for any $k \geq 0$, it holds that
	\begin{equation} \nonumber
		x^{k+1} = h^{k} + \delta_{k} q^{k}.
	\end{equation}
\end{lemma}
\begin{proof}
	According to Lemma \ref{Lem_Eq1}, the sequence $\{x^k\}_{k \geq 0}$ generated by Algorithm \ref{RCGLS_Alg} coincides with that generated by the iteration scheme \eqref{iteration-I}. Therefore, it suffices to establish the structural identity using the variables from \eqref{iteration-I}. 
	
	Set $\theta_{-1}=\frac{1}{\|d_{1}^{0}\|_{2}^{2}}$ if $d_{1}^{0} \neq 0$, and $\theta_{-1}=1$ otherwise. Let the sequences $\{\eta_{k}, \theta_{k}, x^{k}, r^{k}, q^{k}, q_{*}^{k}\}_{k \geq 0}$ be generated by the iteration scheme \eqref{iteration-I}. To map these to the scheme \eqref{iteration-II}, we set the initial variables as $\widehat{\delta}_{-1}^{*}=0$, $\widehat{d}_{1}^{0}=S_{0}^{\top}A^{\top}(b-Ax^{0})$, and $(\widehat{h}^{0},\widehat{h}_{*}^{0})=(x^{0}, Ax^{0})$. For $k \geq 0$, we construct the following recursive process:
	\begin{equation}\nonumber
		\begin{cases}
			\widehat{\delta}_{k} =\widehat{\delta}_{k-1}^{*} + \eta_k, \\
			\widehat{d}_{1}^{k+1}=S_{k+1}^{\top}A^{\top}r^{k+1}, \widehat{d}^{k+1}=S_{k+1}\widehat{d}_{1}^{k+1}, \widehat{d}_{2}^{k+1}=AS_{k+1}\widehat{d}_{1}^{k+1}, \\
			\text{If} \; \langle \widehat{d}_{2}^{k+1}, q_{*}^{k} \rangle \neq 0, \; \text{then} \;
			\widehat{\delta}_{k}^{*}=\widehat{\delta}_{k}, (\widehat{h}^{k+1}, \widehat{h}_{*}^{k+1})=(\widehat{h}^{k}-\widehat{\delta}_{k} \theta_{k} \widehat{d}^{k+1}, \widehat{h}_{*}^{k}-\widehat{\delta}_{k} \theta_{k} \widehat{d}_{2}^{k+1}). \\
			\text{If} \; \langle \widehat{d}_{2}^{k+1}, q_{*}^{k} \rangle = 0,  \; \text{then} \; \widehat{\delta}_{k}^{*}=0, (\widehat{h}^{k+1}, \widehat{h}_{*}^{k+1})=(\widehat{h}^{k} + \widehat{\delta}_{k} q^{k}, \widehat{h}_{*}^{k} + \widehat{\delta}_{k} q_{*}^{k}),
		\end{cases}
	\end{equation}
	where it maintains $\widehat{h}_{*}^{k+1}=A\widehat{h}^{k+1}$.
	
	We now show the identity $x^{k+1}=\widehat{h}^{k} + \widehat{\delta}_{k} q^{k}$ for $k \geq 0$ by induction. For the base case $k=0$, since $(\widehat{h}^{0}, \widehat{\delta}_{0})=(x^{0}, \eta_{0})$, we have $\widehat{h}^{0} + \widehat{\delta}_{0} q^{0}=x^{0}+\eta_{0} q^{0}=x^{1}$ from \eqref{iteration-I}. Assume by induction that $x^{j+1}=\widehat{h}^{j} + \widehat{\delta}_{j} q^{j}$ holds for all $j \leq k$ and some $k \geq 0$. We show that it also holds at step $k+1$, by considering two cases based on the values of $\langle \widehat{d}_{2}^{k+1}, q_{*}^{k} \rangle$:
	
	\textbf{Case 1.} If $\langle \widehat{d}_{2}^{k+1}, q_{*}^{k} \rangle \neq 0$, it follows from the construction and the update rules in \eqref{iteration-I} that
	\begin{align*}
		\widehat{h}^{k+1} + \widehat{\delta}_{k+1} q^{k+1}=&(\widehat{h}^{k}-\widehat{\delta}_{k} \theta_{k} \widehat{d}^{k+1})+(\widehat{\delta}_{k}^{*} + \eta_{k+1})q^{k+1} \\
		=&(\widehat{h}^{k}-\widehat{\delta}_{k} \theta_{k} \widehat{d}^{k+1})+\widehat{\delta}_{k}(q^{k}+\theta_{k}\widehat{d}^{k+1})+ \eta_{k+1}q^{k+1} \\
		=& (\widehat{h}^{k}+\widehat{\delta}_{k}q^{k})+\eta_{k+1}q^{k+1} \\
		=&x^{k+1}+\eta_{k+1}q^{k+1}=x^{k+2}.
	\end{align*}
	
	\textbf{Case 2.} If $\langle \widehat{d}_{2}^{k+1}, q_{*}^{k} \rangle = 0$, we have $\widehat{\delta}_{k}^{*}=0$ and $\widehat{h}^{k+1} = \widehat{h}^{k} + \widehat{\delta}_{k} q^{k} = x^{k+1}$. Thus,
	$$
	\widehat{h}^{k+1} + \widehat{\delta}_{k+1} q^{k+1}=x^{k+1}+(\widehat{\delta}_{k}^{*} + \eta_{k+1})q^{k+1}=x^{k+1}+\eta_{k+1}q^{k+1}=x^{k+2}.
	$$
	Therefore, by induction, $x^{k+1}=\widehat{h}^{k} + \widehat{\delta}_{k} q^{k}$ for $k \geq 0$. Using this identity, the variable $r^{k+1}$ in \eqref{iteration-I} can be expressed as $r^{k+1}=b-Ax^{k+1}=b-A(\widehat{h}^{k} + \widehat{\delta}_{k} q^{k})=b-\widehat{h}_{*}^{k}-\widehat{\delta}_{k} q_{*}^{k}$. Substituting this into the definition $\widehat{d}_{1}^{k+1}=S_{k+1}^{\top}A^{\top}r^{k+1}$ yields
	$$
	\widehat{d}_{1}^{k+1}=S_{k+1}^{\top}A^{\top}b-S_{k+1}^{\top}A^{\top}\widehat{h}_{*}^{k}-\widehat{\delta}_{k}S_{k+1}^{\top}A^{\top}q_{*}^{k}.
	$$
	With this expression, we consider the composite system of sequences
	$$
	\mathcal{T} := \{\widehat{\delta}_{k}, \widehat{\delta}_{k}^{*}, \widehat{d}_1^{k+1}, \widehat{d}^{k+1}, \widehat{d}_{2}^{k+1}, \widehat{h}^{k+1}, \widehat{h}_*^{k+1}, \eta_k, \theta_k, q^{k+1}, q_*^{k+1}\}_{k \geq 0},
	$$
	where the latter four sequences are provided by \eqref{iteration-I}. One can observe that the elements of $\mathcal{T}$ satisfy the same recursive rules as prescribed by the iteration scheme \eqref{iteration-II}. In addition, the initial conditions of $\mathcal{T}$ are the same as those defined in \eqref{iteration-II}. Consequently, the composite system $\mathcal{T}$ coincides with the sequence system generated by \eqref{iteration-II}. As a result, for the sequences $\{h^k, q^k, \delta_k\}_{k \geq 0}$ generated by the iteration scheme \eqref{iteration-II}, we have $x^{k+1} = h^{k} + \delta_{k} q^{k}$ for all $k \geq 0$.
	This completes the proof of the lemma.
\end{proof}

Now, we are ready to prove Proposition \ref{Theo}.
\begin{proof}[Proof of Proposition \ref{Theo}]
	Since Algorithms \ref{RCGLS_Alg} and \ref{RCGLS_EI_Q} share the same sampling matrices $\{S_k\}_{k\geq0}$ and initial point $x^{0}$, it follows from Lemma \ref{Lem_Eq2} that to establish the identity $x^{k+1} = h^{k} + \delta_{k} q^{k}$ for all $k \geq 0$, it suffices to show that the sequences $\{l_{k}, q_{*}^{k}\}_{k \geq 0}$ generated by Algorithm \ref{RCGLS_EI_Q} satisfy $l_{k}=\|q_{*}^{k}\|_{2}^{2}$.
	
	We proceed by induction on $k$. For the base case $k=0$, the identity $l_{0}=\|q_{*}^{0}\|_{2}^{2}$ holds by definition. Now assume that $l_{k}=\|q_{*}^{k}\|_{2}^{2}$ for some $k \geq 0$. We consider the update for $l_{k+1}$:
	
	\textbf{Case 1.} If $\langle \widehat{d}_{2}^{k+1}, q_{*}^{k} \rangle \neq 0$, then
	$$
	\|q_{*}^{k+1}\|_{2}^{2}=\|q_{*}^{k}+\theta_{k} d_{2}^{k+1}\|_{2}^{2}=l_{k}+\theta_{k}^{2} \|d_{2}^{k+1}\|_{2}^{2}+2\theta_{k} \langle d_{2}^{k+1}, q_{*}^{k} \rangle=\theta_{k}^{2} \|d_{2}^{k+1} \|_{2}^{2}-l_{k}=l_{k+1},
	$$
	where the third equality follows from the definition of $\theta_{k}$.
	
	\textbf{Case 2.} If $\langle \widehat{d}_{2}^{k+1}, q_{*}^{k} \rangle = 0$, then
	$$
	\|q_{*}^{k+1}\|_{2}^{2}=\|d_{2}^{k+1}\|_{2}^{2}=l_{k+1}.
	$$
	Hence, by induction, the identity $l_{k}=\|q_{*}^{k}\|_{2}^{2}$ holds for all $k \geq 0$.
	This completes the proof of the proposition.
\end{proof}

		\section{Application to ridge regression}
					
					In this section, we apply the proposed RCGLS method to solve the ridge regression problem, also known as the $\ell_2$-regularized least-squares problem
					\begin{equation}\label{eq:ridge_prob}
						\min_{x \in \mathbb{R}^d} \frac{1}{2} \|\bar{A}x - \bar{b}\|_2^2 + \frac{\lambda}{2} \|x\|_2^2,
					\end{equation}
					where $\bar{A} \in \mathbb{R}^{n \times d}$, $\bar{b} \in \mathbb{R}^n$, and the regularization parameter $\lambda > 0$. It is well-known that \eqref{eq:ridge_prob} can be equivalently reformulated as an augmented linear system \cite{ivanov2013kaczmarz, hefny2017rows}
					\begin{equation} \label{Aug}
						\underbrace{
							\begin{bmatrix}
								\sqrt{\lambda}I_{n} & \bar{A} \\
								\bar{A}^\top & -\sqrt{\lambda}I_{d}
						\end{bmatrix}}_{\textstyle \widehat{A}}
						\underbrace{
							\begin{bmatrix}
								y \\
								x
						\end{bmatrix}}_{\textstyle \widehat{x} \vphantom{\widehat{A}}}
						= 
						\underbrace{
							\begin{bmatrix}
								\bar{b} \\
								0
						\end{bmatrix}}_{\textstyle \widehat{b} \vphantom{\widehat{A}}}.
					\end{equation}
				We then construct the corresponding least-squares formulation for the augmented system \eqref{Aug}
					\begin{equation} \label{AugRR_LS}
						\min\limits_{\widehat{x} \in \mathbb{R}^{n+d}} \frac{1}{2} \|\widehat{A} \widehat{x} - \widehat{b}\|_{2}^{2}.
					\end{equation}
					
					Let $U=[\sqrt{\lambda}I_{n} \; \bar{A}]^\top$ and $V=[\bar{A}^{\top} \; -\sqrt{\lambda}I_{d}]^{\top}$ denote the column blocks of $\widehat{A}$. One can verify that the range spaces of $U$ and $V$ are mutually orthogonal, since
					\[
					V^{\top}U = [\bar{A}^{\top} \; -\sqrt{\lambda}I_{d}] [\sqrt{\lambda}I_{n} \; \bar{A}]^\top = \sqrt{\lambda}\bar{A}^{\top} - \sqrt{\lambda}\bar{A}^{\top} = 0.
					\]
					Using this orthogonality property, the objective function in \eqref{AugRR_LS} admits a separable decomposition 
					\[ 
						\begin{aligned}
							\frac{1}{2} \|\widehat{A} \widehat{x} - \widehat{b}\|_{2}^{2} &= \frac{1}{2} \left\| Vx+Uy - \widehat{b} \right\|_{2}^{2} = \frac{1}{2} \|Vx - \widehat{b}\|_{2}^{2} + \langle Vx - \widehat{b}, Uy \rangle + \frac{1}{2} \|Uy\|_{2}^{2} \\
							&= \frac{1}{2} \|Vx - \widehat{b}\|_{2}^{2} + \frac{1}{2} \|Uy - \widehat{b}\|_{2}^{2} - \frac{1}{2} \|\widehat{b}\|_{2}^{2},
						\end{aligned}
					\]
					where the last equality follows from $\langle Vx, Uy \rangle=x^{\top}(V^{\top}U)y=0$.  This equivalence implies that solving the original ridge regression problem \eqref{eq:ridge_prob} reduces to solving the separable optimization problem
					\begin{equation} \label{AugRR_LS-1}
						\min\limits_{(x,y) \in \mathbb{R}^{d}\times \mathbb{R}^{n}} \frac{1}{2} \|Vx - \widehat{b}\|_{2}^{2} + \frac{1}{2} \|Uy - \widehat{b}\|_{2}^{2}
					\end{equation}
					 with respect to either $x$ or $y $ individually. Specifically, we can either solve
					 \begin{equation} \label{RR_ELSx}
					 	\min_{x \in \mathbb{R}^{d}} \frac{1}{2} \|Vx - \widehat{b}\|_{2}^{2}
					 \end{equation}
					 to directly obtain the desired variable $x$, or solve
					 \begin{equation} \label{RR_ELSy}
					 	\min_{y \in \mathbb{R}^{n}} \frac{1}{2} \|Uy - \widehat{b}\|_{2}^{2}
					 \end{equation}
					 for the auxiliary variable $y$ and then recover $x$ via the structural relation $x = \bar{A}^{\top}y / \sqrt{\lambda}$, which is derived from the second block row of \eqref{Aug}.

					 The proposed RCGLS method can be directly applied to the above least squares problems with inputs $V$ (or $U$) and $\widehat{b}$. More importantly, the inherent block structure of $V$ and $U$ enables efficient structured implementation.
					 To elaborate, we first consider  \eqref{RR_ELSx}.  At each iteration, the augmented residual $\widehat{r}^k := \widehat{b} - Vx^k$ can be explicitly expressed as
					 \[
					 \widehat{r}^k = \left[ (\bar{b} - \bar{A}x^k)^\top, \, (\sqrt{\lambda}x^k)^\top \right]^\top.
					 \]
					 This structural form enables the update of the high-dimensional residual $\widehat{r}^k \in \mathbb{R}^{n+d}$ to rely only on the low-dimensional partial residual $g^{k}:=\bar{b} - \bar{A}x^k \in \mathbb{R}^{n}$. Accordingly, the core auxiliary vector $w^{k}:=S_{k}^{\top}V^{\top} \widehat{r}^k$ simplifies to $w^{k}=S_{k}^{\top}\bar{A}^{\top}g^{k}-\lambda S_{k}^{\top}x^{k}$. Similarly, the computation of $Vp^{k} \in \mathbb{R}^{n+d}$ is reduced to evaluating its low-dimensional component $u^k := \bar{A}p^k$.
					Analogous structural simplifications apply to \eqref{RR_ELSy} by exploiting the block structure of $U$.

					 The selection between  \eqref{RR_ELSx} and \eqref{RR_ELSy} can depend on the availability of row or column information.  For instance, suppose that $S_k = e_{i_k}$, where  $e_{i_k} \in \mathbb{R}^n$ or $e_{i_k} \in \mathbb{R}^d$ is determined by the problem dimension. In this case, $S_{k}^{\top}\bar{A}^{\top}$ exploits the column information of $\bar{A}$, while $S_{k}^{\top}\bar{A}$ utilizes the row information of $\bar{A}$.
					In addition, the selection can also depend on the data dimensions. Note that $V \in \mathbb{R}^{(n+d) \times d}$ and $U \in \mathbb{R}^{(n+d) \times n}$, \eqref{RR_ELSx} is preferable for $n \geq d$ due to its fewer column dimensions, while \eqref{RR_ELSy} is adopted for $n < d$.
					 
					 Now, we have already applied the RCGLS method to the ridge regression problem \eqref{eq:ridge_prob} described in Algorithm \ref{RCGLS-RR}.  We note that the efficient implementation strategy presented in Algorithm \ref{RCGLS_EI_Q} is also applicable to ridge regression.

				\begin{algorithm}[htpb]
					\caption{RCGLS for solving ridge regression (RidgeRCGLS) \label{RCGLS-RR}}
					\begin{algorithmic}
						\Require
						$\bar{A} \in \mathbb{R}^{n \times d}$, $\bar{b} \in \mathbb{R}^n$, and $\lambda >0$.
						
							\noindent If only the column information is available or $n \geq d$
							
							Given distribution $\mathcal{D}$, initialize $x^0 \in \mathbb{R}^d$, and update $x^k$ via Option I.
							
							\noindent If only the row information is available or $n < d$
							
							Given distribution $\mathcal{D}$, initialize $y^0 \in \mathbb{R}^n$ and set $x^{0} = \bar{A}^{\top}y^{0}/\sqrt{\lambda}$, and update $x^k$ via Option II.
						
						\Ensure
						The approximate solution $x^{k}$.
					\end{algorithmic}
				\end{algorithm}

		\begin{table}[htpb]
			\centering
			{
				\begin{tabular}{ |l| }
					\hline
					\qquad \qquad \qquad \qquad \qquad \qquad \qquad \textbf{Option I}  \qquad \qquad \qquad \qquad \qquad \qquad \qquad\\
					1: Randomly select a sampling matrix $S_{0}$ from $\mathcal{D}$. \\
					2:  Set $g^{0}=\bar{b}-\bar{A}x^{0}$, $w^{0}= S_{0}^{\top}\bar{A}^{\top}g^{0}-\lambda S_{0}^{\top}x^{0}$, $p^{0}=S_{0}w^{0}$, $u^{0}=\bar{A}S_{0}w^{0}$, \\
                    \quad \, and $\varsigma_{0}=\|u^{0}\|_{2}^{2}+\lambda \|p^{0}\|_{2}^{2}$. \\
					3: For $k=0, 1, \ldots$ until the stopping rule is satisfied: \\
					4: Set $\mu_{k}=\frac{\|w^{k}\|_{2}^{2}}{\varsigma_{k}}$. \\
					5: Update $x^{k+1}=x^{k}+\mu_{k} p^{k}$ and $g^{k+1}=g^{k}-\mu_{k} u^{k}$. \\
					6: Randomly select a sampling matrix $S_{k+1}$ from $\mathcal{D}$. \\
					7: Compute $w^{k+1}= S_{k+1}^{\top}\bar{A}^{\top}g^{k+1}-\lambda S_{k+1}^{\top}x^{k+1}$, \\
					\qquad \qquad \qquad \qquad\;\;\, $\tau_k=-\frac{\langle \bar{A}S_{k+1}w^{k+1}, u^{k} \rangle+\lambda \langle S_{k+1}w^{k+1}, p^{k} \rangle}{\varsigma_{k}}$, \\
					\qquad \qquad \qquad \qquad $p^{k+1}= S_{k+1}w^{k+1}+\tau_k p^{k}$, \\
					\qquad \qquad \qquad \qquad $u^{k+1}= \bar{A}S_{k+1}w^{k+1}+\tau_k u^{k}$, \\
                    \qquad \qquad \qquad \quad \;\;\;\, $\varsigma_{k+1}= -\tau_{k}^{2} \varsigma_{k}+\lambda \|S_{k+1}w^{k+1}\|_{2}^{2}+\|\bar{A}S_{k+1}w^{k+1}\|_{2}^{2}$. \\					\hline
				\end{tabular}
			}
		\end{table}
		
		\begin{table}[htpb]
			\centering
			{
				\begin{tabular}{ |l| }
					\hline
					\qquad \qquad \qquad \qquad \qquad \qquad \qquad \textbf{Option II}  \qquad \qquad \qquad \qquad \qquad \qquad \qquad\\
					1: Randomly select a sampling matrix $S_{0} $ from $\mathcal{D}$. \\
					2:  Set $w^{0}=\sqrt{\lambda}S_{0}^{\top}\bar{b}-\sqrt{\lambda}S_{0}^{\top}\bar{A}x^{0}-\lambda S_{0}^{\top} y^{0}$, $p^{0}=S_{0}w^{0}$, $u^{0}=\bar{A}^{\top}S_{0}w^{0}$, \\
                    \quad \, and $\varsigma_{0}=\|u^{0}\|_{2}^{2}+\lambda \|p^{0}\|_{2}^{2}$. \\
                    3: For $k=0, 1, \ldots$ until the stopping rule is satisfied: \\
					4: Set $\mu_{k}=\frac{\|w^{k}\|_{2}^{2}}{\varsigma_{k}}$. \\
					5: Update $y^{k+1}=y^{k}+\mu_{k} p^{k}$ and $x^{k+1}=x^{k}+\frac{\mu_{k}}{\sqrt{\lambda}} u^{k}$. \\
					6: Randomly select a sampling matrix $S_{k+1} $ from $\mathcal{D}$. \\
					7: Compute $w^{k+1}= \sqrt{\lambda} S_{k+1}^{\top}\bar{b}-\sqrt{\lambda} S_{k+1}^{\top}\bar{A}x^{k+1}-\lambda S_{k+1}^{\top} y^{k+1}$, \\
					\qquad \qquad \qquad \qquad\;\;\, $\tau_k=-\frac{\langle \bar{A}^{\top}S_{k+1}w^{k+1}, u^{k} \rangle+\lambda \langle S_{k+1}w^{k+1}, p^{k} \rangle}{\varsigma_{k}}$, \\
					\qquad \qquad \qquad \qquad $p^{k+1}= S_{k+1}w^{k+1}+\tau_k p^{k}$, \\
					\qquad \qquad \qquad \qquad $u^{k+1}= \bar{A}^{\top}S_{k+1}w^{k+1}+\tau_k u^{k}$, \\
					\qquad \qquad \qquad \quad \;\;\;\, $\varsigma_{k+1}= -\tau_{k}^{2} \varsigma_{k}+\lambda \|S_{k+1}w^{k+1}\|_{2}^{2}+\|\bar{A}^\top S_{k+1}w^{k+1}\|_{2}^{2}$. \\					\hline
				\end{tabular}
			}
		\end{table}
\begin{remark}
			The block-orthogonal structure of the coefficient matrix $\widehat{A}$ in \eqref{Aug} provides a theoretical explanation for the ``wasted iterations'' identified by Hefny et al. \cite{hefny2017rows} in the Ivanov--Zhdanov (IZ) method. Specifically, the IZ method essentially applies  the RK method to the augmented system $\widehat{A}\widehat{x} = \widehat{b}$. At each iteration, a row index $i_k$ is sampled randomly to update the iterate via
			\begin{equation} \nonumber
				\widehat{x}^{k+1} = \widehat{x}^k + \frac{\widehat{b}_{i_k} - \widehat{A}_{i_k,:}\widehat{x}^k}{\|\widehat{A}_{i_k,:}\|_2^2} \widehat{A}_{i_k,:}^\top.
			\end{equation}
			To analyze the behavior of the process, we consider the partitioned form of $\widehat{A}\widehat{x} = \widehat{b}$:
			\[
				\begin{bmatrix}
					U^\top \\
					V^\top
				\end{bmatrix} \widehat{x} =
				\begin{bmatrix}
					\bar{b} \\
					0
				\end{bmatrix}.
			\]
			Suppose the current iterate $\widehat{x}^k$ satisfies the first block of the system, i.e., $U^\top \widehat{x}^k = \bar{b}$. This implies that $\widehat{x}^k$ lies in the affine solution set $\mathcal{X}_{*,1} := \{z \mid U^\top z = \bar{b}\}$, which can be expressed as $\mathcal{X}_{*,1} = \widehat{x}^k + \operatorname{Null}(U^\top)$. In this case, if the index $i_k$ is selected from the first block of rows, the equality $U^\top \widehat{x}^k = \bar{b}$ implies that $\widehat{b}_{i_k} - \widehat{A}_{i_k,:} \widehat{x}^k = 0$, leading to a stagnant update $\widehat{x}^{k+1} = \widehat{x}^k$. Conversely, if $i_k$ is selected from the second block of rows, the update direction $\widehat{A}_{i_k,:}^\top$ lies in $\operatorname{Range}(V)$. Since $V^\top U = 0$, it follows that $\operatorname{Range}(V) \subseteq \operatorname{Null}(U^\top)$, which implies the update remains within the affine set $\mathcal{X}_{*,1}$
			\[
				\widehat{x}^{k+1} \in \widehat{x}^k + \operatorname{Range}(V) \subseteq \widehat{x}^k + \operatorname{Null}(U^\top) = \mathcal{X}_{*,1}.
			\]
			Consequently, once the first block of equations is satisfied, this state is preserved for all future iterations, rendering any subsequent sampling from this block a ``wasted'' step with a zero step-size.
		\end{remark}
		
		
		\begin{remark}
			RidgeRCGLS can be viewed as an accelerated version of the RGS and RK variants proposed in \cite{hefny2017rows} for the case of single-coordinate sampling. Specifically, these variants can be regarded as applying the standard RCD method with exact line search to the least squares problems \eqref{RR_ELSx} and \eqref{RR_ELSy}, respectively. Under this interpretation, the resulting iterates $x$ match those in \cite{hefny2017rows}, while the iterates $y$ are identical up to a scaling factor $\sqrt{\lambda}$. Accordingly, RidgeRCGLS achieves a convergence upper bound at least that of the corresponding RK and RGS variants, as stated in Remark \ref{rem:comparison_and_acceleration}.
		\end{remark}

			\section{Numerical experiments}
		
		
		In this section, we present preliminary numerical results for the proposed RidgeRCGLS method. We compare our method with the RidgeGRCD method and the RidgeSketch method with heuristic increasing momentum (HImRidgeSketch) proposed in \cite[Section 10]{gazagnadou2022ridgesketch}. We note that the RidgeGRCD method is essentially the application of the GRCD method introduced in Remark \ref{rem:comparison_and_acceleration} to ridge regression problems, and it can also be interpreted as the block variant of the improved RK and RGS method developed in \cite{hefny2017rows}. All the methods are implemented in MATLAB R2022b for macOS Monterey on a MacBook Air with Apple M2 CPU and 16 GB memory. The code to reproduce our results can be found at  \href{https://github.com/xiejx-math/RCGLS}{https://github.com/xiejx-math/RCGLS}.
		%
		%
		
		\subsection{Numerical setup}
		
		We consider two types of matrices, i.e., synthetic and real-world datasets.  For synthetic datasets, we follow the test framework introduced in \cite{ozaslan2023m} and generate the coefficient matrix $\bar{A} \in \mathbb{R}^{n \times d}$  under three settings:  overdetermined ($n > d$), square ($n = d$), and underdetermined ($n < d$). To construct problems involving both column correlation and ill-conditioning, we first sample a base matrix from a multivariate Gaussian distribution $\mathcal{N}(\mathbf{1}_{d}, \Gamma)$, where $\mathbf{1}_d \in \mathbb{R}^d$ denotes an all-ones vector and the covariance matrix entries are defined as $\Gamma_{ij} = 5 \cdot 0.7^{|i-j|}$. Subsequently, utilizing the singular value decomposition, we replace its singular values with those derived from the \texttt{phillips} test problem in Regularization Tools (RegTools) \cite{hansen1994regularization}, and scale them to set the condition number $\kappa(\bar{A})$ to $10^4$. Adopting the  reference signal $\bar{x}$ provided by RegTools, we generate the observation vector as $\bar{b} = \bar{A}\bar{x} + \bar{b}_{e}$, where the additive Gaussian noise $\bar{b}_{e}$ is scaled to achieve a relative noise level $\|\bar{b}_{e}\|_2 / \|\bar{A}\bar{x}\|_2 = 0.1$.  The real-world datasets are obtained from LIBSVM \cite{chang2011libsvm}, and their corresponding matrices $\bar{A}$ and vectors $\bar{b}$ are used directly. 
		

		For the underlying sampling strategy, we adopt the following uniform sampling scheme. We use $q$ to denote the block size and let $m := \min\{n, d\}$.
		 At every step, we uniformly sample $q$ unique indices to form the index set $\mathcal{J} \subseteq [m]$ with cardinality $|\mathcal{J}|=q$. Since there are $\binom{m}{q}$ possible choices for $\mathcal{J}$, the probability of selecting any particular set is $\text{Prob}(\mathcal{J}) = 1/\binom{m}{q}$. The corresponding randomized sketching matrix is constructed  as $S = I_{:, \mathcal{J}}\in\mathbb{R}^{m\times q}$. Under this sampling scheme, we refer to the specific implementations of RidgeRCGLS and RidgeGRCD as RidgeRCGLSU and RidgeGRCDU, respectively. The HImRidgeSketch method also employs uniform sampling (HImRidgeSketchU), and all parameter configurations follow the settings in \cite [Section 10]{gazagnadou2022ridgesketch}.

		All algorithms are initialized from zero. Specifically, we set $x^{0} = 0$ for $n \geq d$. For $n < d$, we initialize $y^0 = 0$, which yields $x^0 = \bar{A}^\top y^0 / \sqrt{\lambda} = 0$. The computations are terminated once the relative solution error (RSE), defined as RSE $=\|x^{k}-x^{*}\|_{2}^{2}/ \|x^{0}-x^{*}\|_{2}^{2}$, is less than a specific error tolerance.  In practice, we consider a variable as zero when it is less than \texttt{eps}. For each experiment, we run $10$ independent trials. In addition, both row and column information of $\bar{A}$ is accessible during implementation, enabling all algorithms to adaptively select the appropriate option according to the relative sizes of $n$ and $d$.
		

		\subsection{Acceleration efficiency and the impact of block size $q$}
		
		In this subsection, we evaluate the acceleration efficiency of the proposed RidgeRCGLSU method and investigate the impact of block size $q$ on its convergence performance using synthetic datasets. In particular, we compare RidgeRCGLSU with the unaccelerated baseline RidgeGRCDU  to demonstrate the acceleration gains. We set the regularization parameter $\lambda$ to $0.05$ and $0.005$ for all corresponding tests. 
		  
		Figures \ref{figureR3} and \ref{figureR4} report the computational CPU time and the number of epochs required by each method. We define the epoch as $(k \cdot \frac{q}{m})$, which guarantees consistent computational overhead for a full data pass across all compared methods. 
		The bold line represents the median computed over $10$ independent runs. The lightly shaded area signifies the range from the minimum to the maximum values, while the darker shaded one indicates the data lying between the $25$th and $75$th quantiles.

		
		
		\begin{figure}[hptb]
			\centering
			\begin{tabular}{cc}
				\includegraphics[width=0.33\linewidth]{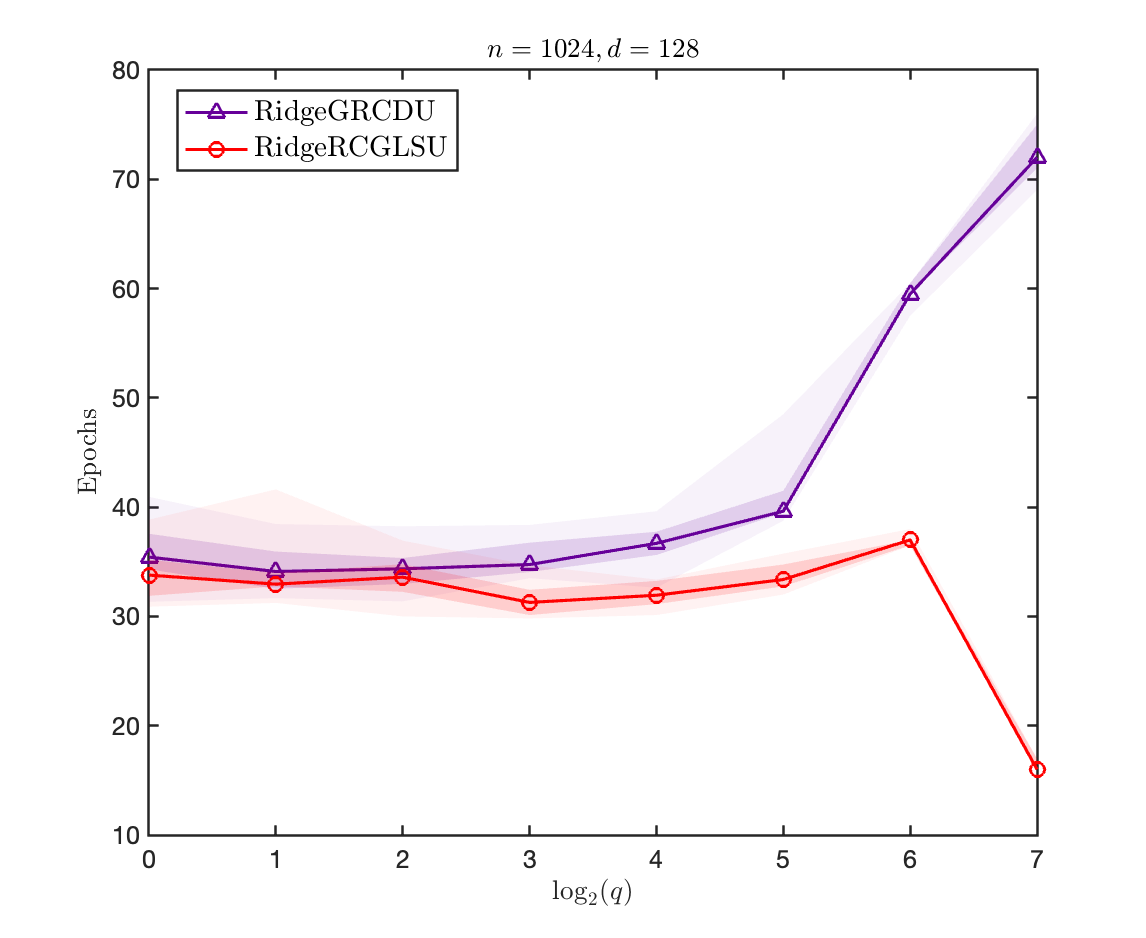}
				\includegraphics[width=0.33\linewidth]{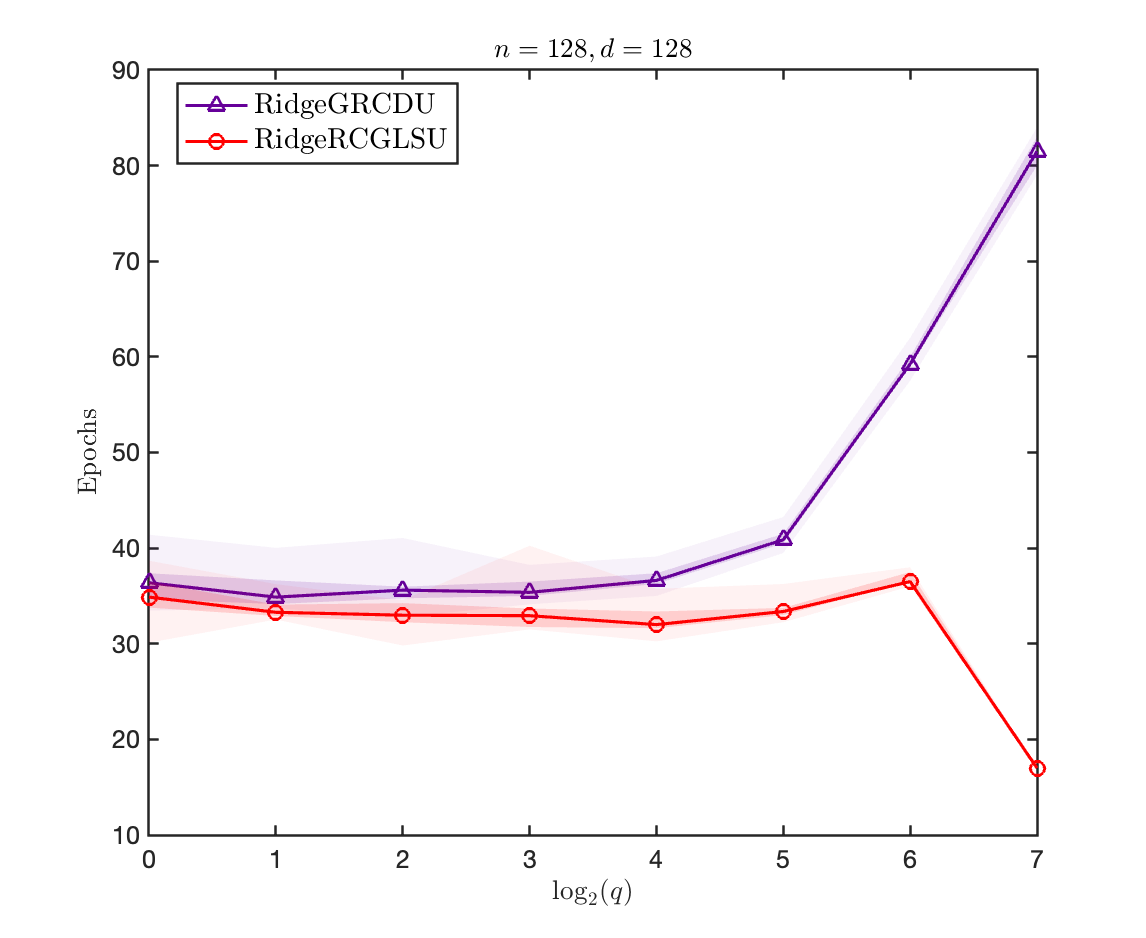}
				\includegraphics[width=0.33\linewidth]{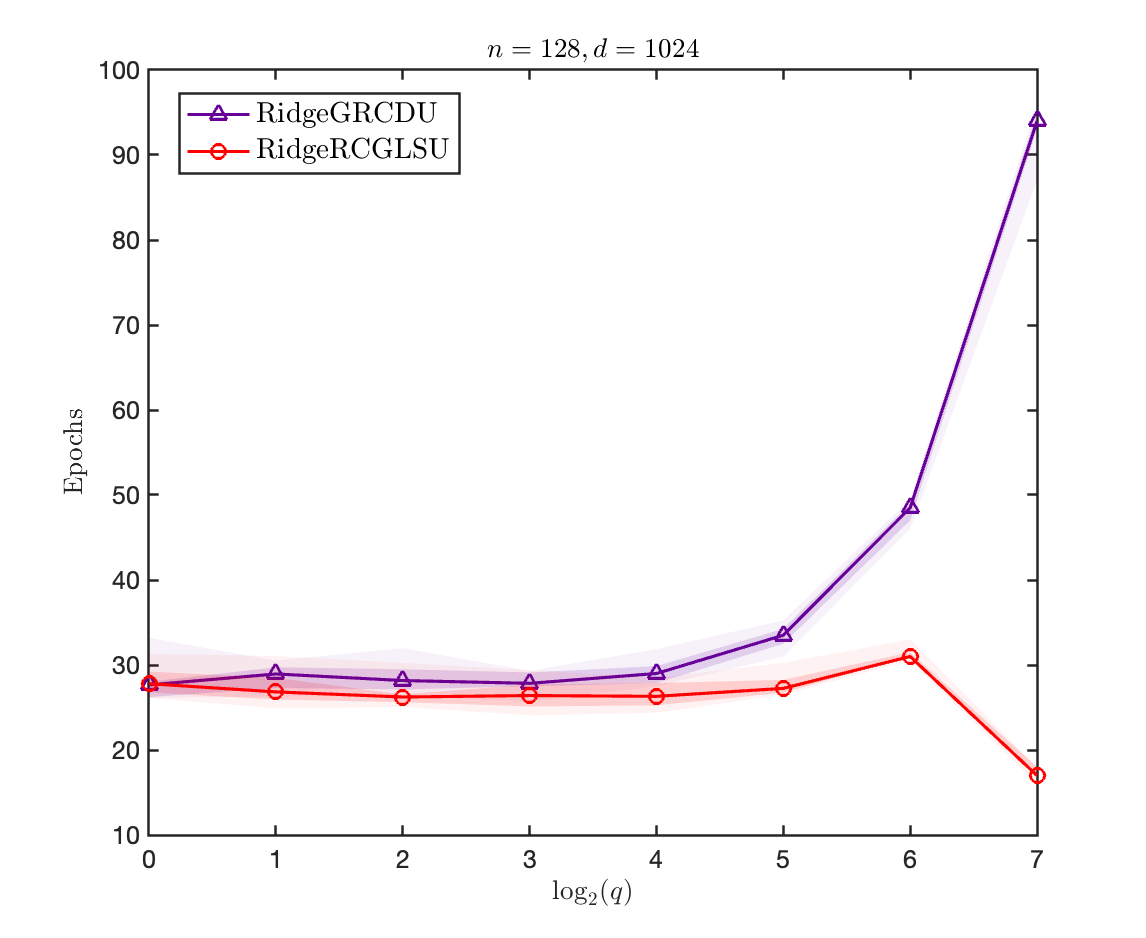}\\
				\includegraphics[width=0.33\linewidth]{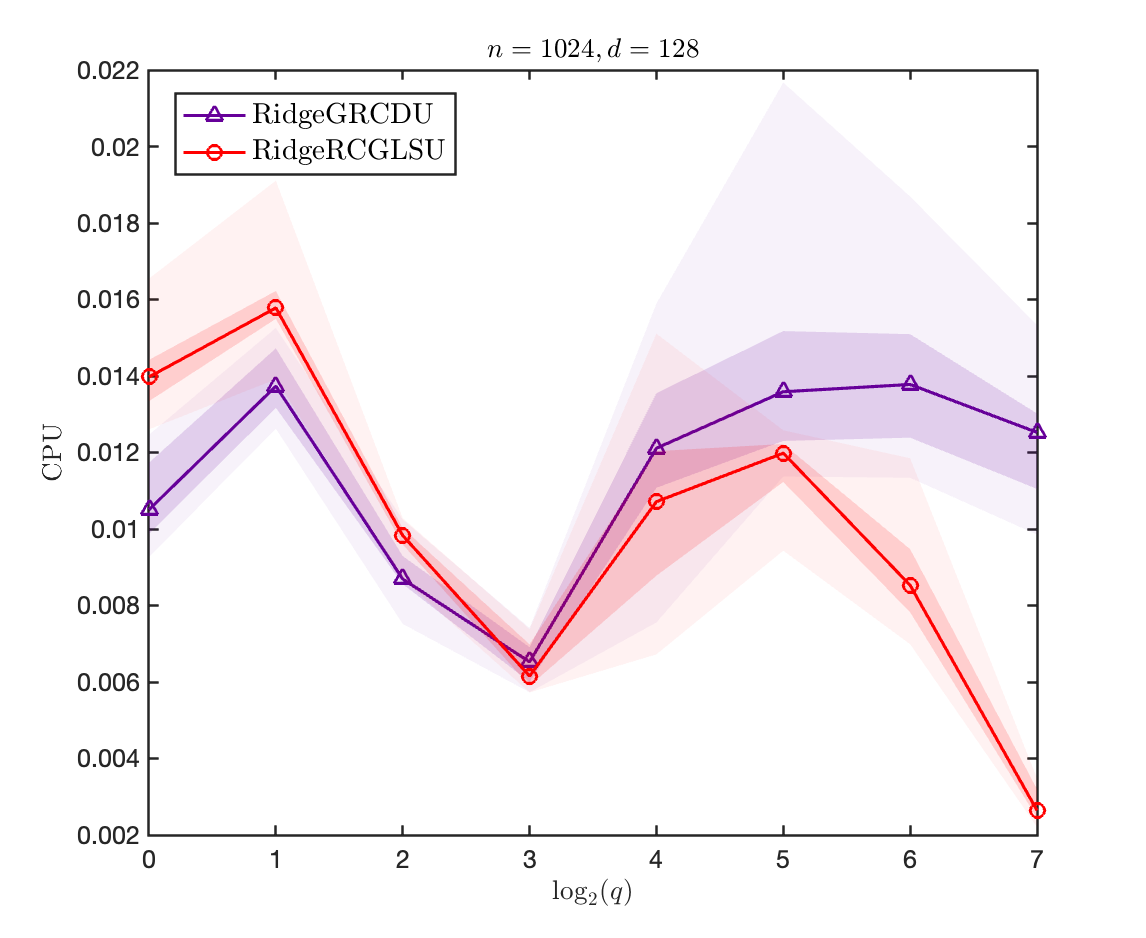}
				\includegraphics[width=0.33\linewidth]{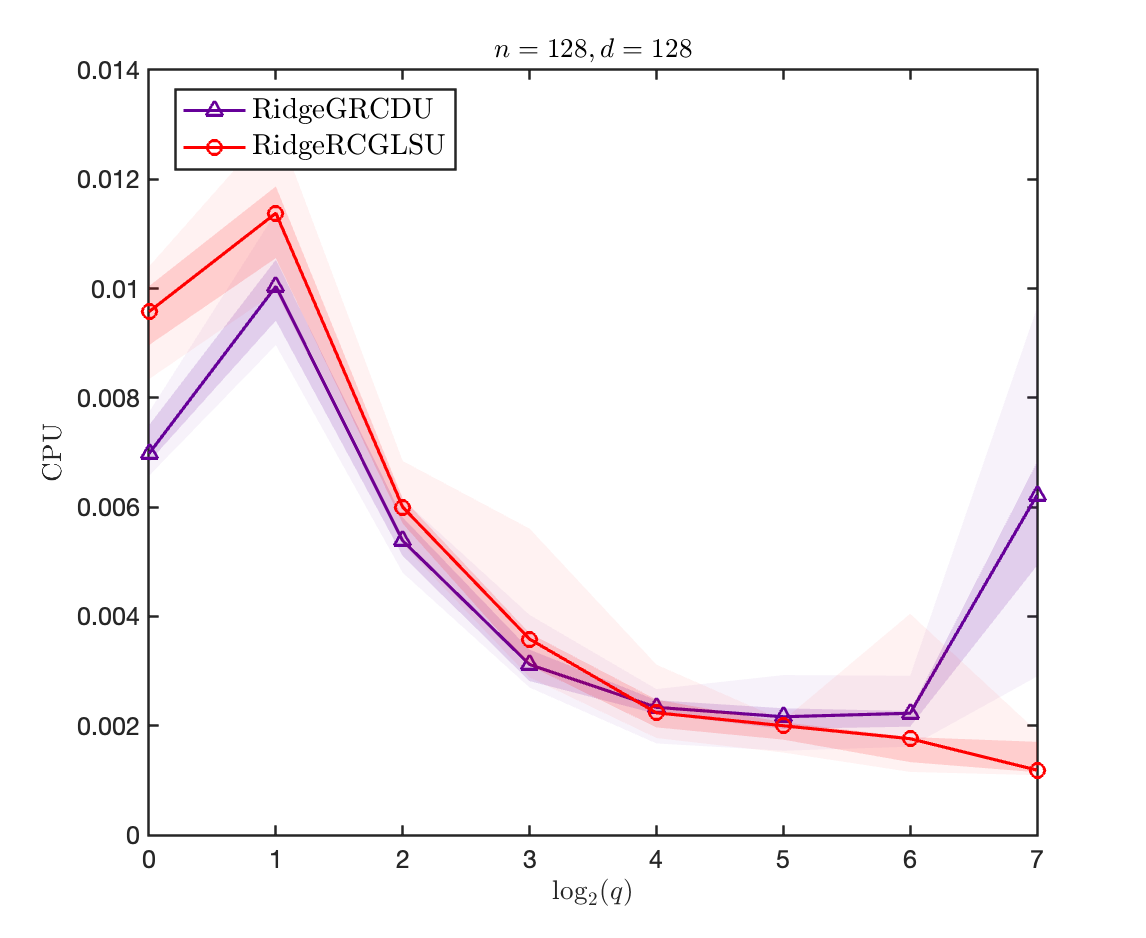}
				\includegraphics[width=0.33\linewidth]{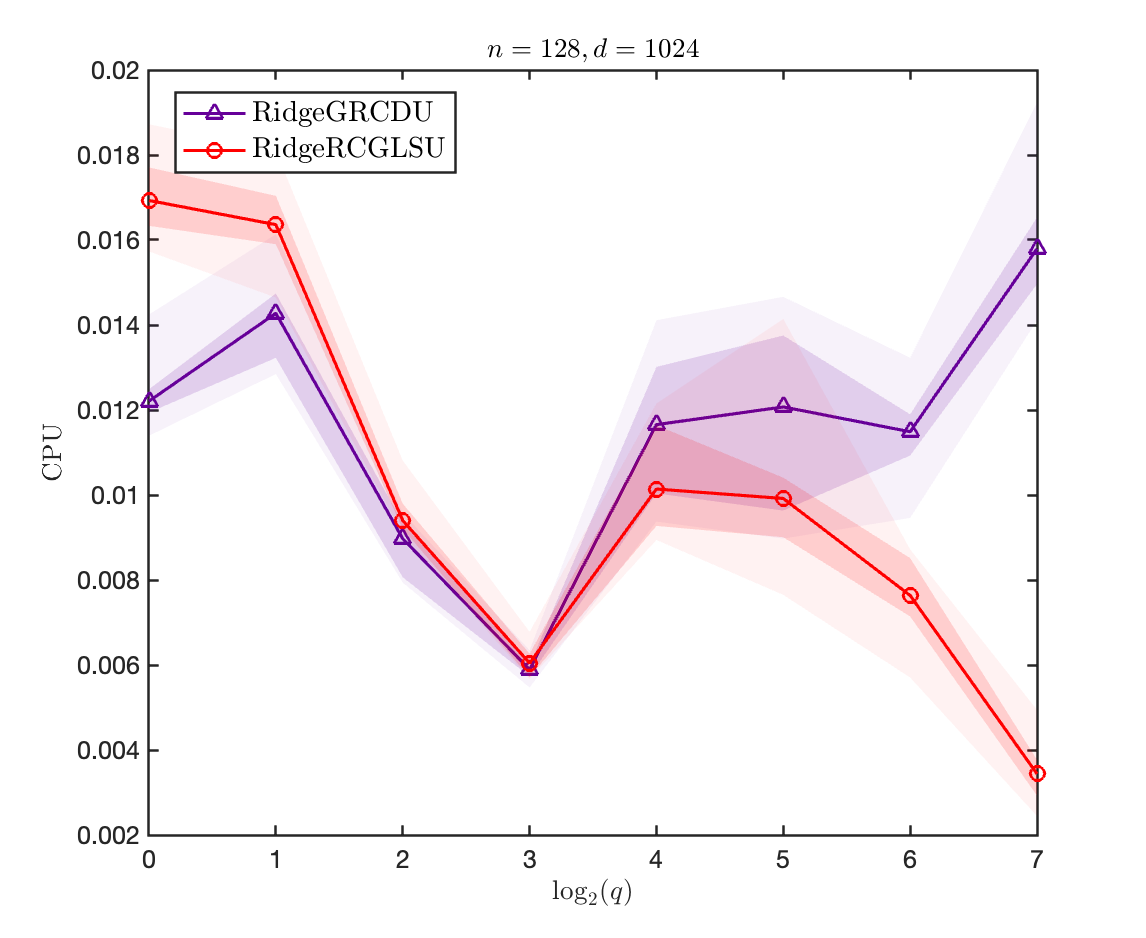}
			\end{tabular}
			\caption{Figures depict the evolution of the number of epochs and the CPU time with respect to the block size $q$ for synthetic datasets, with $\lambda=0.05$. The title of each plot indicates the values of $n$ and $d$. All computations are terminated once $\text{RSE}<10^{-10}$. }
			\label{figureR3}
		\end{figure}
		
		\begin{figure}[hptb]
			\centering
			\begin{tabular}{cc}
				\includegraphics[width=0.33\linewidth]{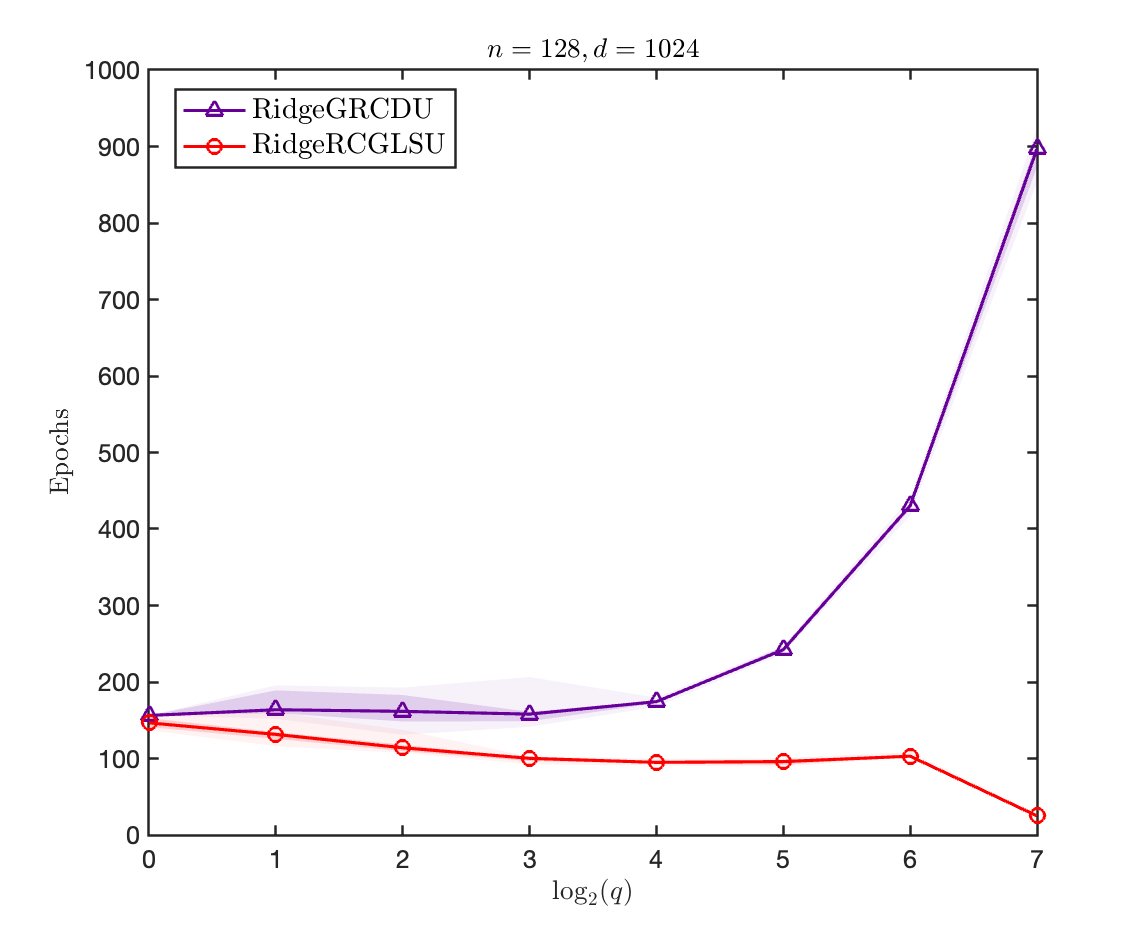}
				\includegraphics[width=0.33\linewidth]{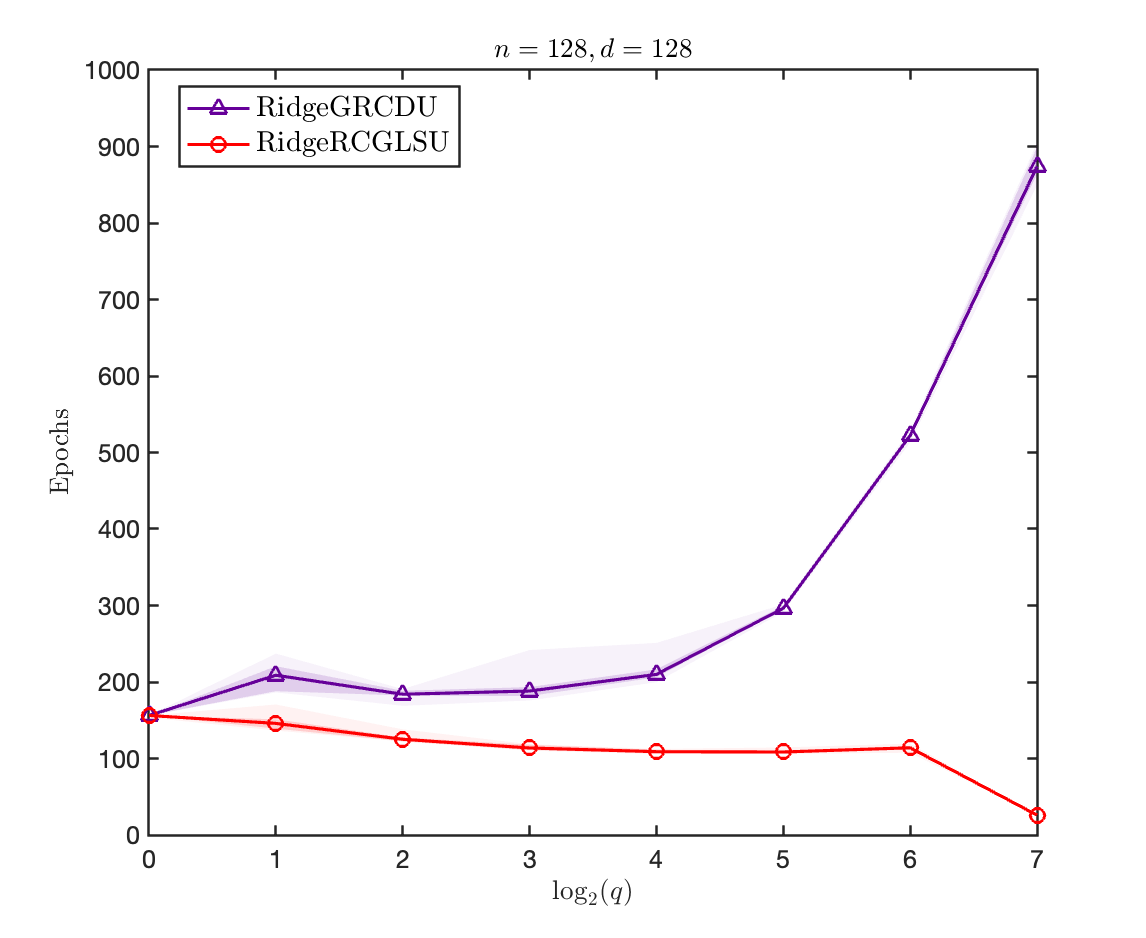}
				\includegraphics[width=0.33\linewidth]{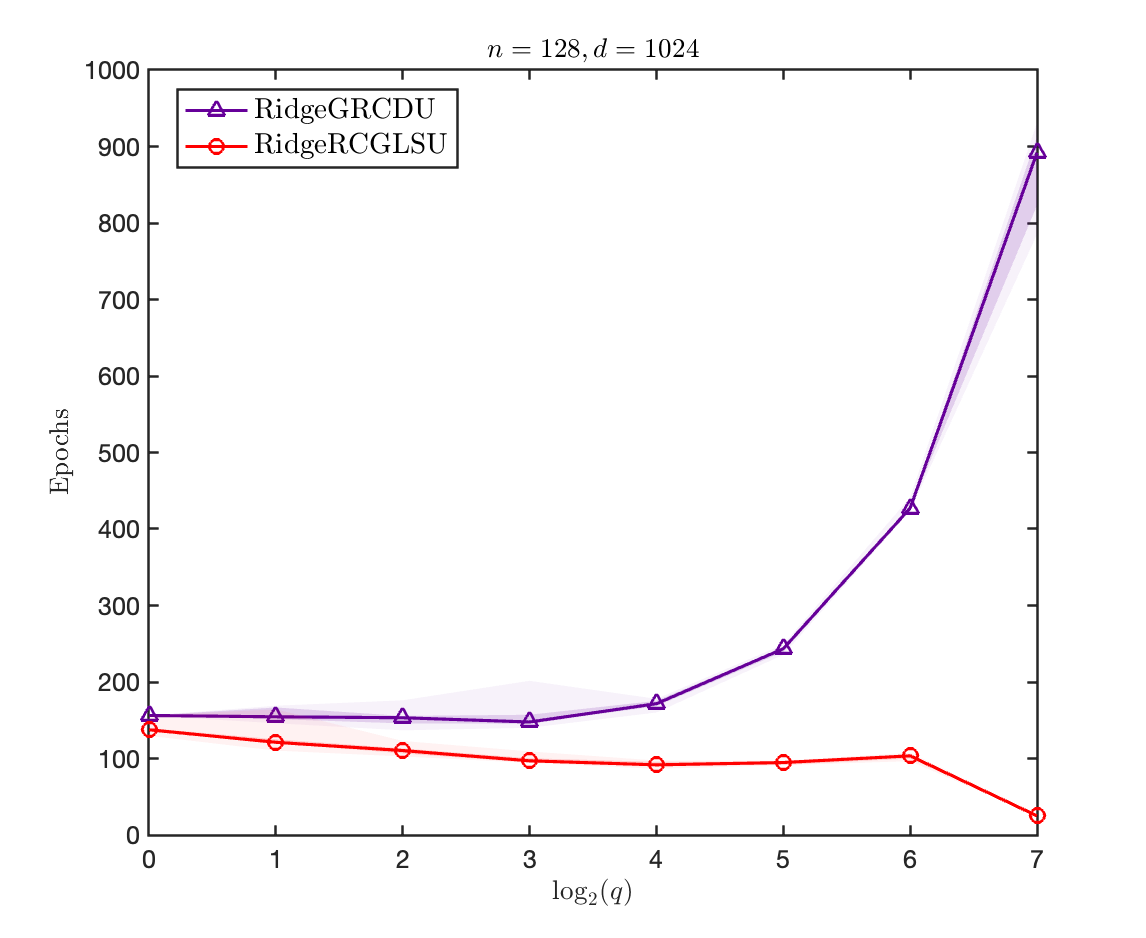}\\
				\includegraphics[width=0.33\linewidth]{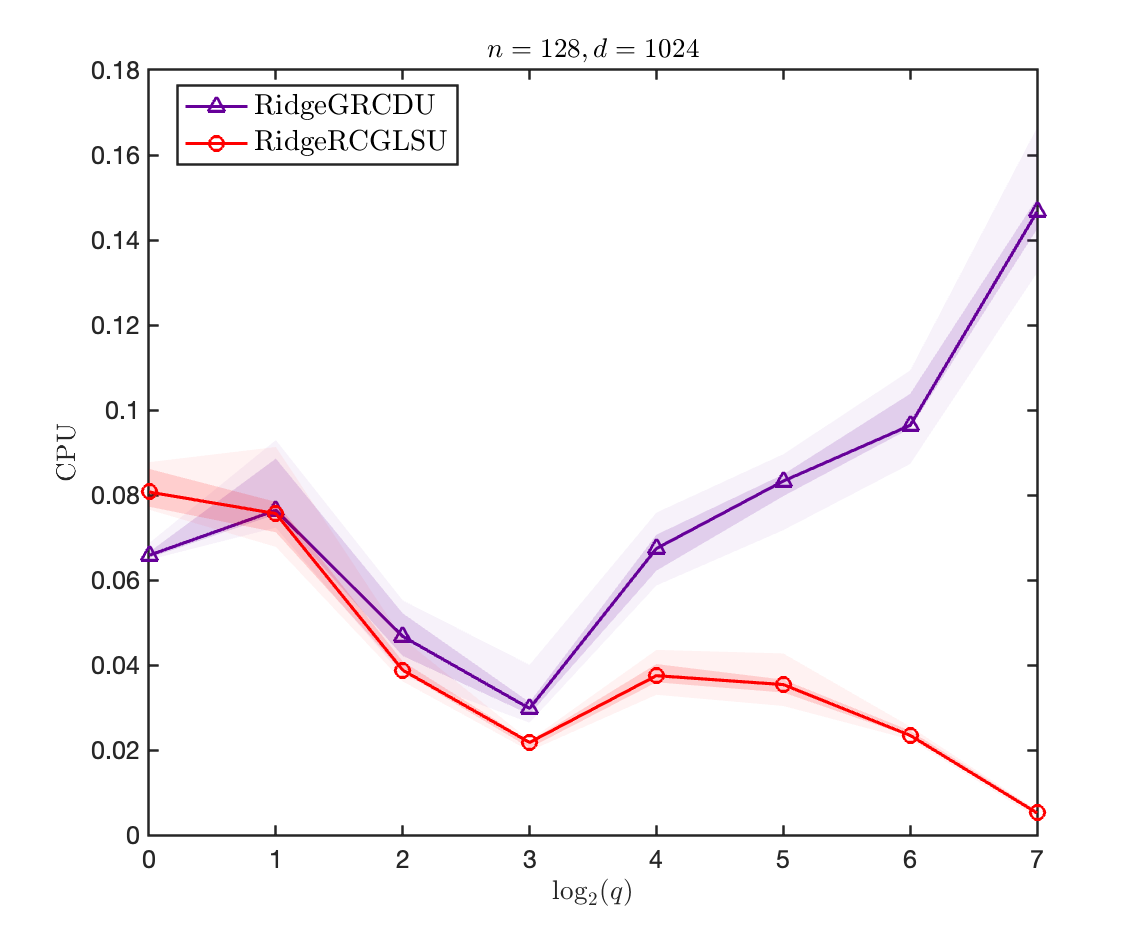}
				\includegraphics[width=0.33\linewidth]{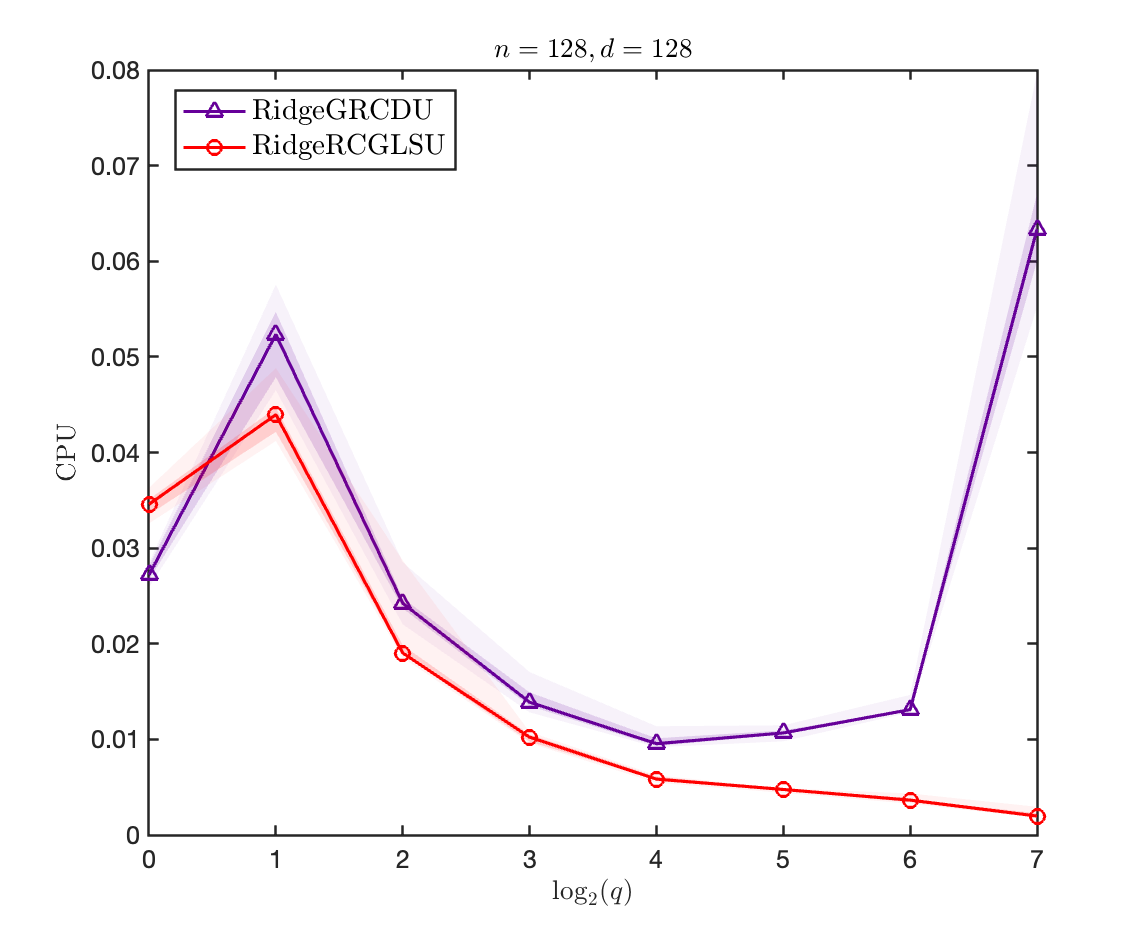}
				\includegraphics[width=0.33\linewidth]{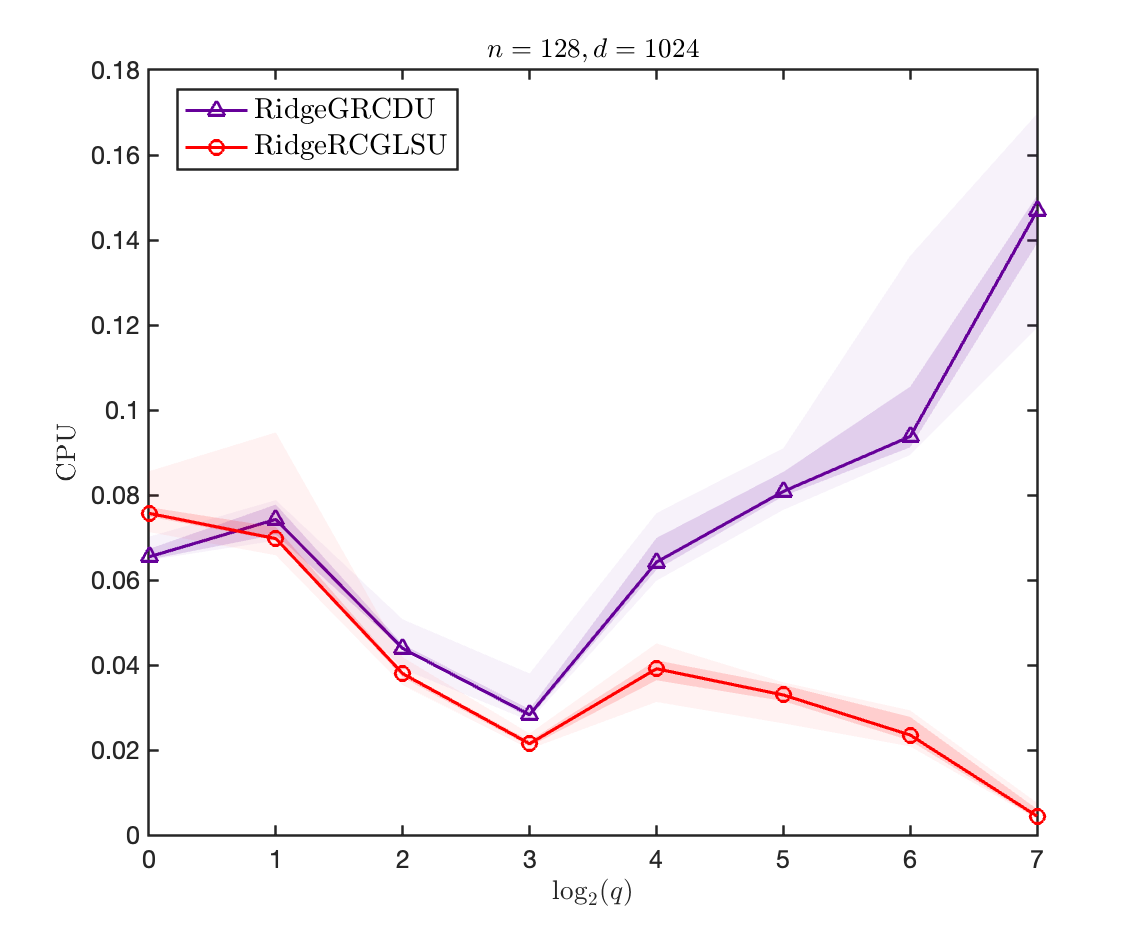}
			\end{tabular}
			\caption{Figures depict the evolution of the number of epochs and the CPU time with respect to the block size $q$ for synthetic datasets, with $\lambda=0.005$. The title of each plot indicates the values of $n$ and $d$. All computations are terminated once $\text{RSE}<10^{-10}$. }
			\label{figureR4}
		\end{figure}
		
		It can be observed  that RidgeRCGLSU consistently requires fewer epochs than RidgeGRCDU for any fixed block size $q$, and such acceleration advantages become more prominent for large $q$ values. In terms of CPU time, RidgeGRCDU performs better when $q$ is relatively small, owing to the lower per-iteration computational cost of RidgeGRCDU compared with RidgeRCGLSU. Nevertheless, RidgeRCGLSU outperforms RidgeGRCDU in CPU time as $q$ increases (e.g., $q \geq 16$), validating the practical efficiency of the proposed accelerated method for moderate and large block sizes.

		\subsection{Comparison to HImRidgeSketchU}
	
		We compare the performance of RidgeRCGLSU with HImRidgeSketchU in this subsection. 
		Figures \ref{figureR5} and \ref{figureR6} present the computational results for synthetic datasets with $\lambda=0.05$ and $\lambda=0.005$, respectively. It can be observed that RidgeRCGLSU delivers better scalability in terms of flops compared with HImRidgeSketchU. The performance gap gradually enlarges with increasing matrix dimensions. Meanwhile, RidgeRCGLSU achieves faster CPU runtime than HImRidgeSketchU in all tested cases.
		
		
		\begin{figure}[hptb]
			\centering
			\begin{tabular}{cc}
				\includegraphics[width=0.33\linewidth]{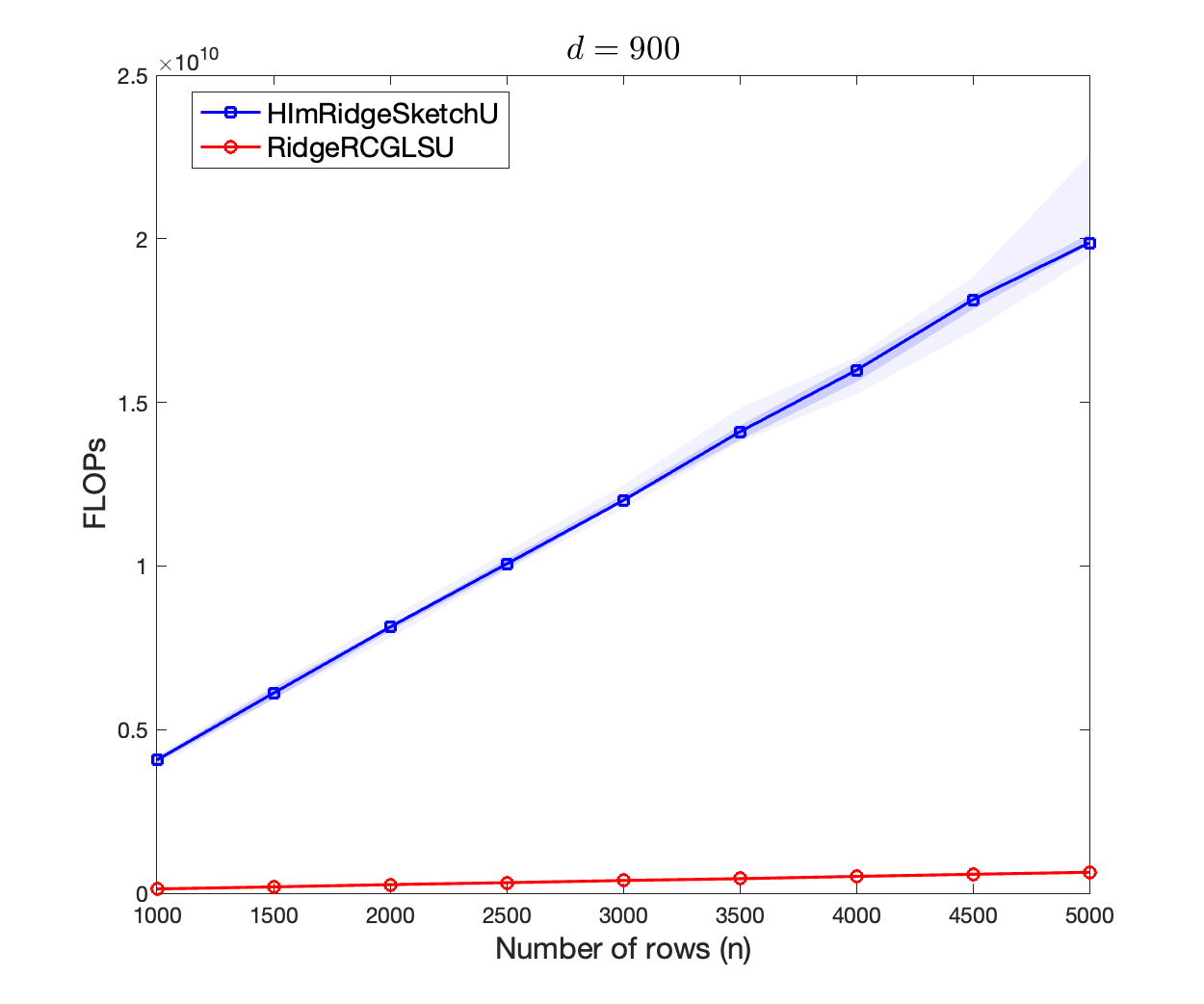}
				\includegraphics[width=0.33\linewidth]{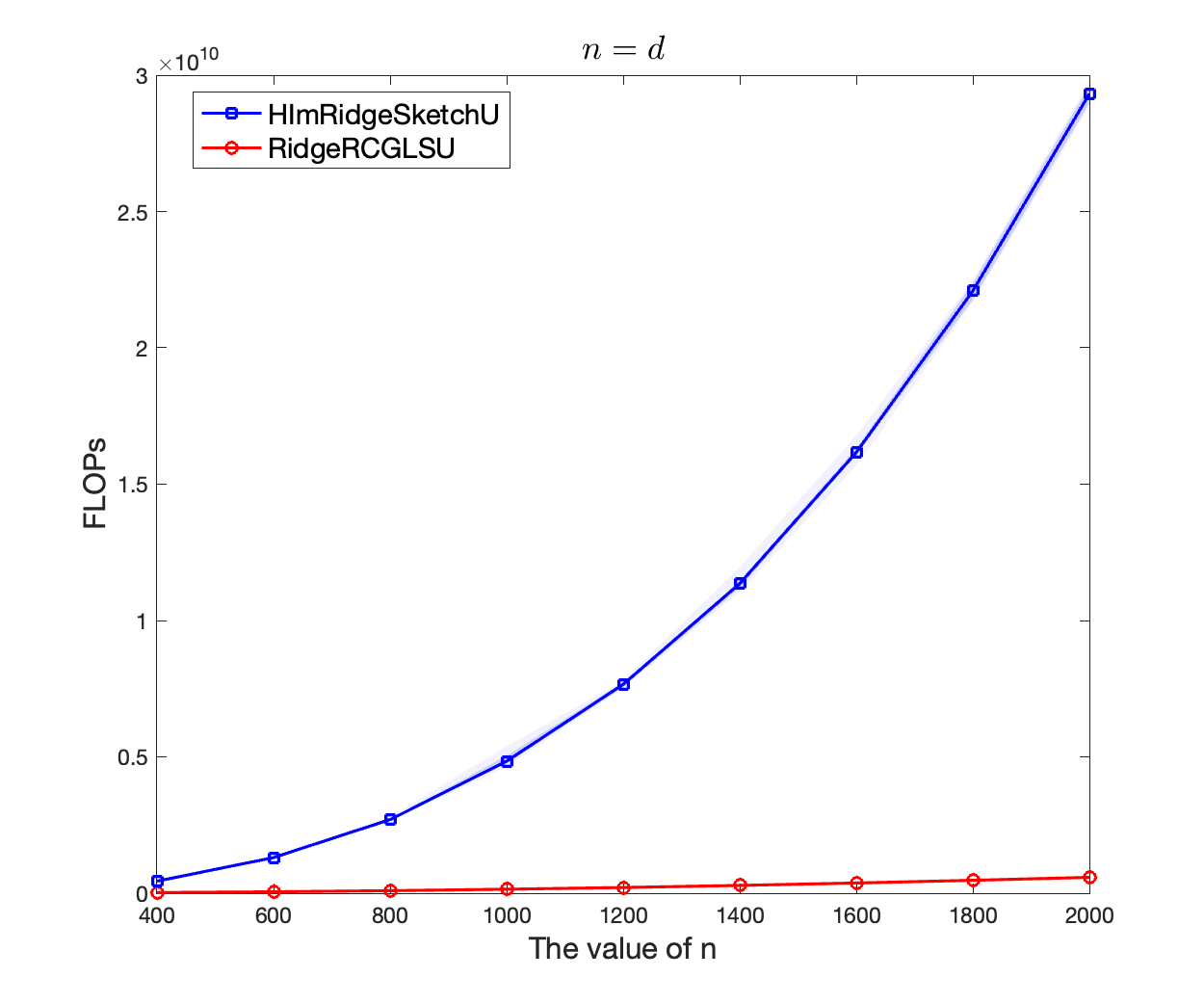}
				\includegraphics[width=0.33\linewidth]{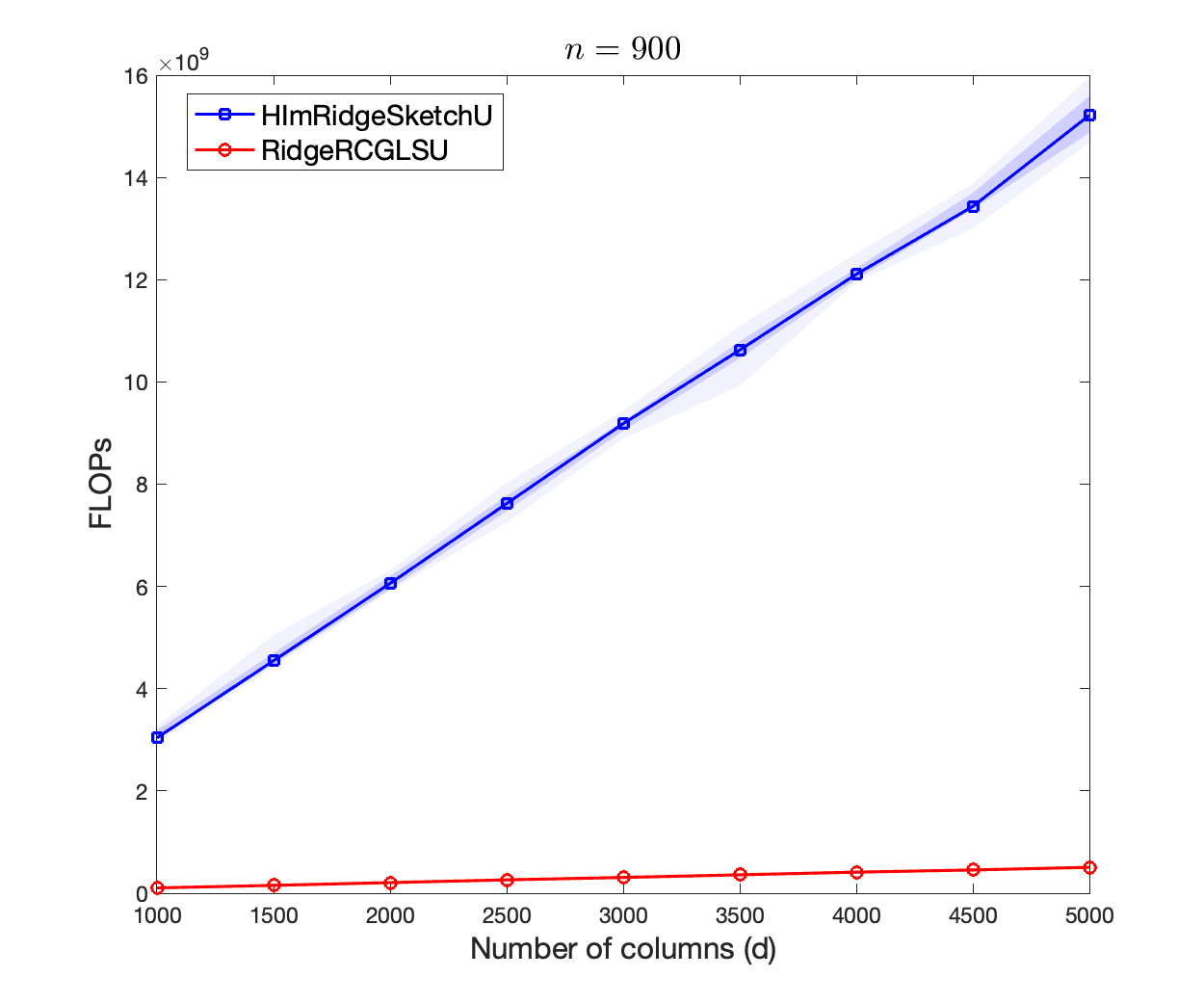}
				\\
				\includegraphics[width=0.33\linewidth]{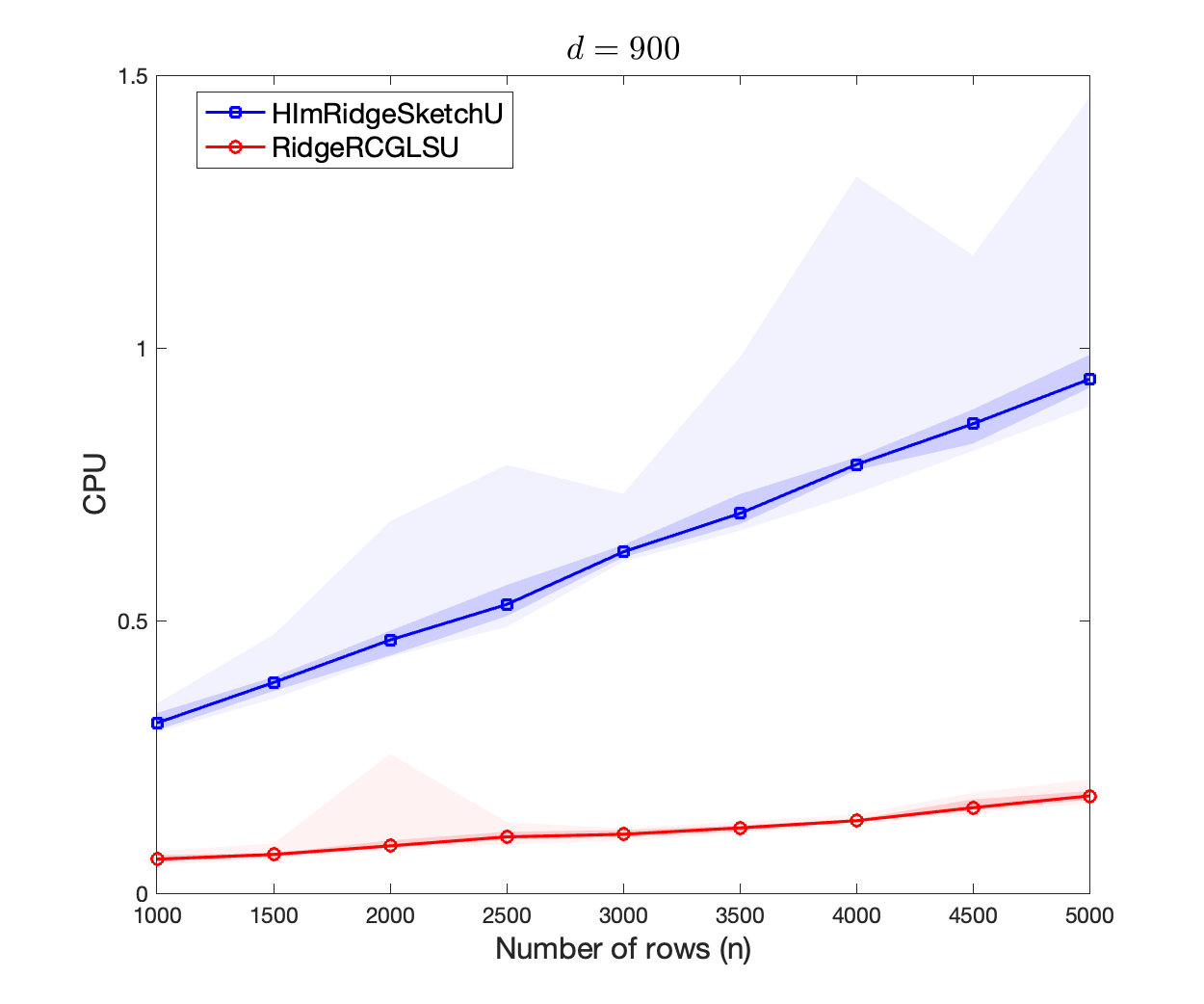}
				\includegraphics[width=0.33\linewidth]{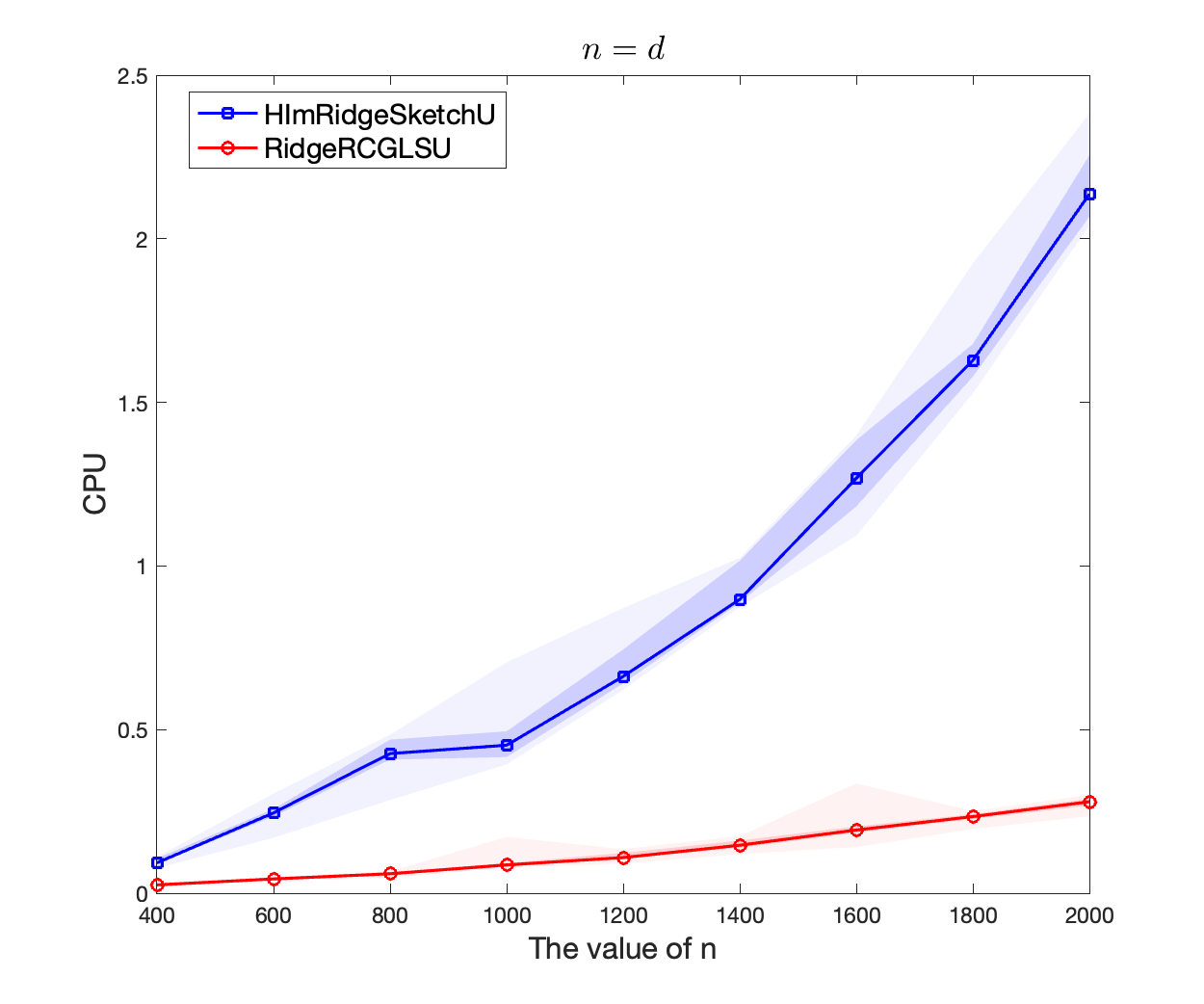}
				\includegraphics[width=0.33\linewidth]{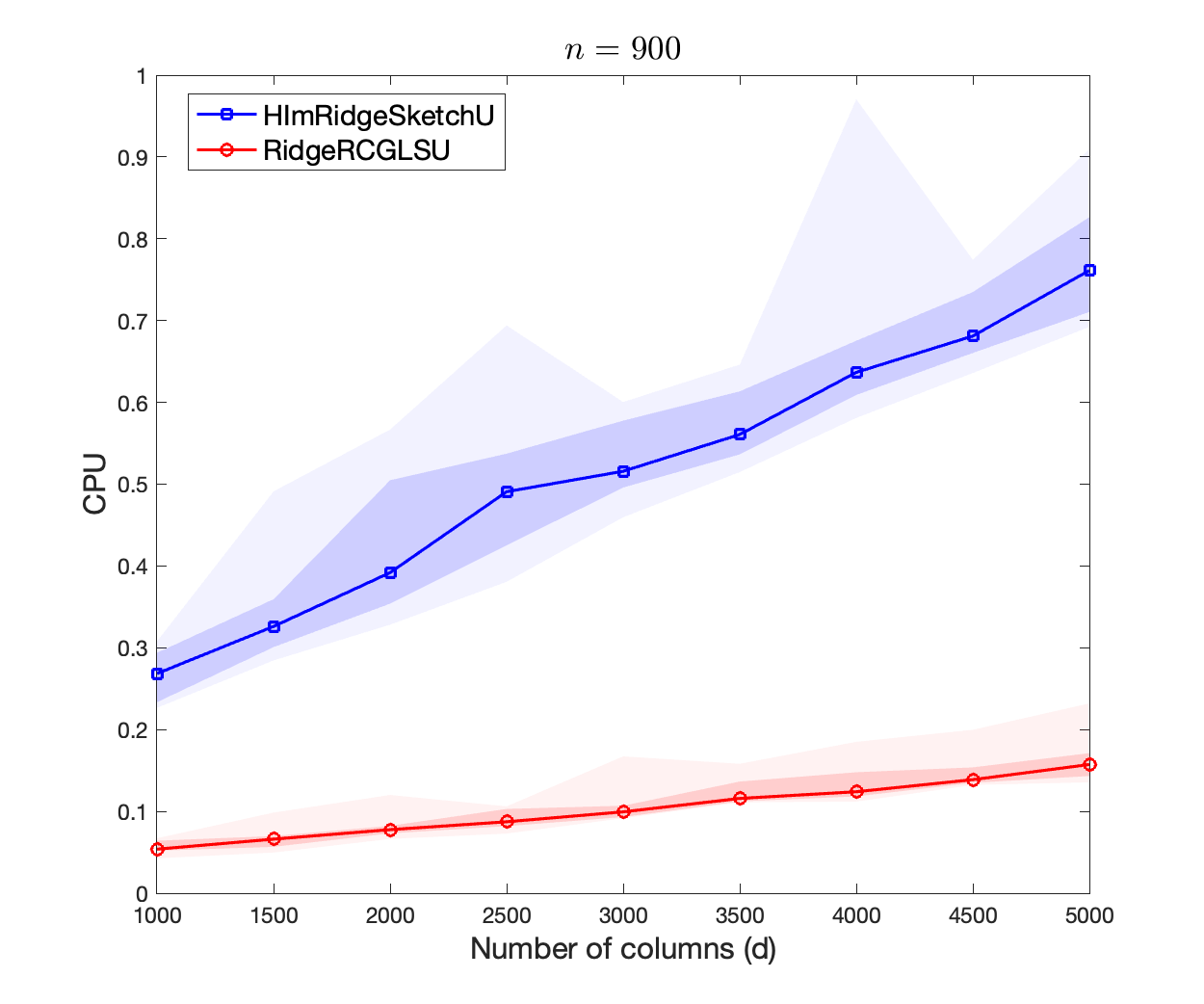}
			\end{tabular}
			\caption{Performance of HImRidgeSketchU and RidgeRCGLSU on synthetic datasets with $\lambda=0.05$. Figures depict the total flops and the CPU time across varying dimensions for overdetermined ($n>d$, left), square ($n=d$, middle), and underdetermined ($n<d$, right) cases. We set $q=50$ and terminate the algorithms once $\text{RSE}<10^{-10}$. The titles specify the fixed parameters for each case.}
			\label{figureR5}
		\end{figure}
		
		\begin{figure}[hptb]
			\centering
			\begin{tabular}{cc}
				\includegraphics[width=0.33\linewidth]{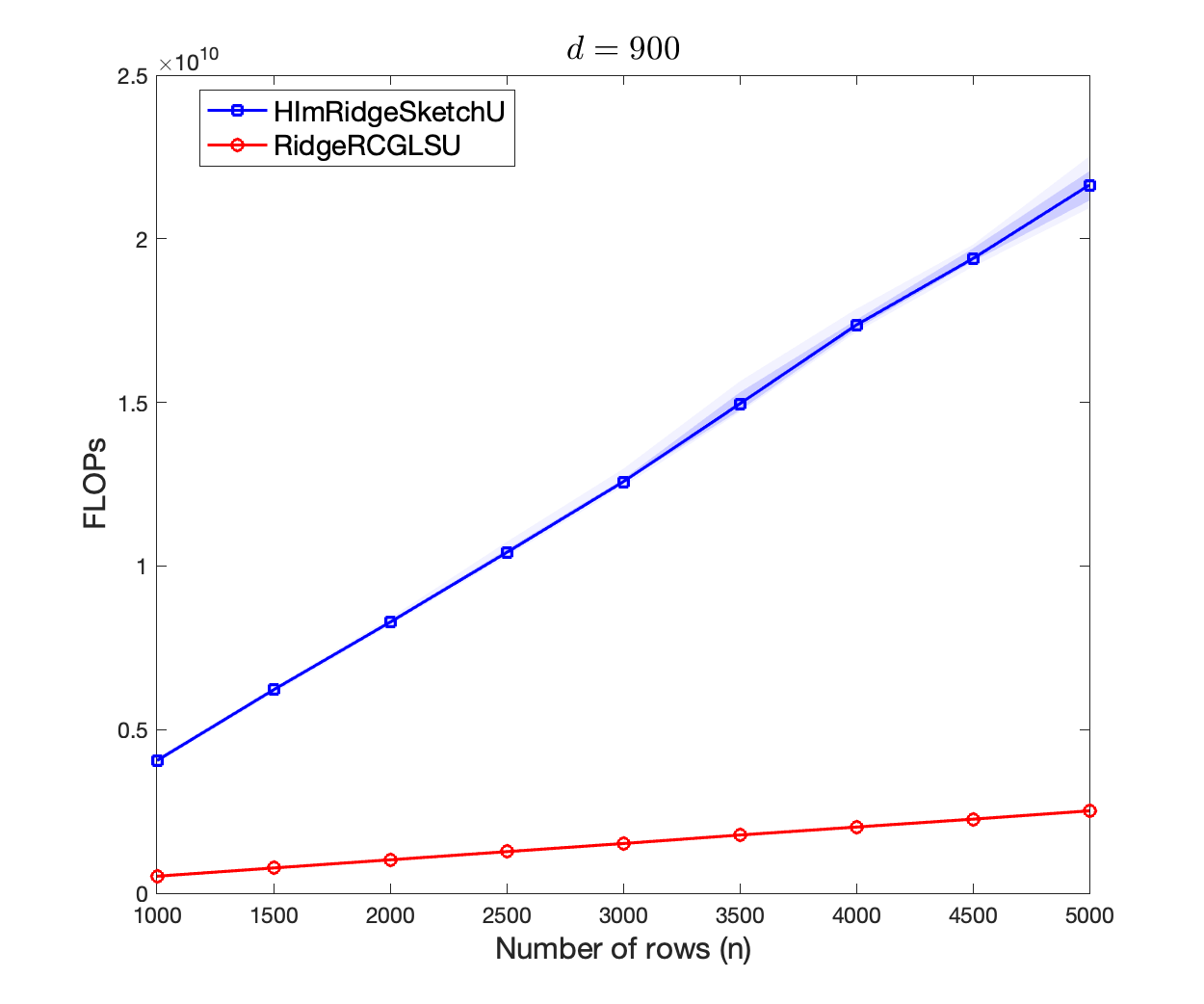}
				\includegraphics[width=0.33\linewidth]{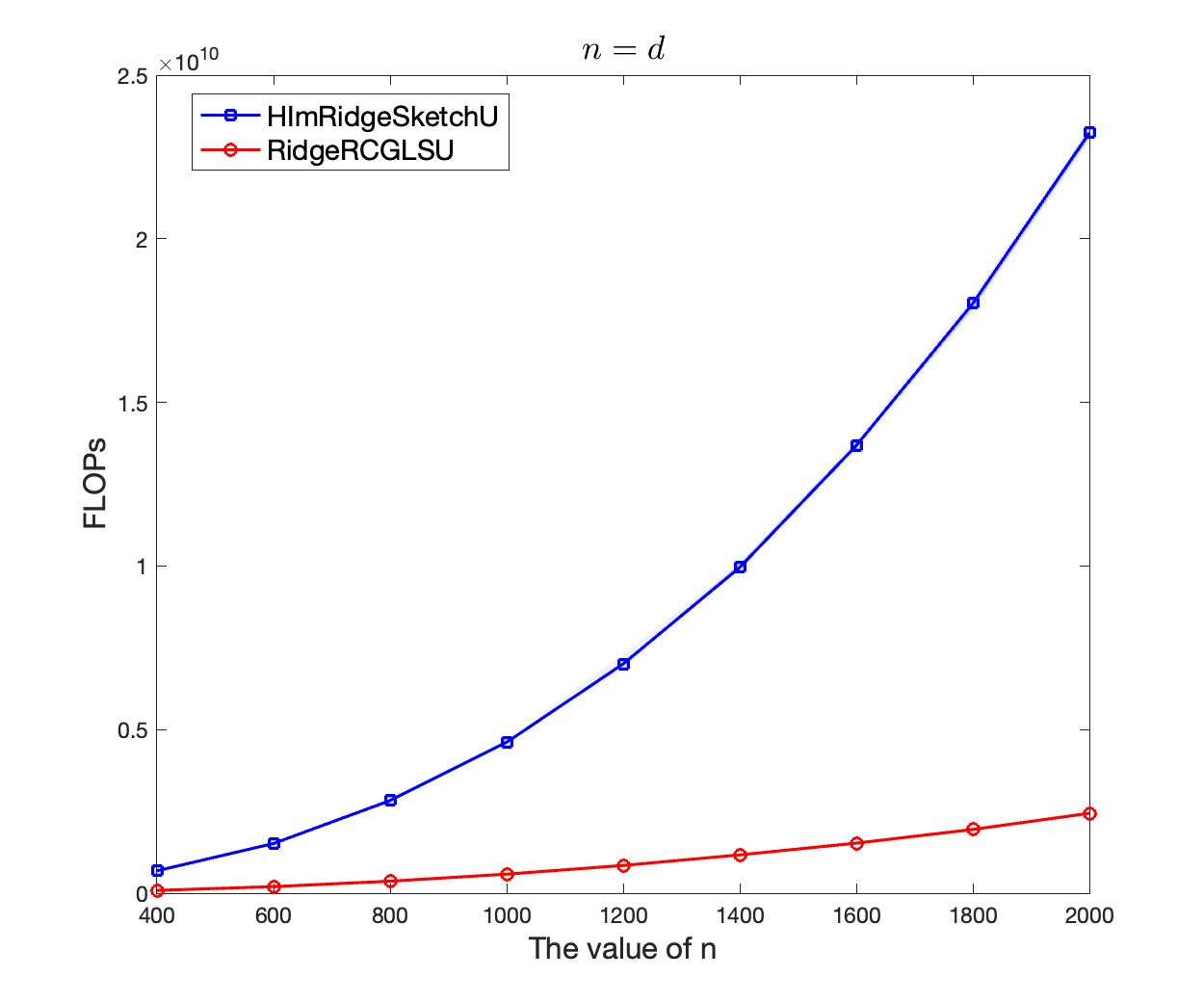}
				\includegraphics[width=0.33\linewidth]{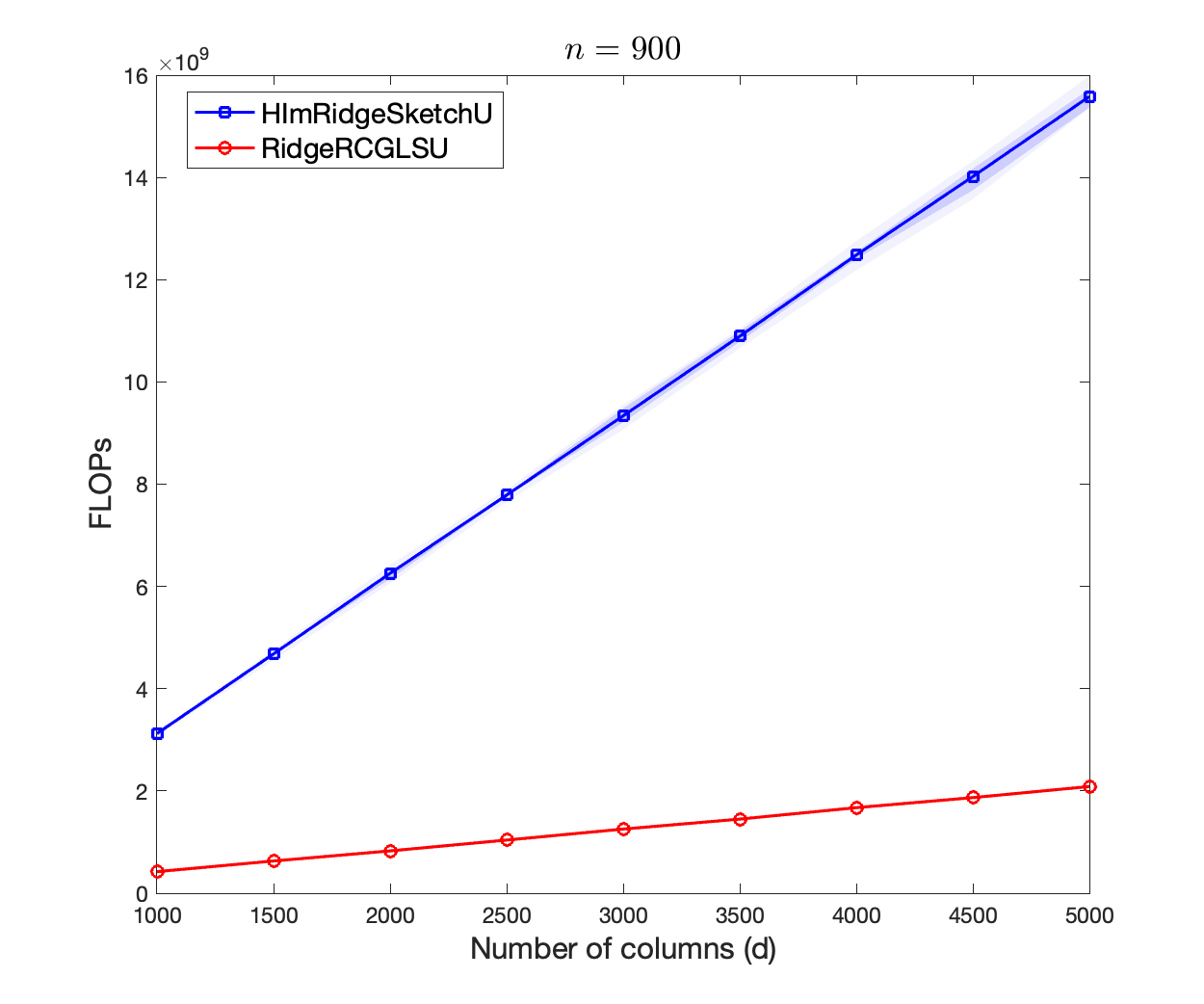}
				\\
				\includegraphics[width=0.33\linewidth]{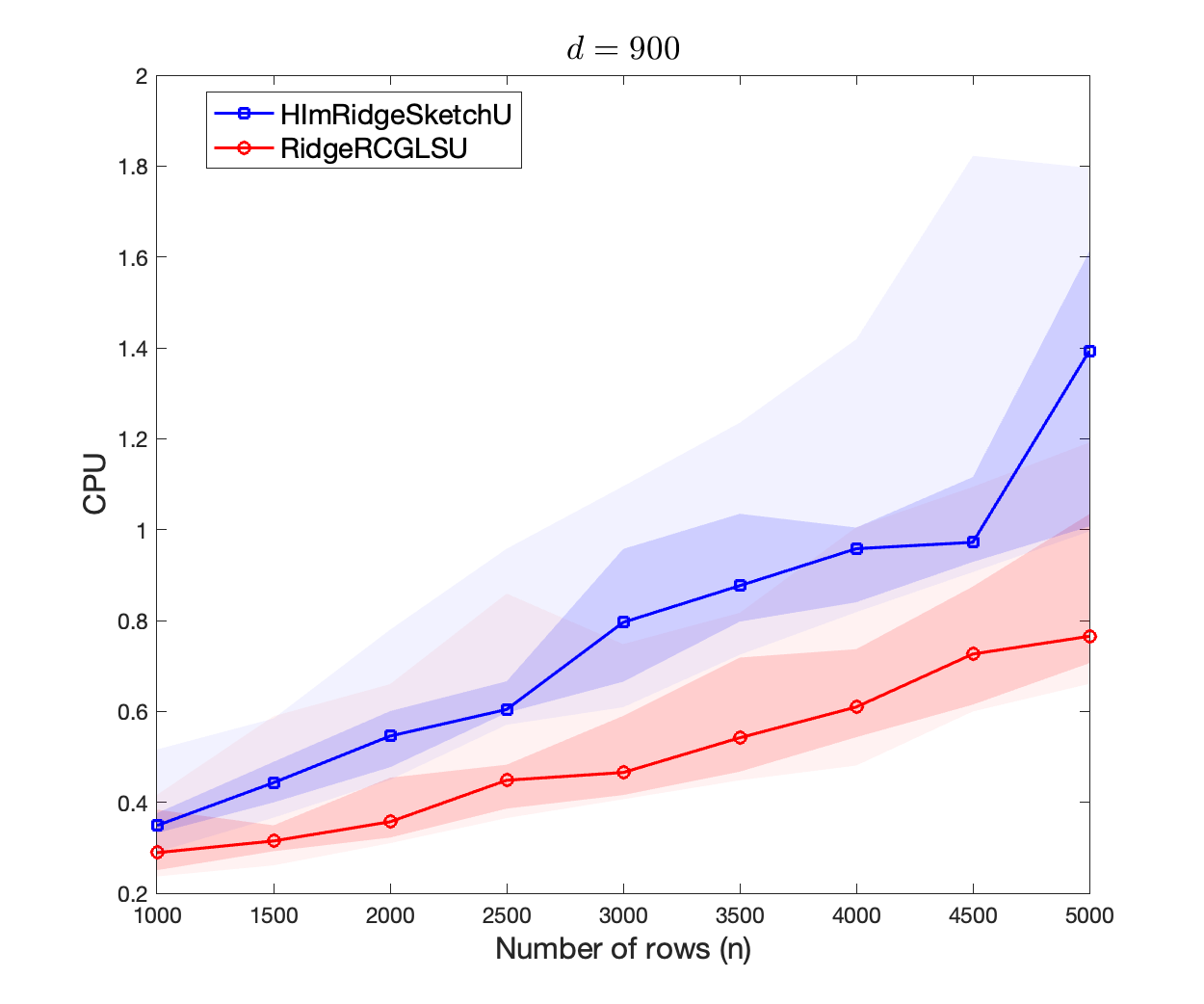}
				\includegraphics[width=0.33\linewidth]{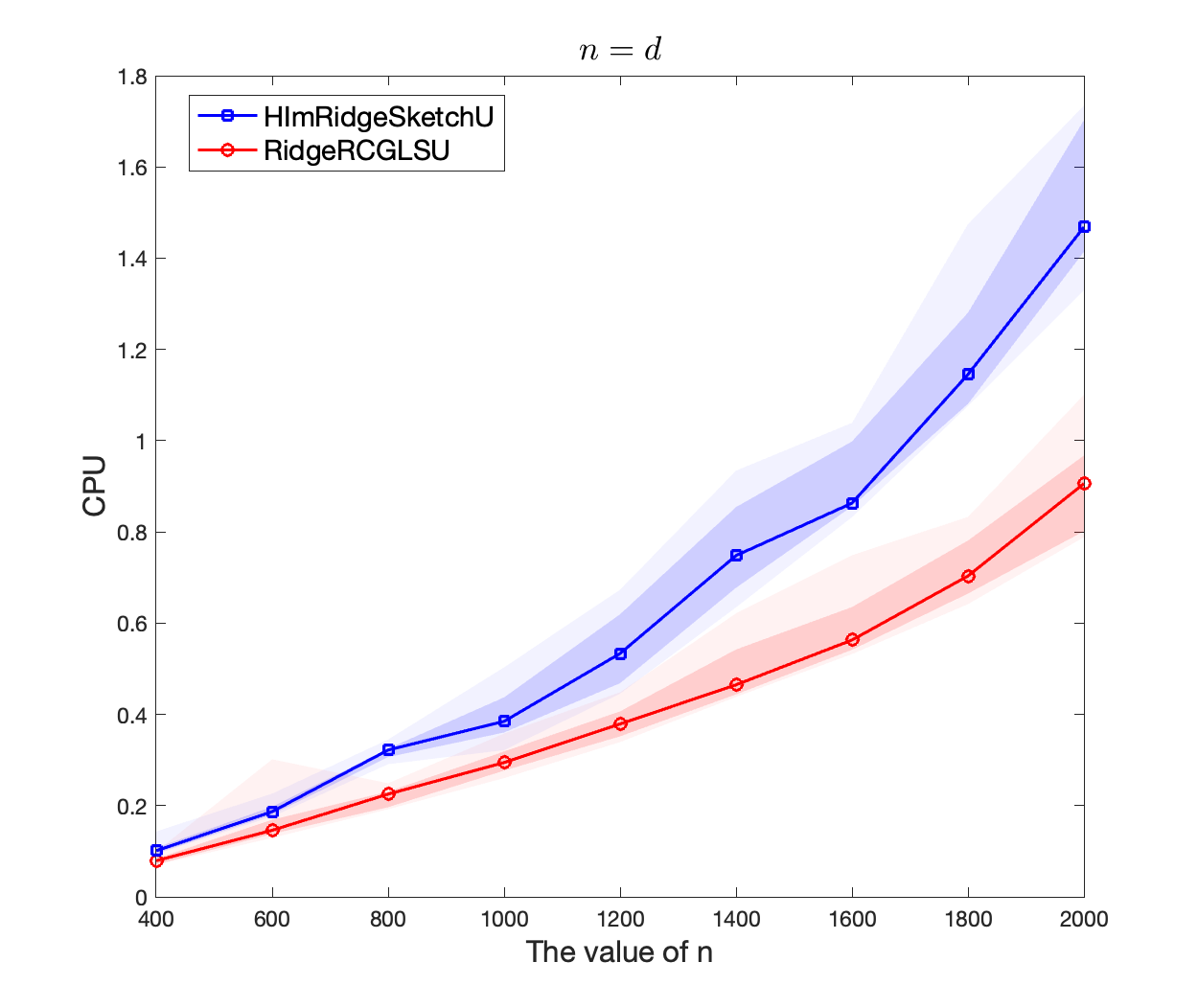}
				\includegraphics[width=0.33\linewidth]{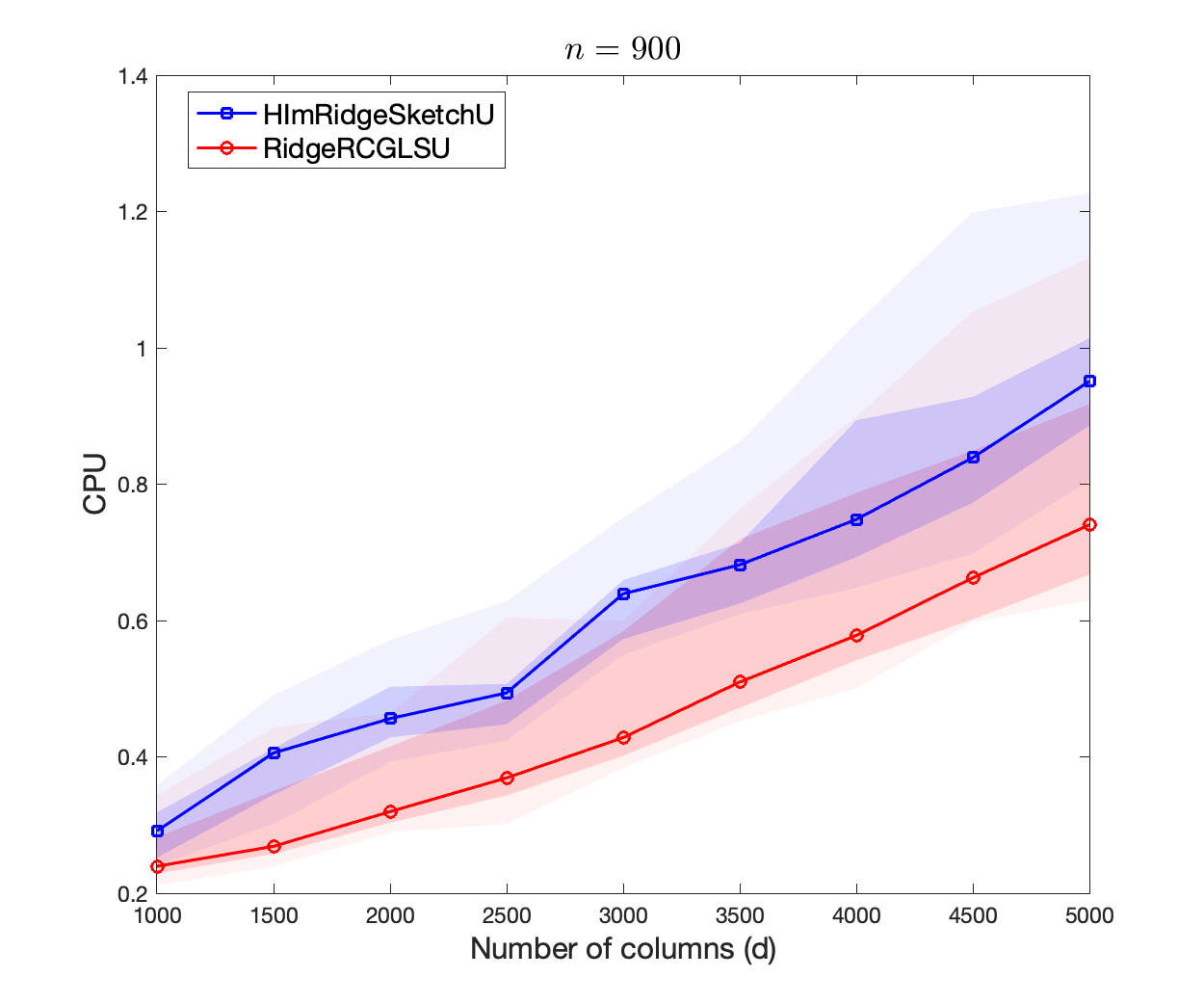}
			\end{tabular}
			\caption{Performance of HImRidgeSketchU and RidgeRCGLSU on synthetic datasets with $\lambda=0.005$. Figures depict the total flops and the CPU time across varying dimensions for overdetermined ($n>d$, left), square ($n=d$, middle), and underdetermined ($n<d$, right) cases. We set $q=50$ and terminate the algorithms once $\text{RSE}<10^{-10}$. The titles specify the fixed parameters for each case.}
			\label{figureR6}
		\end{figure}
		
		Figures \ref{figureR7} and \ref{figureR8} report the flops and CPU time of the compared methods on real-world LIBSVM datasets \cite{chang2011libsvm} with $\lambda=0.05$. 
		%
%
		It can be seen that RidgeRCGLSU outperforms HImRidgeSketchU in terms of both flops and CPU time across all real-world datasets. In particular, on the \texttt{protein} dataset, RidgeRCGLSU converges much faster to the prescribed tolerance, while HImRidgeSketchU converges considerably slowly. In addition, RidgeRCGLSU yields narrower statistical intervals over independent trials, demonstrating better numerical stability. Such superior performance benefits from the variance reduction property of the proposed gradient estimator in Remark \ref{remark-VR-RCD2}, as well as the exact line search and conjugacy constraints adopted in our algorithm.
		
		\begin{figure}[hptb]
			\centering
			\begin{tabular}{cc}
				\includegraphics[width=0.33\linewidth]{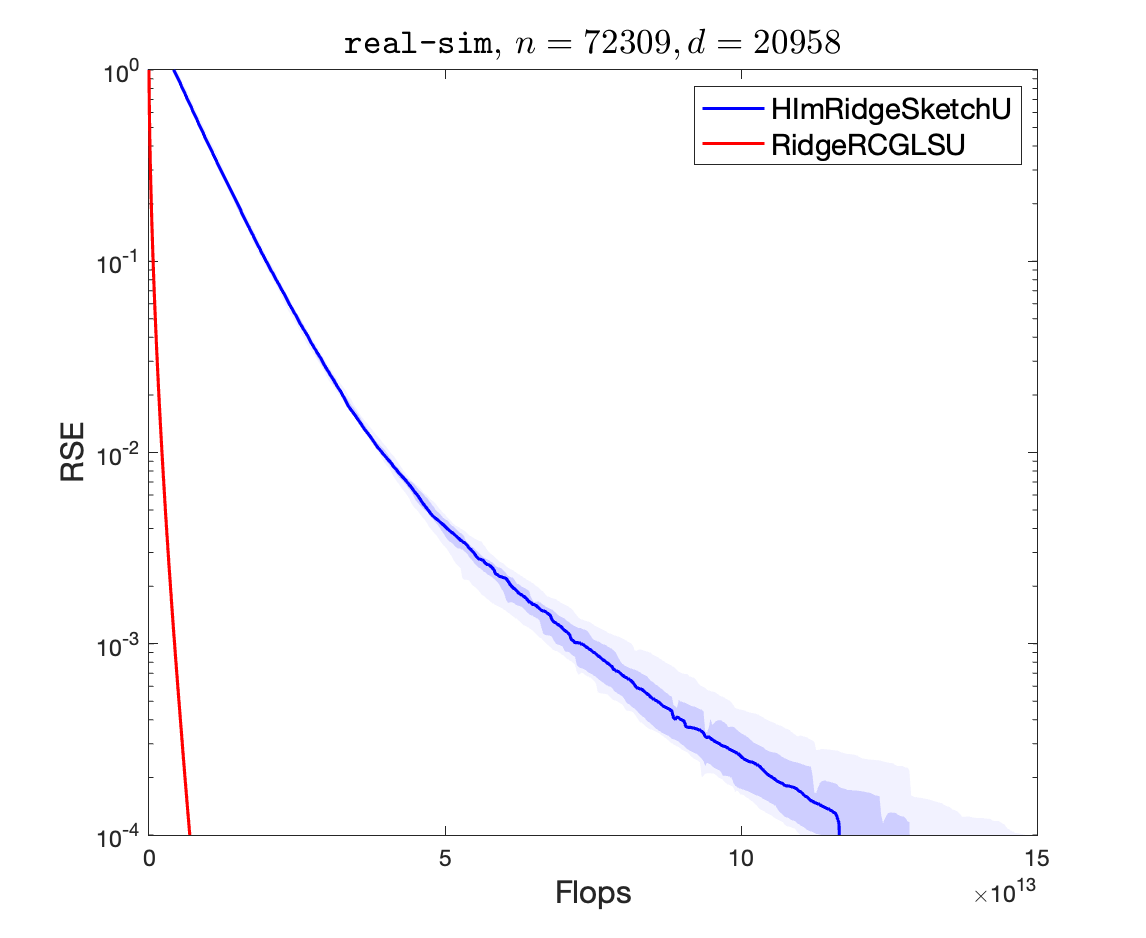}
				\includegraphics[width=0.33\linewidth]{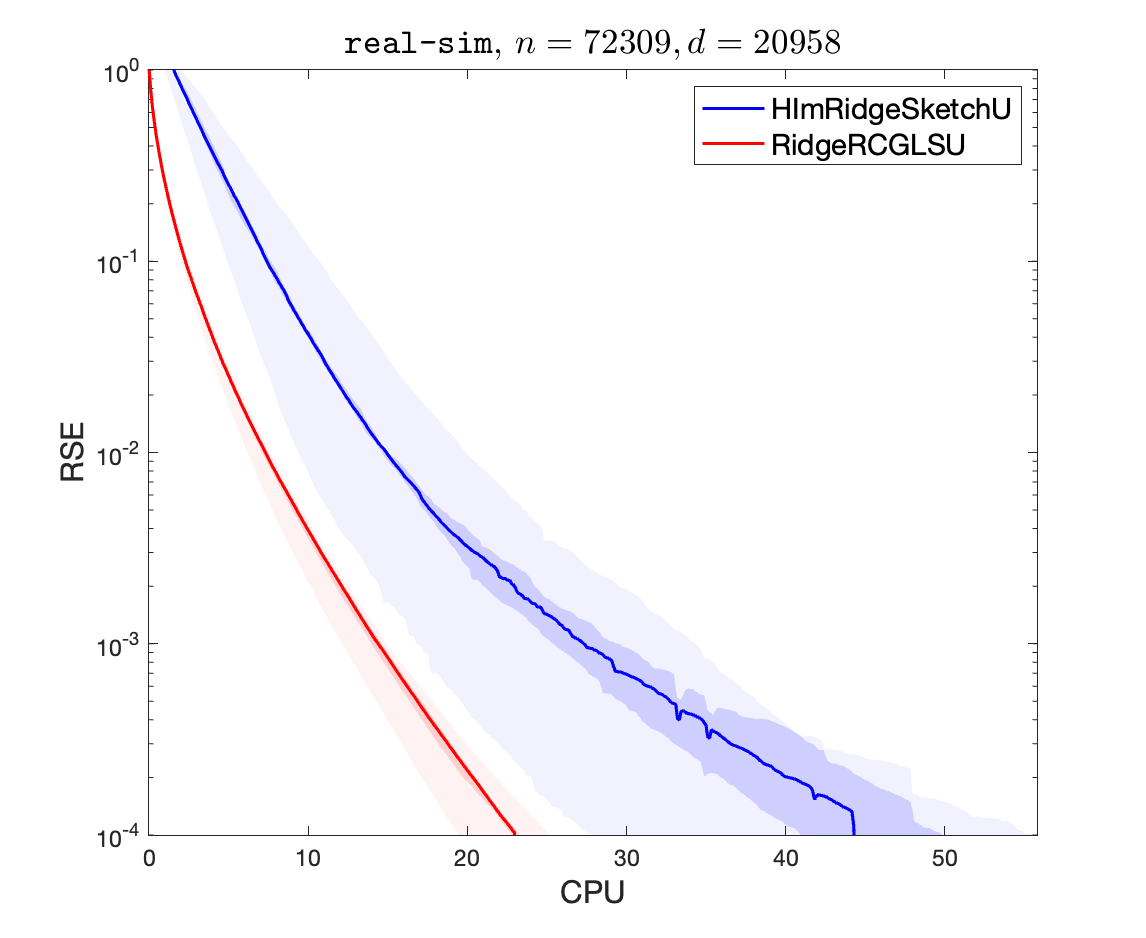}
				\\
				\includegraphics[width=0.33\linewidth]{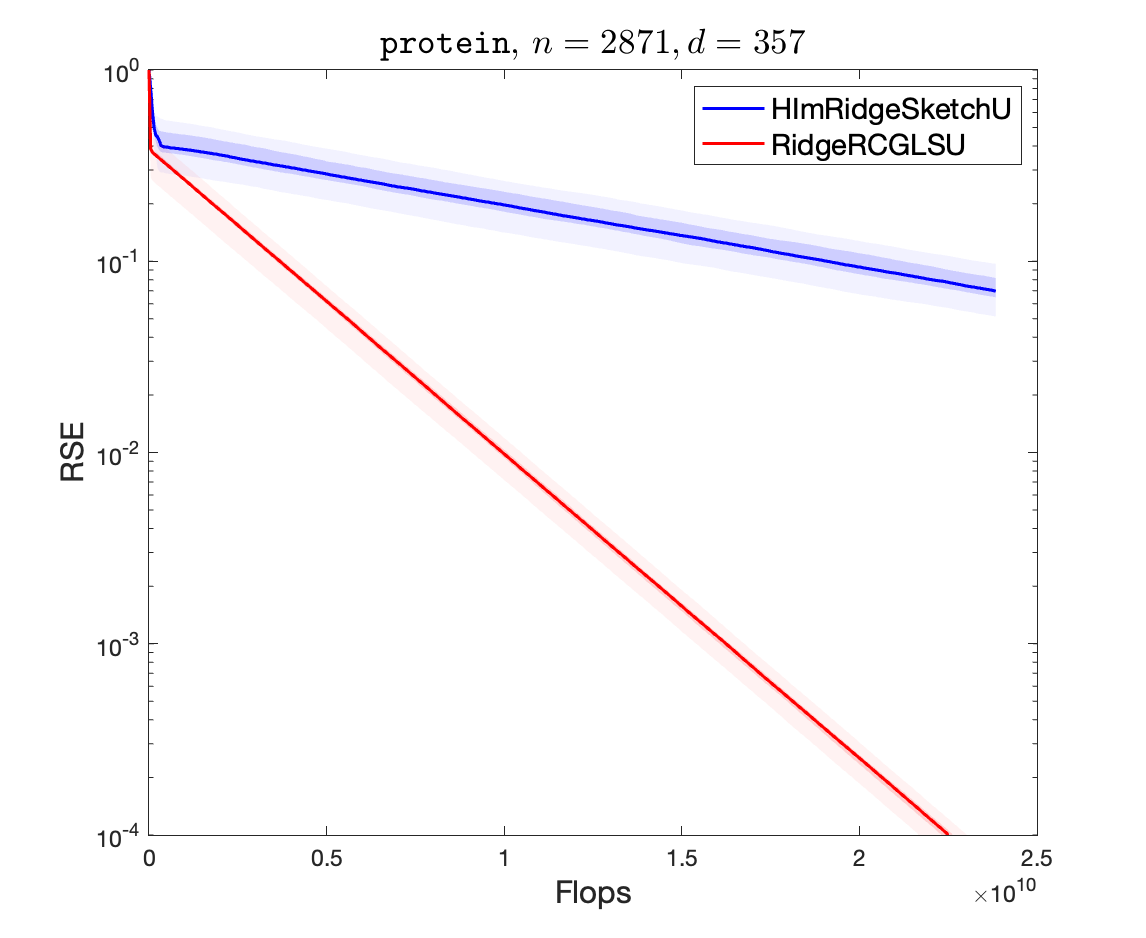}
				\includegraphics[width=0.33\linewidth]{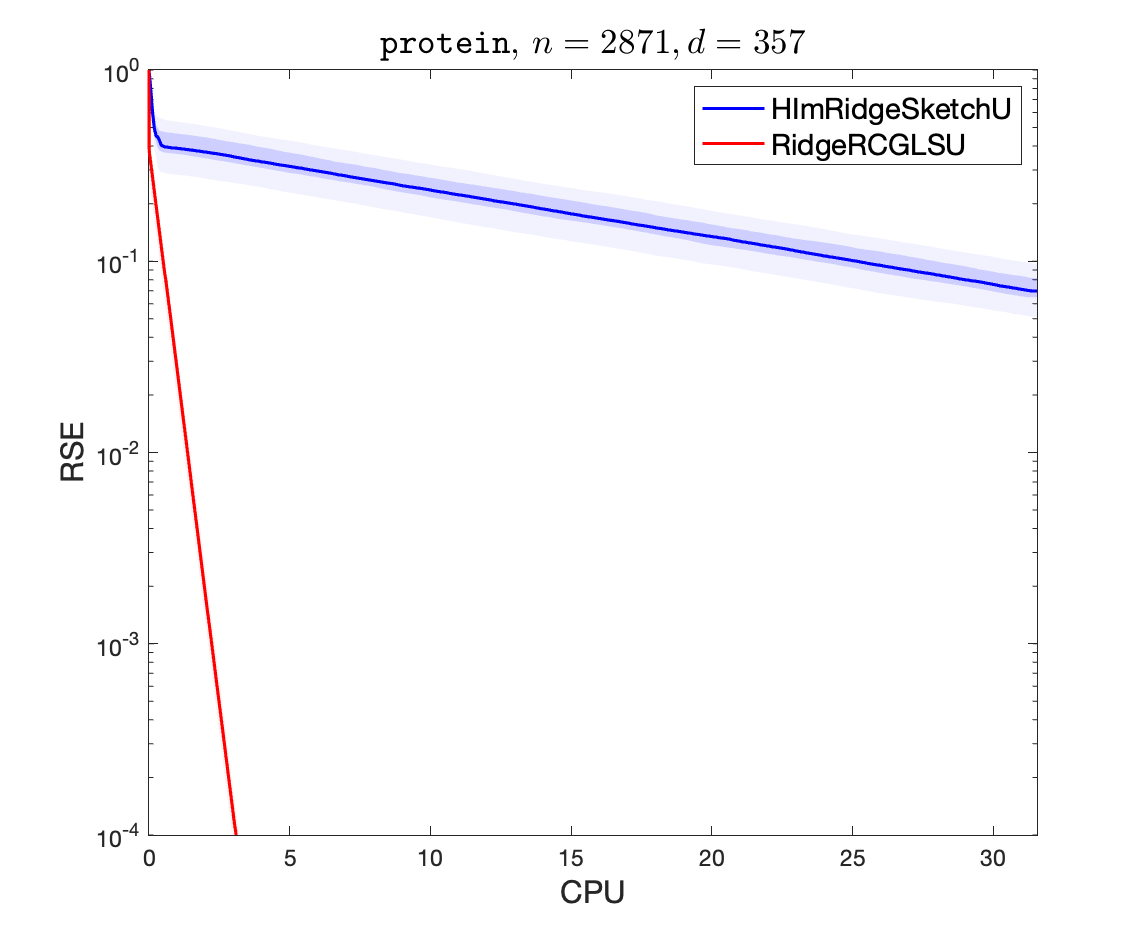}
			\end{tabular}
			\caption{Performance on overdetermined datasets from LIBSVM \cite{chang2011libsvm} with $\lambda=0.05$. Figures depict the evolution of RSE with respect to the total flops and the CPU time. We set $q=2000$ for {\tt real-sim} and $q=50$ for {\tt protein}, and terminate the algorithms once RSE $< 10^{-4}$ or the number of iterations exceeds a certain limit. The title of each plot indicates the names and sizes of the data.}
			\label{figureR7}
		\end{figure}
		
		\begin{figure}[hptb]
			\centering
			\begin{tabular}{cc}
				\includegraphics[width=0.33\linewidth]{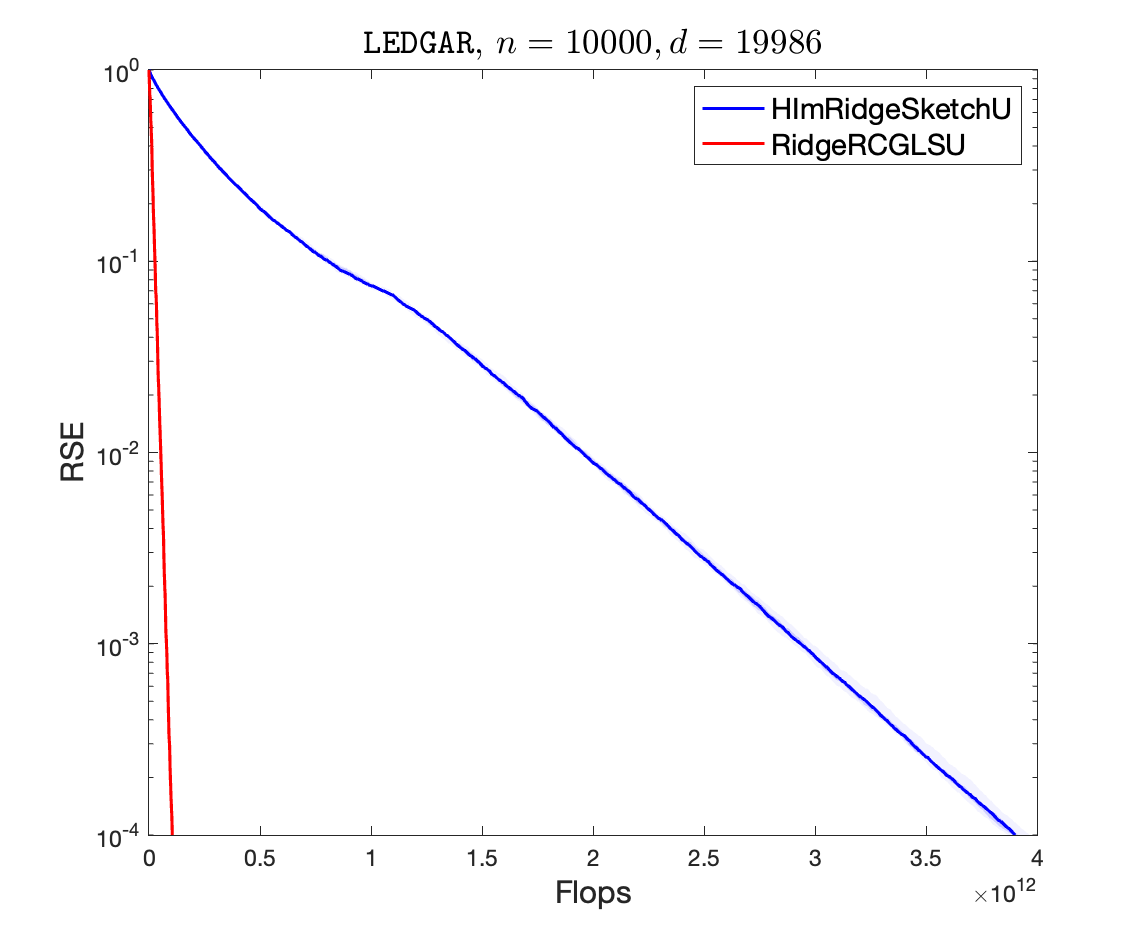}
				\includegraphics[width=0.33\linewidth]{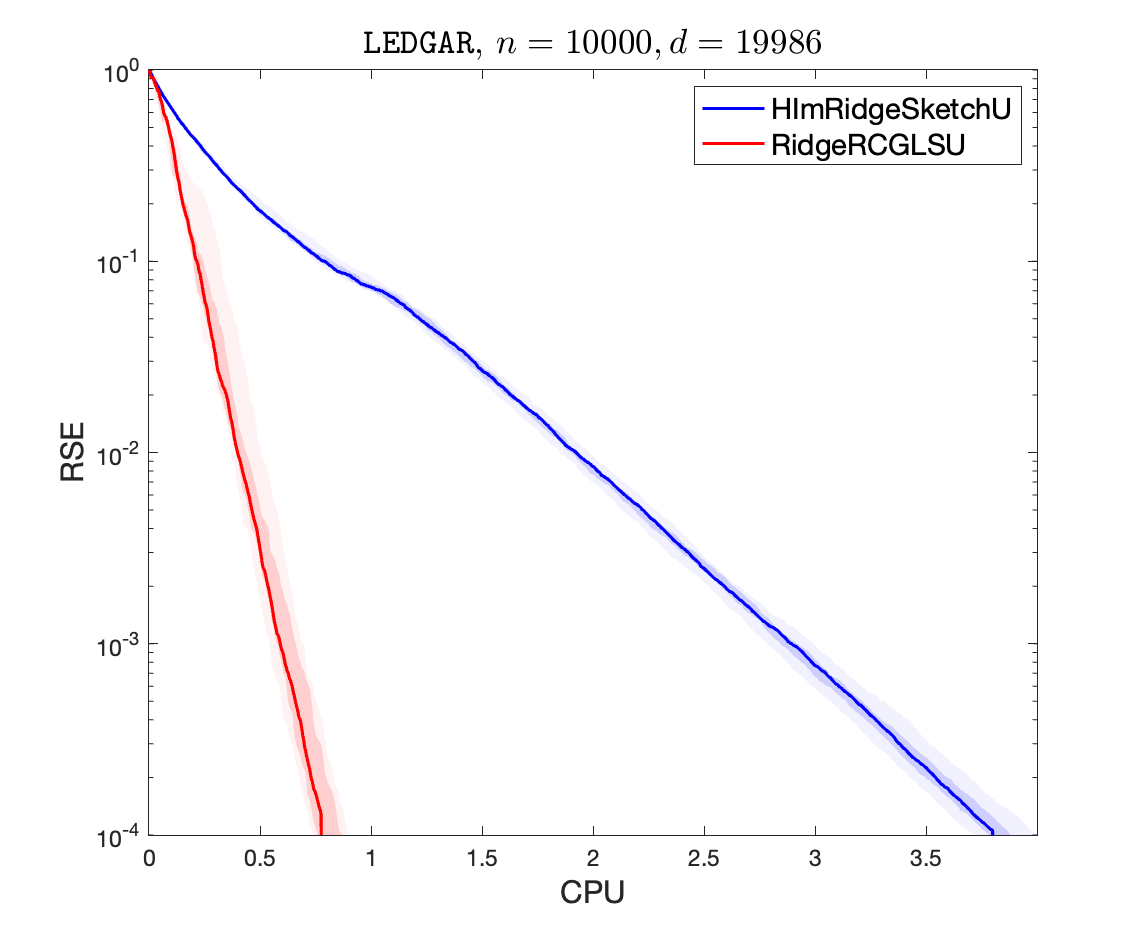}
				\\
				\includegraphics[width=0.33\linewidth]{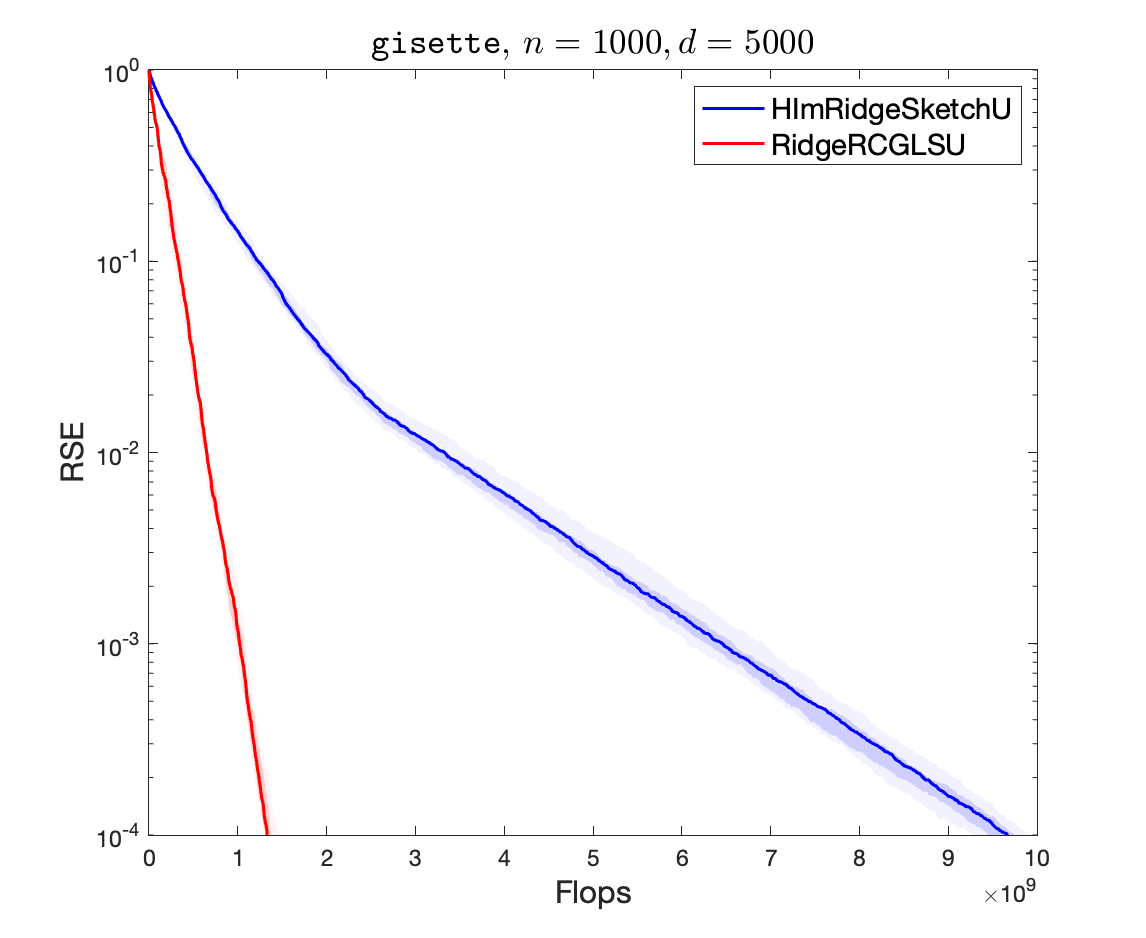}
				\includegraphics[width=0.33\linewidth]{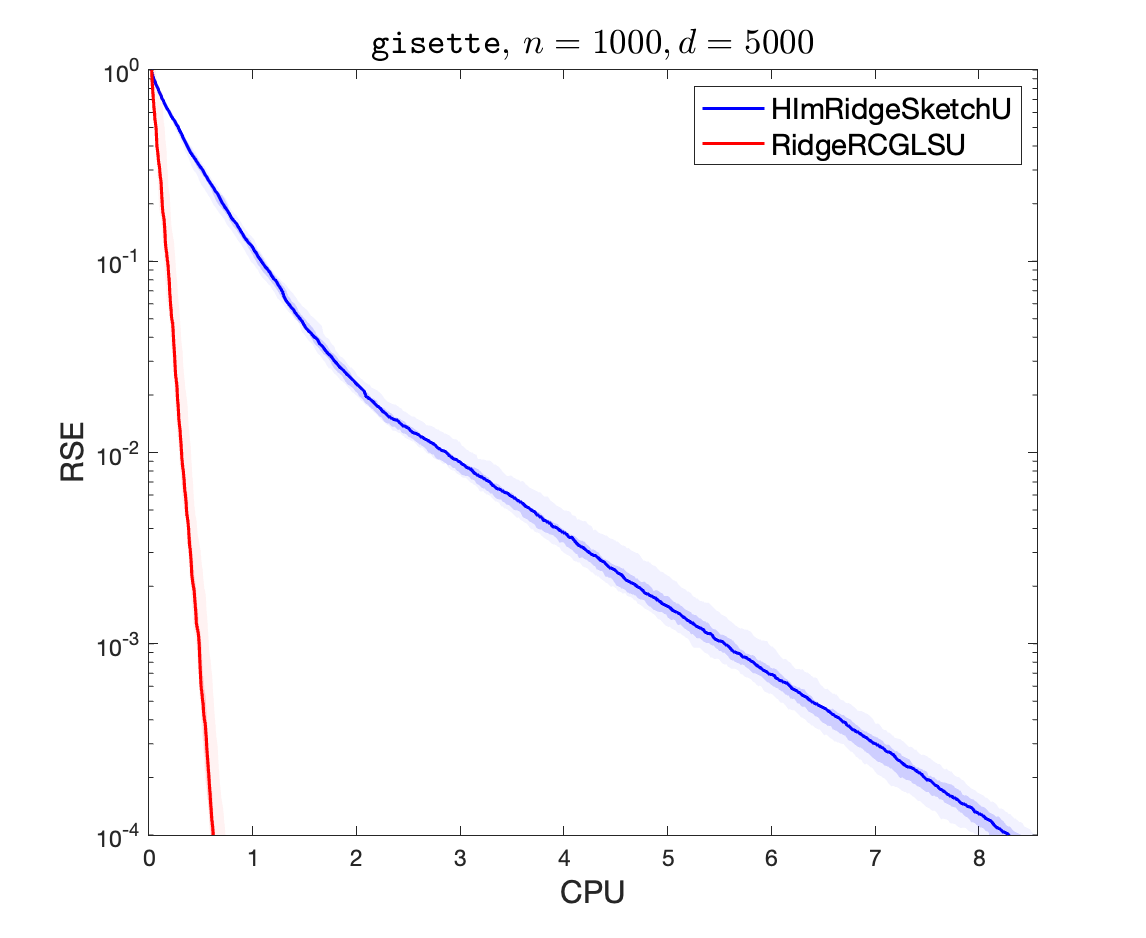}
			\end{tabular}
			\caption{Performance on underdetermined datasets from LIBSVM \cite{chang2011libsvm} with $\lambda=0.05$. Figures depict the evolution of RSE with respect to the total flops and the CPU time. We set $q=500$ for {\tt LEDGAR} and $q=50$ for {\tt gisette}, and terminate the algorithms once RSE $< 10^{-4}$. The title of each plot indicates the names and sizes of the data.}
			\label{figureR8}
		\end{figure}

	\section{Concluding remarks}
	\label{Section-6}
	
	We developed a novel RCGLS framework for least-squares problems, derived from a new reformulation of the classical CGLS method. The proposed RCGLS method achieves variance reduction by adopting randomized coordinate gradients to construct updated search directions. Theoretical analysis has verified that RCGLS exhibits a more favorable convergence factor compared with the conventional RCD method. Furthermore, we proved that RCGLS can be equivalently reformulated to substantially avoid full-dimensional operations. By exploiting the block-orthogonal structure inherent in ridge regression, we extended the RCGLS framework to RidgeRCGLS, a lightweight, parallelizable, and accelerated solver tailored for ridge regression tasks. Numerical experiments validated our theoretical results and demonstrated the superior computational efficiency of the proposed method.
	
	There are still many possible future avenues of research. It is well known that the CGLS method converges to the unique minimum Euclidean norm least-squares solution $A^\dagger b$ when starting from an initial point $x^0 \in \operatorname{Range}(A^\top)$ (e.g., $x^0 = 0$). According to Theorem \ref{RCGLS_rate}, RCGLS converges to $A^\dagger b$ when $A$ has full column rank, but only to a general least-squares solution when $A$ is rank-deficient. A potential direction for future research is to integrate the proposed method with the REGS scheme \cite{ma2015convergence,Du19} to ensure convergence to the minimum-norm solution $A^\dagger b$ in the rank-deficient case. Furthermore, the core idea of the RCGLS framework can be extended to solve general convex quadratic problems of the form $f(x) = \frac{1}{2}x^\top H x + b^\top x$, where $H$ is symmetric positive definite. Such an extension would lead to randomized CG methods for a broader class of convex quadratic optimization tasks, which would also be a valuable topic for future investigation.

	\bibliography{zeng2023}
	
	

 \end{document}